\documentclass[hidelinks,onefignum,onetabnum]{siamart220329}
\usepackage{amsfonts}
\usepackage{amssymb}
\usepackage{amsopn}
\usepackage{cleveref}
\usepackage{algorithm}
\usepackage{algpseudocode}
\usepackage{xcolor}
\usepackage{subfig}
\usepackage{graphicx} %
\algrenewcommand\algorithmicrequire{\textbf{Input:}}
\algrenewcommand\algorithmicensure{\textbf{Output:}}

\newcommand{\ind}{\,\mbox{d}}

\headers{Low-rank structure for inverse scattering}{Y~Zhou, L~Audibert, S~Meng, and  B~Zhang}

\title{Exploring low-rank structure for an inverse scattering problem with far-field data\thanks{Submitted to the editors May, 2024; updated May, 2025}}

\author{Yuyuan Zhou\thanks{Academy of Mathematics and Systems Science, Chinese Academy of Sciences, Beijing, 100190, 
China and School of Mathematical Sciences, University of Chinese Academy of Sciences, Beijing 100049, China
(\email{zhouyuyuan@amss.ac.cn})}
\and 
Lorenzo Audibert\thanks{PRISME, EDF R\&D, 6 Quai Watier, 78400 Chatou, France and UMA, Inria, ENSTA Paris, 
Institut Polytechnique de Paris, 91120 Palaiseau, France (\email{lorenzo.audibert@edf.fr})}
\and 
Shixu Meng\thanks{Department of Mathematics, Virginia Tech, 24061 Blacksburg,  USA. Corresponding author. (\email{sgl22@vt.edu})}
\and 
Bo Zhang\thanks{State Key Laboratory of Mathematical Sciences and Academy of Mathematics and Systems Science,
Chinese Academy of Sciences, Beijing 100190, China and School of Mathematical Sciences, University of Chinese
Academy of Sciences, Beijing 100049, China (\email{b.zhang@amt.ac.cn})}
}

\begin{document}

\maketitle

\begin{abstract}
In this work, we introduce a novel low-rank structure tailored for solving the inverse scattering problem.  The particular low-rank structure is given by the generalized prolate spheroidal wave functions, computed stably and accurately via a Sturm-Liouville problem.
{ We first process the far-field data to obtain a post-processed data set within a disk domain. Subsequently, the post-processed data are projected onto a low-rank space given by the low-rank structure. The unknown is approximately solved in this low-rank space, by dropping higher-order terms.    The low-rank structure leads to {  an explicit}  stability estimate for unknown functions belonging to standard Sobolev spaces, and a Lipschitz stability estimate for unknowns belonging to a finite dimensional low-rank space}.
Various numerical experiments are conducted to validate its performance, encompassing assessments of resolution capability, robustness against randomly added noise and modeling errors, and demonstration of increasing stability.
\end{abstract}

\maketitle


\section{Introduction}

Inverse scattering merits important applications in { nondestructive testing, seismic imaging, ocean acoustics and many other areas}. For a more comprehensive introduction, we refer to \cite{cakoni2016inverse,colton2012inverse}. This problem is very challenging due to its intrinsically ill-posed and nonlinear nature. In this work, we aim to mitigate this challenge by exploring the low-rank structure {  associated with the inverse scattering problem}. In a broader perspective, this is in the spirit of learning low dimensional feature spaces by deep learning \cite{Goodfellow16} and kernel machine learning \cite{CV1995, HSS08}. We point out recent machine learning approaches to inverse scattering { \cite{DesaiLahivaaraMonk, gao2022artificial,khooying19,2024Reconstruction, liu2022deterministic,mengzhang24, ning2023direct,  Zhou2coef}}. 

Low-rank methods are capable of mitigating the curse of dimensionality for high dimensional PDEs, see for instance the monograph  \cite{Hackbusch19} and the survey  \cite{Bachmayr23}. For the inverse scattering problem and for inverse problems in general,  { exploring low-rank structures} is a natural choice since the solution intrinsically has a low-rank structure due to its ill-posedness: putting it another way, one can think about solving a linear system with an ill-conditioned matrix and looking for a solution in a low dimensional space spanned by principal components. However unlike a simple linear system, it is  much more difficult to study the low-rank structure for inverse scattering problems. The linear sampling method \cite{ColtonKirsch96} and factorization method \cite{Kirsch98} are shape identification methods {  without assumptions such as Born approximation or physical optics approximation}. Recently it was shown in \cite{Kirsch17} that the eigenvalues of the data operator and these of its linearized version (i.e., Born data operator) decay in the same rate. The work \cite{moskowSchotland08} also shares a similar spirit on the connection between the ill-posedness of the fully nonlinear model and the Born model. Thus important insights may be drawn from the Born model.   The recent work \cite{khooying19} explored the low-rank structure of the (discretized) inverse scattering by a neural network model called SwitchNet.  In a continuous setting, the low-rank structure was explored  by \cite{meng23data} which focused on the theoretical basis for the low-rank structure and proved an explicit stability estimate  for contrast in standard Sobolev spaces. We mention that the idea of low-rank structures can also be extended to other inverse problems such as the multi-frequency inverse source problem \cite{GriesmaierSchmiedecke-source} and \cite{meng23parameter}.

The low-rank structure of \cite{khooying19} is based on the discrete Fourier transform and ideas of butterfly algorithm; see also a theoretical analysis of a butterfly-net in \cite{LiChengLu20}. The low-rank structure of \cite{meng23data} was based on the prolate {  spheroidal} wave functions and their generalizations. These special functions were studied in a series of work  \cite{Slepian61,Slepian64,Slepian78} and we refer to  the  prolate { spheroidal} wave functions in two dimensions as disk PSWFs for convenience. One of the remarkable properties of the disk PSWFs is the so-called dual property: the disk PSWFs  are the eigenfunctions of a restricted Fourier integral operator and of a Sturm-Liouville operator at the same time. It is this important dual property that leads to nontrivial and robust numerical algorithms, see \cite{osipov2013prolate} for a comprehensive introduction to the one dimensional PSWFs. For more recent studies on multidimensional generalizations of the PSWFs, we refer to \cite{GreengardSerkh18,Shkolnisky07,ZLWZ20} and the references therein. { The PSWFs, in contrast to other special functions, are the eigenfunctions of an unknown-to-data operator associated with the inverse scattering problem, making the PSWFs well-suited for this study.}

In this work we demonstrate { how to tackle the inverse scattering problem based on the low-rank structure given by the disk prolate spheroidal wave functions.}
We { demonstrate} the idea based on the linearized or Born model, 
and we show numerically by using the fully non-linear model that it is robust to linearization. There are several key ingredients in our algorithm:  evaluation of the PSWFs eigensystems, data processing, { low dimensional projection and regularization}.

To evaluate the disk PSWFs eigensystem in two dimensions, we apply the ideas of \cite{GreengardSerkh18,ZLWZ20}. In particular we compute the disk PSWFs by the Sturm-Liouville eigenvalue problem: this is robust and accurate in similar spirit to the well-known spectral methods \cite{Gottlieb77book}; it is completely different from directly computing the eigenfunction from a compact   operator {  with numerically repeated eigenvalues},  which is a much more challenging problem \cite{osborn75}. After the evaluation of the disk PSWFs, we then compute the prolate eigenvalues by the interplay between the Sturm-Liouville operator and the Fourier integral operator. More computational details may be found in \cite{P Greengard2024arxiv,GreengardSerkh18,Shkolnisky07,ZLWZ20}.
We next proceed to process the far-field data and { solve the unknown in the low-rank space via} a spectral cutoff regularization. This can be applied  to both the Born far-field data and the full far-field data. We process the far-field data to obtain a  data set in the unit disk, { leading to an explicit unknown-to-data map}.
According to appropriate mock-quadrature nodes (which are related to the disk PSWFs) in the unit disk, we continue to extract a smaller set of data (or filtered data) from the data set (which can be large scale). This idea is in similar spirit to  \cite{boydxu09}.   Using such post-processed data in the disk, we then obtain their projections onto the disk PSWFs whose corresponding prolate eigenvalues are larger than a spectral cutoff parameter. Such a low-rank structure is tailored for the inverse scattering problem.

To shed light on the potential of our algorithm, we have conducted various numerical experiments and have observed the following interesting features. First, the algorithm leads to improved resolution. We have numerically tested objects which are separated less than half wavelength, and these objects can be clearly distinguished from the reconstruction. We point out that similar improved resolution was also reported in \cite{novikov22a}. Second, the proposed algorithm is robust to both randomly added noises and modeling errors. For experiments with Born data, the performance of the algorithm remains good even when we add large random noises; for experiments with full data,
we have tested our algorithm when the Born modeling  error (computed according to \cite{Burgel19}) is not small, and the performance is still very good. This demonstrates that our algorithm has great potential to facilitate more advanced algorithms for the fully nonlinear model.  Third, the proposed algorithm has the potential to be more favorable than direct implementations of Matlab built-in fft and certain iterative methods. Our direct implementation of Matlab built-in fft requires { extending the data from a unit disk to a unit square by zero} and that the { wave number respects certain periodicity, which imposes more restrictions  on the problem  and possible blurring effect.}
The iterative methods
\cite{Burgel19} usually depend on the choice of initial guess and is more computationally expensive than our proposed algorithm. Fourth, {  since the dimension of the selected low-rank space in the spectral cutoff regularization increases with the frequency or the wave number,  the proposed numerical algorithm is capable of demonstrating the so-called increasing stability \cite{HrycakIsakov04,SubbarayappaIsakov07}.  }

The remaining of the paper is organized as follows. In Section \ref{section: model of inverse scattering} we introduce the mathematical model of the inverse scattering problem. To solve the inverse scattering problem, we need several analytical and computational results of disk PSWFs which are discussed in Section \ref{section: disk PSWFs}. {  In Section \ref{section: low-rank method} we propose the method based on the low-rank structure}. Finally we conduct various numerical experiments in Section \ref{section: numerical experiments} to illustrate the potential of our algorithm.

\section{Mathematical model of inverse scattering} \label{section: model of inverse scattering}
In this section, we introduce the inverse medium scattering problem in $\mathbb{R}^2$.  Let $k>0$ be the  wave number. A plane wave  takes the following form
\begin{equation*} %
e^{i k x \cdot \hat{\theta}}, \quad \hat{\theta} \in \mathbb{S} :=\{ x \in \mathbb{R}^2: |x|=1\},
\end{equation*}
where $\hat{\theta}$ is the direction of propagation.
Let $\Omega \subset \mathbb{R}^2$ be an open and bounded set with
Lipschitz boundary $\partial \Omega$ such that $\mathbb{R}^2 \backslash \overline{\Omega}$  is connected. To best illustrate the parameter identification theory, we let the real-valued function $q(x) \in L^\infty(\mathbb{R}^2)$ be the contrast of the medium, $q > -1$ on $\Omega$ and $q=0$ on $\mathbb{R}^2 \backslash \overline{\Omega}$  such that the support of the contrast is $\Omega$.   The scattering due to a plane wave $e^{i k x \cdot \hat{\theta} }$ is to find total wave field $u^t(x;\hat{\theta};k) = e^{i k x \cdot \hat{\theta} } + u^s(x;\hat{\theta};k)$ belonging to $H^1_{loc}(\mathbb{R}^2)$  such that
\begin{eqnarray}
\Delta_x u^t(x;\hat{\theta};k) + k^2 \left(1+q(x)\right) u^t(x;\hat{\theta};k) =  0   &\mbox{ in }&  \mathbb{R}^2,   \label{medium us eqn1}\\
\lim_{r:=|x|\to \infty} r^{\frac{1}{2}}  \big( \frac{\partial u^s(x;\hat{\theta};k)}{\partial r} -ik u^s(x;\hat{\theta};k)\big) =0,  \label{medium us eqn2}
\end{eqnarray}
{where the last equation, i.e., the Sommerfeld radiation condition, holds  uniformly for all directions} (and a solution is called   radiating   if it satisfies this radiation condition). The scattered wave field is $u^s(\cdot;\hat{\theta};k)$.   The above scattering problem \eqref{medium us eqn1}--\eqref{medium us eqn2} is a special case of the more general problem where one looks for a radiating solution $u^s \in H^1_{loc}(\mathbb{R}^2)$   to
\begin{equation} \label{medium us eqn1+2}
\Delta u^s + k^2(1+q) u^s = -k^2 q f,
\end{equation}
where $f \in L^\infty(\mathbb{R}^2)$. Setting $f(x)=e^{ikx\cdot \hat{\theta}}$ in \eqref{medium us eqn1+2} recovers \eqref{medium us eqn1}--\eqref{medium us eqn2}. This model is referred to as the full model.

It is known that there exists a unique radiating solution to \eqref{medium us eqn1+2}, cf. \cite{colton2012inverse} and \cite{Kirsch17}. For example, the solution can be obtained  with the help of the Lippmann-Schwinger integral equation,
\begin{eqnarray*}
u^s(x) - k^2 \int_\Omega \Phi(x,y)q(y)  u^s(y)  \ind y= k^2 \int_\Omega \Phi(x,y)q(y)   f(y) \ind y, \quad x \in \mathbb{R}^2.
\end{eqnarray*}
where $\Phi(x,y)$ is  the fundamental function for the Helmholtz equation given by
\begin{equation*} \label{Green function def}
\Phi(x,y)
:=
\frac{i}{4} H^{(1)}_0(k|x-y|)
\qquad x\not=y,
\end{equation*}
here $H^{(1)}_0$ denotes the Hankel function of the first kind of order zero (\cite{colton2012inverse}).
From the asymptotic of the scattered wave field (c.f. \cite{cakoni2016qualitative})
\begin{equation} \label{full model 1}
u^{s}(x;\hat{\theta};k)
=
\frac{e^{i\frac{\pi}{4}}}{\sqrt{8k\pi}} \frac{e^{ikr}}{\sqrt{r}}\left(u^{\infty}(\hat{x};\hat{\theta};k)+\mathcal{O}\left(\frac{1}{r}\right)\right)
 \quad\mbox{as }\,r=|x|\rightarrow\infty,
\end{equation}
uniformly with respect to all directions $\hat{x}:=x/|x|\in\mathbb{S}$, we arrive at $u^{\infty}(\hat{x};\hat{\theta};k)$ which is known as the  far-field data (pattern) with $\hat{x}\in\mathbb{S}$ denoting the observation direction. The multi-static data at a fixed frequency are given by
\begin{equation} \label{Section model far-field data}
\{u^{\infty}(\hat{x};\hat{\theta};k): \hat{x}\in\mathbb{S}, \hat{\theta}\in\mathbb{S}\}.
\end{equation}

The inverse scattering problem is to determine the contrast $q$ from these far-field data. It is known that this two dimensional inverse scattering problem has a unique solution, see for instance \cite{bukhgeim08}.
In practical applications, we are usually given the discrete far-field data at a set of uniformly distributed directions
\begin{equation} \label{Section model far-field data discrete}
\{u^{\infty}(\hat{x}_m;\hat{\theta}_\ell;k): ~m= 1,2,\dots, N_1, ~\ell = 1,2,\dots, N_2\}.
\end{equation}

The linear sampling method \cite{ColtonKirsch96} and factorization method \cite{Kirsch98} are shape identification methods {  without assumptions such as Born approximation or physical optics approximation}. Recently it was shown in \cite{Kirsch17} that the eigenvalues of the data operator and these of its linearized version (i.e., Born data operator) decay in the same rate. The work \cite{moskowSchotland08} also shares a similar spirit on the connection between the ill-posedness of the fully nonlinear model and the Born model. Thus important insights may be drawn from the Born model and for later purposes we introduce the Born model.
The Born approximation $u_b^s(x;\hat{\theta};k)$ is the unique radiating solution to the Born model
\begin{equation}\label{Born model}
    \Delta u^s_b+k^2u^{s}_b=-k^2q e^{ik x \cdot \hat{\theta}} \quad \mbox{in} \quad \mathbb{R}^2.
\end{equation}
From the asymptotic behavior
\begin{equation*}
    u_b^s(x;\hat{\theta};k)=\frac{e^{i\frac{\pi}{4}}}{\sqrt{8k\pi}}\frac{e^{ikr}}{\sqrt{r}}\left(u^{\infty}_b(\hat{x};\hat{\theta};k)+\mathcal{O}\left(\frac{1}{r}\right) \right) \quad \mbox{as}\quad r\rightarrow\infty,
\end{equation*}
we obtain  the Born far-field pattern $u^{\infty}_b(\hat{x};\hat{\theta};k)$, $\hat{x}=x/|x|\in \mathbb{S}$. One advantage of the Born far-field data is that we can directly obtain an explicit formula by
\begin{align}\label{Born fourier}
    u_b^\infty (\hat{x};\hat{\theta};k)&=k^2\int_\Omega e^{-ik\hat{x}\cdot y}q(y) e^{iky\cdot \hat{\theta}}\ind y.
\end{align}
This formula play a key role in our  method. If we identify $p= \frac{\hat{\theta} - \hat{x}}{2} \in B(0,1)$ where $B(0,1)$ denotes the unit disk, then the knowledge of the multi-static Born far-field data,  { i.e., $\{u_b^{\infty}(\hat{x};\hat{\theta};k) : \hat{x} \in \mathbb{S}, \hat{\theta} \in \mathbb{S}\}$,}
gives the knowledge of the restricted Fourier transform of the unknown $q$, i.e.,
$
\int_\Omega e^{ic py }q(y)  \ind y
$ for $p \in B(0,1)$ {  with $c=2k$}.
It is known that the restricted Fourier integral operator has a low-rank structure, see for instance \cite{meng23data,Slepian64} for the continuous case and \cite{khooying19} for the discrete case. In the following we show how to explore the low-rank structure to solve for our (linearized and nonlinear) inverse scattering problem.

\section{Generalization of prolate spheroidal wave functions} \label{section: disk PSWFs}
The low-rank structure of our paper is based on the generalizations of prolate spheroidal wave functions (PSWFs). The PSWFs and their generalizations were studied in a series of work by Slepian \cite{Slepian61,Slepian64,Slepian78}.  For a comprehensive introduction to the one dimensional PSWFs, we refer to \cite{osipov2013prolate}. For more recent studies on multidimensional generalizations of the PSWFs, we refer to \cite{GreengardSerkh18,Shkolnisky07,ZLWZ20}. For our inverse problem, we rely on the generalization of  PSWFs to two dimensions.
\subsection{Disk PSWFs}
According to \cite{Slepian64},  there exist real-valued  eigenfunctions $\{\psi_{m,n,l}(x;c)\}^{l\in\mathbb{I}(m)}_{m,n\in \mathbb{N}}$ of the restricted Fourier transform with bandwidth parameter c:
    \begin{align}\label{eigen_R_Fourier}
        \mathcal{F}_{B(0,1)}^c \psi_{m,n,l}(x;c)&=\int_{B(0,1)}e^{ic x\cdot y}\psi_{m,n,l}(y;c)dy \nonumber\\
                                     &=\alpha_{m,n}(c)\psi_{m,n,l}(x;c),\quad x\in B(0,1),
    \end{align}
    where $\mathbb{N}=\{0,1,2,3,\dots\}$ and
    \begin{eqnarray*}
        \mathbb{I}(m)=\left\{
            \begin{array}{cc}
                \{1\} & m=0 \\
                \{1,2\} & m \geq 1
            \end{array}\right..
    \end{eqnarray*}
For notation convenience, we refer  to $\psi_{m,n,l}(x;c)$ as the disk PSWFs and $\alpha_{m,n}(c)$ as the prolate eigenvalues in this paper.

One important property of the disk PSWFs is the so-called dual property. \cite{Slepian64} gave the Sturm-Liouville problem for the radial part and one can show by a direct calculation that the disk PSWFs are also eigenfunctions of a Sturm-Liouville operator, i.e.,
\begin{equation}\label{sturm-liouvill}
    \mathcal{D}_c [\psi_{m,n,l}](x)=\chi_{m,n} \psi_{m,n,l}(x),\quad x\in B(0,1),
\end{equation}
where
\begin{align*}
    \mathcal{D}_{c} := -(1-r^2)\partial_r^2-\frac{1}{r}\partial_r+3r\partial_r-\frac{1}{r^2}\Delta_0+c^2r^2
\end{align*}
and the Laplace–Beltrami operator $\Delta_0= \partial^2_\theta $ is the spherical part of Laplacian $\Delta$. More details can be found in \cite{meng23data} and \cite{ZLWZ20}. We refer to $\chi_{m,n}(c)$ as the Sturm-Liouville eigenvalue.

We briefly state the following properties of disk PSWFs and refer to \cite{Slepian64} and \cite{ZLWZ20} for a proof.

\begin{lemma} \label{lemma: SL}
   For any $c>0$, $\{\psi_{m,n,l}(x;c)\}^{l\in\mathbb{I}(m)}_{m,n\in \mathbb{N}}$ forms a complete and orthonormal system of $L^2(B(0,1))$, i.e., for $\forall~m,~n,~m',~n'\in\mathbb{N},~l\in\mathbb{I}(m),~l'\in\mathbb{I}(m')$ it holds that
   \begin{equation*}
       \int_{B(0,1)}\psi_{m,n,l}(y;c)\psi_{m',n',l'}(y;c)\ind y=\delta_{m m'}\delta_{n n'}\delta_{l l'},
   \end{equation*}
where $\delta$ denotes the Kronecker delta. The corresponding  Sturm-Liouville eigenvalues $\{\chi_{m,n}\}_{m,n\in \mathbb{N}}$ in \eqref{sturm-liouvill} are real positive which  are ordered for fixed $m$ as follows
       \begin{equation*}
       0<\chi_{m,0}(c)<\chi_{m,1}(c)<\chi_{m,2}(c)<\cdots .
   \end{equation*}
\end{lemma}

\begin{lemma} \label{lemma: FIO}
   For any $c>0$, every prolate eigenvalue $\alpha_{m,n}(c)$ is  non-zero, and $\lambda_{m,n} = |\alpha_{m,n}(c)|$ can be arranged for fixed $m$ as
    \begin{equation*}
       \lambda_{m,n_1}(c)>\lambda_{m,n_2}(c)>0, \quad \forall n_1<n_2.
   \end{equation*}
   Moreover $\lambda_{m,n}(c)\longrightarrow 0$  as  $m,n\longrightarrow +\infty$.

\end{lemma}

Lemma \ref{lemma: FIO} states that the restricted Fourier integral operator has a low-rank structure. More importantly, Lemma \ref{lemma: SL} states that the computation of the disk PSWFs can be done using the Sturm-Liouville operator so that it is   stable and efficient.
The disk PSWFs serve as an important tool to solve the two-dimensional inverse scattering problem. We next illustrate their computations and relevant quadrature in order to propose { the reconstruction algorithm}. For a theoretical analysis of disk PSWFs with application to  inverse scattering, we refer to \cite{meng23data}.

\subsection{Evaluation of  disk PSWFs and prolate eigenvalues}
Using polar coordinates, each disk PSWF  $\psi_{m,n,l} (x;c)$ can be obtained by separation of variables by (cf. \cite{Slepian64} or \cite{meng23data})
\begin{equation*}
    \psi_{m,n,l}(x;c)={ r^m}\varphi_{m,n}(2{  r}^2-1;c)Y_{m,l}(\hat{x}),\quad x\in B(0,1),
\end{equation*}
where $x=r\hat{x}=(r\cos{\theta},r\sin{\theta})^T$ and the spherical harmonics $Y_{m,l}(\hat{x})$ are given by
\begin{align}\label{spherical_harmonic}
    Y_{m,l}(\hat{x})=\left\{
            \begin{array}{cc}
                \frac{1}{\sqrt{2\pi}}, & m=0,l=1 \\
                \frac{1}{\sqrt{\pi}}\cos{m\theta}, & m\geq 1,l=1 \\
                \frac{1}{\sqrt{\pi}}\sin{m\theta},& m\geq 1,l=2
            \end{array}\right. .
\end{align}
An efficient method to evaluate the disk PSWFs is to expand $\varphi_{m,n}(\eta;c)$ by normalized Jacobi polynomials $\{P^{(m)}_{j}(\eta)\}^{j\in\mathbb{N}}_{\eta\in (-1,1)}$,
\begin{equation}\label{expansion}
    \varphi_{m,n}(\eta;c)=\sum_{j=0}^{\infty}\beta_j^{m,n}(c) P^{(m)}_j(\eta).
\end{equation}
In the following section, with reference to \cite{ZLWZ20} we will  discuss the  evaluation of the  coefficients $\{\beta_j^{m,n}(c)\}$ and use them to determine the prolate eigenvalues.

\subsubsection{Jacobi polynomials}
We introduce the normalized Jacobi polynomials $\{P^{(\alpha,\beta)}_n(\eta)\}_{\eta\in (-1,1) }^{n\in\mathbb{N}}$ which are eigenfunctions of the following Sturm-Liouville problem,
\begin{equation}\label{jacobi polynomial}
    -\frac{1}{w_{\alpha,\beta}(\eta)}\partial_{\eta} \left(w_{\alpha+1,\beta+1}(\eta) \partial_{\eta} P^{(\alpha,\beta)}_n(\eta) \right)=n(n+\alpha+\beta+1)P^{(\alpha,\beta)}_n(\eta)
\end{equation}
where $w_{\alpha,\beta}(\eta)=(1-\eta)^\alpha (1+\eta)^{\beta}$, and they satisfy the orthogonality condition
\begin{equation}\label{orthogonal_JacobiPolynomial}
    \int_{-1}^1 P_n^{(\alpha,\beta)}(\eta)P_{n'}^{(\alpha,\beta)}(\eta)d\eta=2^{\alpha+\beta+2}\delta_{nn'}, \quad \forall n,n' \in \mathbb{N}.
\end{equation}
In this paper we work with $\alpha=0$ and $\beta = m$. Hence we write, without the danger of confusion, $P_n^{(m)} = P_n^{(0,m)}$.

The normalized Jacobi polynomials $\{P_n^{(m)}(x)\}_{x\in (-1,1) }$ can be obtained through the three-term recurrence relation
\begin{align*}
         P^{(m)}_{n+1}(x)&=\frac{1}{a_n} [ (x-b_n) P_n^{(m)}(x)-a_{n -1}P^{(m)}_{n-1}(x) ],\quad n\geq 1\\
    P^{(m)}_{0}(x)&=\frac{1}{h_0 },\quad P^{(m)}_{1}(x)=\frac{1}{2h_1  }[(m+2)x-m]
\end{align*}
where
\begin{align}\label{jacobi polynomial recurrence}
   \left\{
            \begin{array}{cc}
                a_{n}=&\frac{2(n+1)(n+m+1)}{(2 n+m+2)\sqrt{(2 n+{ m+1})(2 n+m+3)}} \\
                b_{n}=&\frac{m^{2}}{(2 n+m)(2 n+m+2)} \\
               h_{n}=& \frac{1}{\sqrt{2(2 n+m+1)}}
            \end{array}\right. , \qquad n\in\mathbb{N}.
\end{align}
For a more comprehensive introduction to special polynomials, we refer to \cite{Abramowitz64}.

\subsubsection{Appoximation of disk PSWFs system}

To compute the coefficients of \eqref{expansion},  we plug this expansion into \eqref{sturm-liouvill} and we can derive (by noting \eqref{jacobi polynomial} and the fact that
$\Delta_0 Y_{m,l}(\hat{x})=-m^2 Y_{m,l}(\hat{x})$, or by \cite{ZLWZ20}) that the  coefficients $\{\beta_j^{m,n}\}_{j\in\mathbb{N}}$ satisfy
    \begin{eqnarray} \label{tridiagonal_linear_system}
       &&\Big(\gamma_{m+2j}+\frac{(1+b_j)c^2}{2}-\chi_{m,n}(c) \Big) \beta_j^{m,n} \nonumber \\&& \hspace{0.3\linewidth} +\frac{a_{j-1}c^2}{2}\beta_{j-1}^{m,n}+\frac{a_{j}c^2}{2}\beta_{j+1}^{m,n}=0, \quad j\geq 0
    \end{eqnarray}
where $\gamma_{m+2j}=(m+2j)(m+2j+2)$. Here $\chi_{m,n}(c)$  is  the Sturm-Liouville eigenvalue.

To evaluate eigensystem $\{\psi_{m,n,l}(x;c),\chi_{m,n}\}$ for $2n+m\leq N$, we follow \cite{ZLWZ20} to set ${ \tilde{M}}=2N+30$ and obtain the approximations $\{\tilde{\psi}_{m,n,l}(x;c),\tilde{\chi}_{m,n}\}$ by
\begin{equation} \label{section: computing disk PSWFs: polynomial expansion}
    \tilde{\psi}_{m,n,l}(x ; c)=\sum_{j=0}^{K} \tilde{\beta}_{j}^{m,n} { \|x\|^m} P_j^{(m)}(2||x||^2-1)Y_{m,l}(\hat{x}), \quad 2 n+m \leq N
\end{equation}
where $K=\left\lceil\frac{{ \tilde{M}}-m}{2}\right\rceil$. As such, $\{\tilde{\beta}_{j}^{m,n}\}$ and $\tilde{\chi}_{m,n}$ can be solved from the following tridiagonal linear system
\begin{equation}\label{eigensystem}
{A}{\tilde{\beta}^{m,n}}=\tilde{\chi}_{m,n}{\tilde{\beta}^{m,n}}
\end{equation}
where ${\tilde{\beta}^{m,n}}=(\tilde{\beta}_{0}^{m,n},\tilde{\beta}_{1}^{m,n},...,\tilde{\beta}_{K}^{m,n})^T$, ${A}$ is a $(K+1)\times(K+1)$ symmetric tridiagonal matrix whose nonzero entries are given by
\begin{eqnarray*}
    A_{j,j} =\gamma_{m+2j}+\frac{(1+b_j)c^2}{2},\qquad 
    A_{j,j+1} =A_{j+1,j}=\frac{a_{j}c^2}{2}, \quad j\geq 0.
\end{eqnarray*}
Since $\| \psi_{m,n,l}(x;c)\|_{L^2(B(0,1))}=1$, the coefficient vector ${\tilde{\beta}^{m,n}} $ are normalized to satisfy
\begin{equation*}
    \| {\tilde{\beta}^{m,n}}\|_2=\sqrt{\sum_{j=0}^{K} |\tilde{\beta}_{j}^{m,n}|^2}=1
\end{equation*}
owing to $\| P_{m,j,l} \|_{L^2(B(0,1))}=1$ where $P_{m,j,l}(x)={ \|x\|^m}P_j^{(m)}(2\|x\|^2-1)Y_{m,l}(\hat{x})$. We further sort the eigenvalues by 
    $\tilde{\chi}_{m,j}<\tilde{\chi}_{m,i}$ for any $j<i$.

Finally to evaluate the prolate eigenvalues $\{\alpha_{m,n}\}_{m,n\in\mathbb{N}}$, the following formula was given in \cite{ZLWZ20}
\begin{equation}\label{eigenvalues}
   \tilde{\alpha}_{m,n}(c)={ i^{m}}\tilde{\lambda}_{m,n}(c), \quad  \tilde{\lambda}_{m,n}(c)=\frac{\pi c^m }{2^{m-\frac{1}{2}}m!\sqrt{m+1}}\frac{\tilde{\beta}_0^{m,n}}{\tilde{\varphi}_{m,n}(-1;c)}
\end{equation}
where $\tilde{\varphi}_{m,n}(-1;c)=\sum_{j=0}^{K} \tilde{\beta}^{m,n}_j P_j^{(m)}(-1)$.
This formula is sufficient for our inverse problem.

We summarize the following algorithm for the evaluation of PSWFs and prolate eigenvalues.
\begin{algorithm}
	\caption{Evaluation of {  disk} PSWF and prolate eigenvalue}
\begin{algorithmic}[1]\label{Algorithm: PSWFs}
\Require  Parameter $c$, index of disk PSWF $(m,n,l)$ and $x\in\mathbb{R}^2$.
\Ensure  $\psi_{m,n,l}(x;c)$ and prolate eigenvalue $\alpha_{m,n}(c)$.

\State Compute the tridiagonal linear system \Cref{eigensystem} (and sort the eigenvalues from small to large) to obtain the eigenvector $\tilde{\beta}^{m,n}$ and the eigenvalue $\tilde{\chi}_{m,n}(c)$ corresponding to the given $(m,n)$. Here ${\tilde{\beta}^{m,n}}=(\tilde{\beta}_{0}^{m,n},\tilde{\beta}_{1}^{m,n},...,\tilde{\beta}_{K}^{m,n})^T$.
\State  Evaluate $(P_{m,0}(2\|x\|^2-1),P_{m,1}(2\|x\|^2-1),\cdots,P_{m,K}(2\|x\|^2-1) )^T$ at $x\in\mathbb{R}^2$ by the recurrence \eqref{jacobi polynomial recurrence} and the spherical harmonics $Y_{m,l}(\hat{x})$ at $\hat{x}$  by \Cref{spherical_harmonic}.

\State Compute the approximation of disk PSWFs by $\psi_{m,n,l}(x;c) \approx{  \|x  \|^m}Y_{m,l}(\hat{x})\sum_{j=0}^K   \tilde{\beta}_{j}^{m,n} P_j^{(m)}(2\|x\|^2-1)$.
\State Evaluate the prolate eigenvalue $ \alpha_{m,n}(c) \approx{ i^{m}}\tilde{\lambda}_{m,n}(c)$ according to \Cref{eigenvalues}.
\end{algorithmic}
\end{algorithm}

\subsection{Quadrature} \label{section: quadrature}
Later on, we need to project our data onto the disk PSWFs. This motivates us to consider an appropriate quadrature for the projection (i.e. integral). In the remaining of this section, we simply identify $w(x)$ as $w(r,\theta)$ where $(r,\theta)$ is the polar representation of $x$. Here we employ a  Gauss-Legendre quadrature rule for the radial part and a trapezoidal quadrature rule for the angular part to obtain, for any function $w \in L^2(B(0,1))$
\begin{eqnarray*}
    w_{m,n,l} &=&\int_{B(0,1)} w(x)\psi_{m,n,l}(x) \ind x= \int_0^{2\pi}\int_0^1 w(r,\theta) \psi_{m,n,l}(r,\theta) r \ind r \ind\theta\\
    && (\mbox{change of variable } t=2r^2-1)\\
             &=&\frac{1}{4}\int_0^{2\pi}\int_{-1}^1 w(\sqrt{(1+t)/2},\theta) \psi_{m,n,l}(\sqrt{(1+t)/2},\theta) \ind t \ind \theta \\
             &\approx&\frac{1}{4}  \sum_{i=0}^{M-1} \sum_{j=0}^{T-1} w(\sqrt{(1+t_j)/2},\theta_i)\psi_{m,n,l}(\sqrt{(1+t_j)/2},\theta_i)\omega_{t_j}\omega_{\theta_i},
\end{eqnarray*}
where
\begin{eqnarray} \label{GL+trapezoidal quadrature}
\left\{t_j,\omega_{t_j}\right\}_{j=0}^{T-1}: \mbox{Gauss-Legendre quadrature nodes and weights},    \nonumber \\
\left\{\theta_i=\frac{2i\pi}{M},\omega_{\theta_i}=\frac{2\pi}{M}\right\}_{i=0}^{M-1}: \mbox{trapezoidal quadrature nodes and weights}.
\end{eqnarray}
{ Suppose that the data are projected onto a low-rank space $\{\psi_{m,n,\ell}\}_{(m,n,\ell) \in J_\epsilon}$ with $J_{\epsilon}:=\{ (m,n,\ell): |\alpha_{m,n}| > \epsilon \}$.
With the above quadrature and polynomial approximation error estimate  \cite{P Greengard2024arxiv}, one way to choose the number of quadrature nodes $TM$ is given by \begin{equation}\label{eq: T choice}
    T=\max{\{T(m,m',n,n'):\exists l\in \{1,2\},~s.t.,~(m,n,l),(m',n',l')\in J_\epsilon}\},
\end{equation}
and
\begin{equation}\label{eq: M choice}
    M=2m_{\rm max}+1, \quad m_{\rm max} = \max\{m: \exists l\in \{1,2\},~s.t.,~(m,n,l) \in J_\epsilon \},
\end{equation}
for the given $J_\epsilon$; here
$T(m,m',n,n')=\lceil \frac{1}{2}(\frac{m+m'+2}{2}+K(m,n)+K(m',n'))\rceil$,
and
$K(m,n)= \lceil \frac{1}{2}(\log_{1/2}\left(\frac{0.05a \epsilon}{\pi} \right)-m+\frac{1}{2})\rceil$.
We point out that it is possible to make more use of the disk PSWFs to investigate other quadratures depending on the parameter $c$, see for instance  \cite{P Greengard2024arxiv,osipov2013prolate}.}

\section{{ Proposed method based on the low-rank structure}} \label{section: low-rank method}
Now we are ready to elaborate { our proposed  method} based on the low-rank structure for the inverse scattering problem.  The two main ingredients are { data processing} and low-rank projection.
\subsection{{ Data processing}}

For the inverse scattering problem, we work with the   far-field data  which are given by \eqref{Section model far-field data discrete}, i.e.,
\begin{equation*}
\{u^{\infty}(\hat{x}_m;\hat{\theta}_\ell;k): ~m = 1,2,\dots, N_1, ~\ell = 1,2,\dots, N_2\},
\end{equation*}
where both $\{\hat{x}_m\}_{m=1}^{N_1}$ and $\{\hat{\theta}_\ell\}_{\ell=1}^{N_2}$ are uniformly distributed directions.
{ We first apply the transformation}
$$
(\hat{x}_m, \hat{\theta}_\ell) \mapsto \frac{\hat{\theta}_\ell -\hat{x}_m }{2}
$$
to transform the datum $u^{\infty}(\hat{x}_m;\hat{\theta}_\ell;k)$ on { $\mathbb{S}\times \mathbb{S}$ }to a datum evaluated at $\frac{\hat{\theta}_\ell -\hat{x}_m}{2} \in B=B(0,1)$, particularly we identify
$$
u(p_{m\ell};c) = \frac{1}{k^2}u^{\infty}(\hat{x}_m;\hat{\theta}_\ell;k), \quad p_{m\ell} = \frac{\hat{\theta}_\ell -\hat{x}_m}{2}, \quad c= 2k.
$$

{  The post-processed data set  $\{u(p):p\in B(0,1),p\neq 0\}$ is equivalent to the Born far-field data set $\{u_b^{\infty}(\hat{x};\hat{\theta};k) : \hat{x}, \hat{\theta} \in \mathbb{S},\hat{\theta}\neq \hat{x}\}$. On the one hand, the post-processed data can be derived from the Born far-field data through the transformation above. One the other hand, for every $0\neq p \in B(0,1)$, there are two solutions $(\hat{\theta}_j,\hat{x}_j)\in \mathbb{S}\times \mathbb{S}$ such that $\frac{\hat{\theta}_j-\hat{x}_j}{2}=p$  for $j=1,2$ with $\hat{\theta}_2=-\hat{x}_1,~\hat{x}_2=-\hat{\theta}_1$; due to  the reciprocity relation $u_b^\infty(\hat{\theta};\hat{x};k)=u_b^\infty(-\hat{x};-\hat{\theta};k)$, $u(p)$ can uniquely determine the Born far-field data. Since the reciprocity relation remains valid (cf. \cite{cakoni2016qualitative}) for the full model, the same equivalence between the post-processed data set and the far-field data set holds for in the full model case. } 

The new data set $\{u(p_{m\ell};c): m = 1,2,\dots, N_1, ~\ell = 1,2,\dots, N_2\}$ is of dimension $N_1N_2$.
We continue to extract a set of data from the data set $\{u(p_{m\ell};c): m = 1,2,\dots, N_1, ~\ell = 1,2,\dots, N_2\}$ (which can be large scale). Particularly, the subset $\{u(\tilde{p}_n;c): n=1,2,\cdots, TM \}$ will be extracted,  where $\{\tilde{p}_n\}_{n=1}^{T\times M}$ are mock-quadrature nodes that are approximations to  the exact quadrature nodes $\{p_n\}_{n=1}^{T\times M}$ in $B$. Here we identify $\{p_n\}_{n=1}^{TM}$ with $\big\{\sqrt{\frac{t_j+1}{2}} (\cos\theta_i,\sin \theta_i)^T \big\}_{j=0,~i=0}^{T-1,~M-1}$ where the exact quadrature nodes are given by \eqref{GL+trapezoidal quadrature}.
This idea is in similar spirit to the mock-Chebychev nodes \cite{boydxu09}.

The extraction of the approximate or mock-quadrature nodes is as follows.
Given exact quadrature nodes $\{p_n\}_{n=1}^{T\times M}$ and far-field data at the uniformly distributed directions $\{\hat{x}_m, ~\hat{\theta}_\ell: ~m = 1,2,\dots, N_1, ~\ell = 1,2,\dots, N_2\},
$
the approximate or mock-quadrature nodes $\{\tilde{p}_n\}_{n=1}^{T\times M}$ are given by
\begin{align*}
    \widetilde{p}_n =\frac{\hat{\theta}_{\ell^*}-\hat{x}_{j^*}}{2} \quad \mbox{where} \quad
    (\ell^*,j^*) =\mbox{argmin}_{\ell,j} \big\|p_n-\frac{\hat{\theta}_{\ell}-\hat{x}_{j}}{2} \big \|_2.
\end{align*}
We summarize the data extraction algorithm in \Cref{Algorithm: lrData} and illustrate the extraction in
\Cref{figure: dimensionality reduction data}, which shows the exact quadrature nodes in (a) with { $T=24$, $M=47$}, the $N_1 \times N_2 = 100\times 100$ grid nodes before extraction in (b), and the mock-quadrature nodes in (c).

Given $\{\widetilde{p}_n\}_{n=1}^{TM}$, we obtain its equivalent set in polar coordinates given by \\
$\Big\{\sqrt{\frac{\widetilde{t}_j+1}{2}} (\cos\widetilde{\theta}_i,\sin \widetilde{\theta}_i)^T \Big\}_{j=0,~i=0}^{T-1,~M-1}$ and conveniently set
$$\{\tilde{u}(\tilde{t}_j,\tilde{\theta}_i;c)\}_{j=0,~i=0}^{T-1,~M-1} =\{u(\widetilde{p}_n;c)\}_{n=1}^{TM}.
$$
We define the post-processed data (and we also refer to its continuous counterpart as the post-processed data without danger of confusion) by
\begin{equation}\label{DataMatrix}
    U =
    \begin{pmatrix}
        \tilde{u}(\tilde{t}_0,\tilde{\theta}_0;c) & \tilde{u}(\tilde{t}_0,\tilde{\theta}_1;c) & \cdots & \tilde{u}(\tilde{t}_0,\tilde{\theta}_{M-1};c) \\
        \tilde{u}(\tilde{t}_1,\theta_0;c) & \tilde{u}(\tilde{t}_1,\theta_1;c) & \cdots & \tilde{u}(\tilde{t}_1,\tilde{\theta}_{M-1};c) \\
                   \vdots & \vdots & \ddots & \vdots\\
        \tilde{u}(\tilde{t}_{T-1},\tilde{\theta}_0;c) & \tilde{u}(\tilde{t}_{T-1},\tilde{\theta}_1;c) & \cdots & \tilde{u}(\tilde{t}_{T-1},\tilde{\theta}_{M-1};c)
    \end{pmatrix}.
\end{equation}

\begin{figure}[htbp]
\centering
\subfloat[]{ \includegraphics[width=0.3\linewidth]{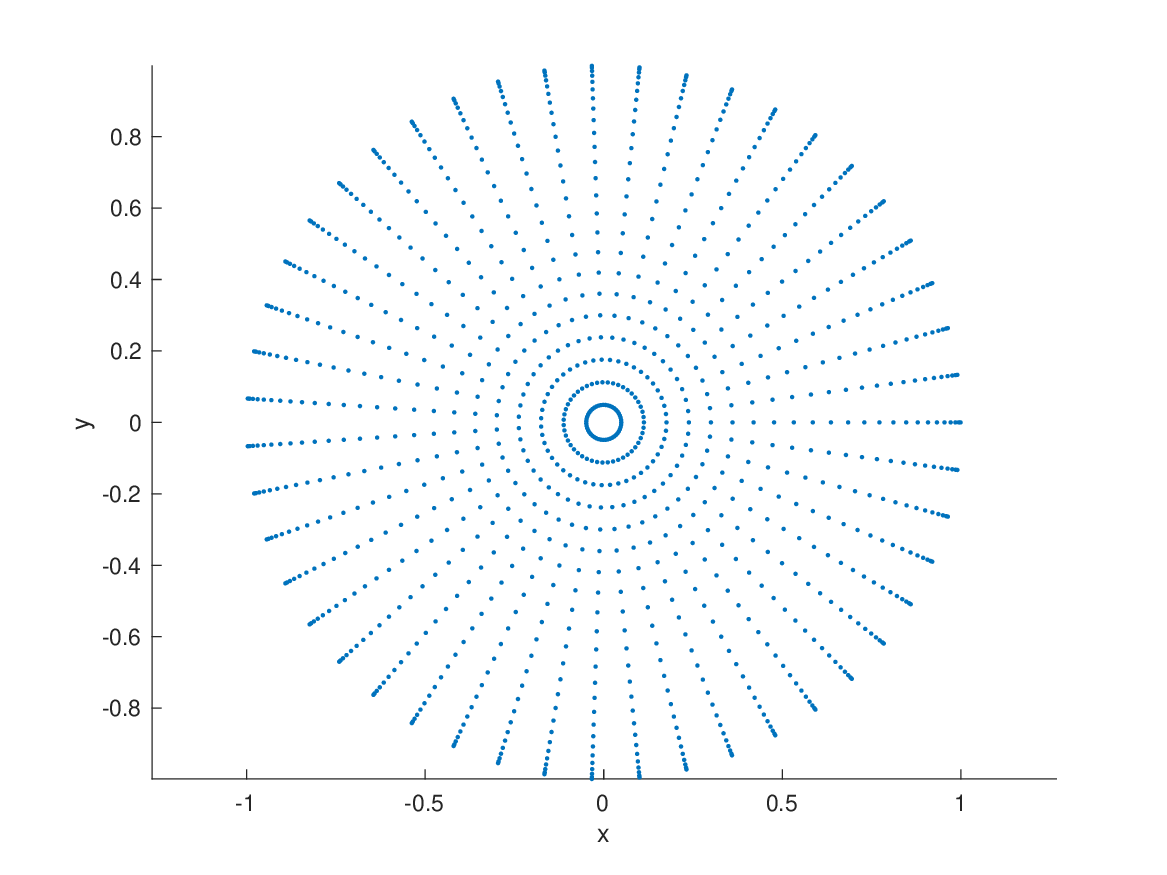} }
\subfloat[]{ \includegraphics[width=0.3\linewidth]{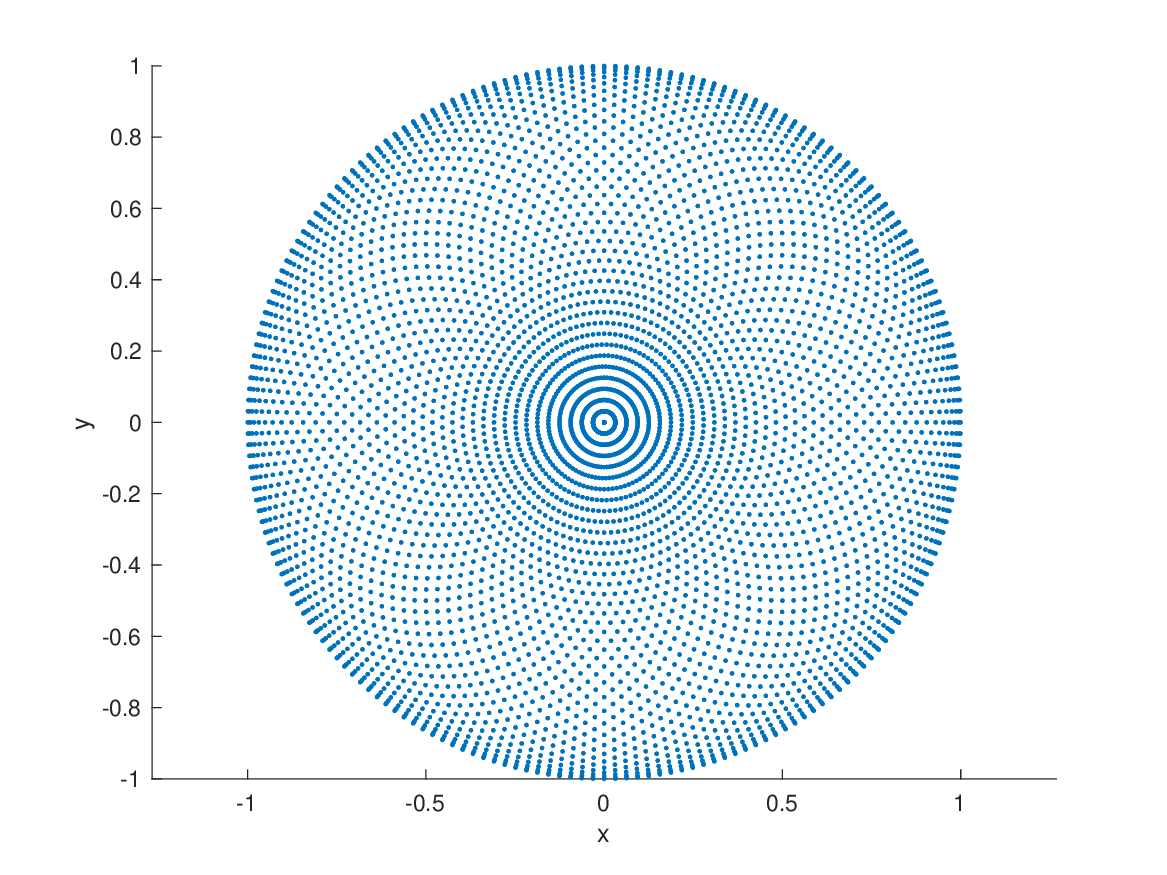} }
\subfloat[]{ \includegraphics[width=0.3\linewidth]{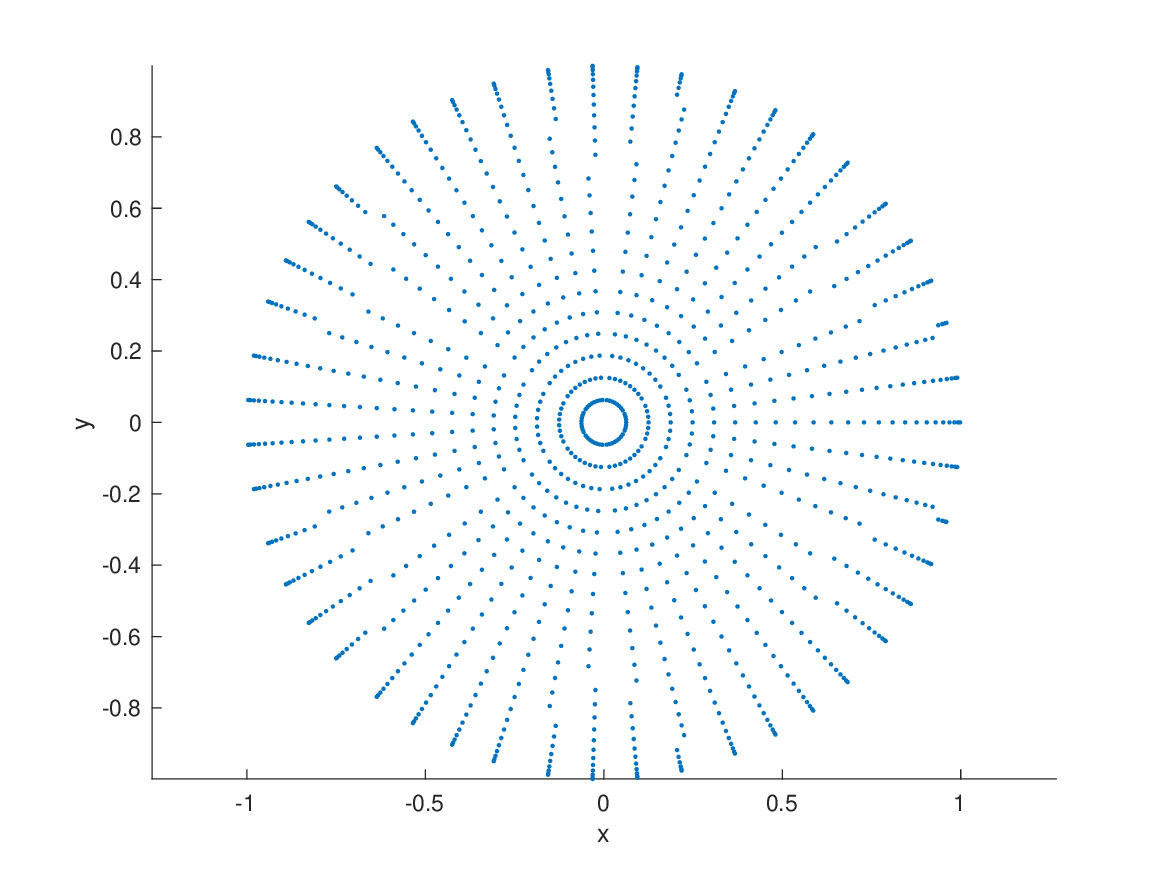} }
\caption{(a) exact quadrature nodes. (b) $N_1 \times N_2 = 100\times 100$ grid nodes before dimentionality reduction. (c) approximate or mock-quadrature nodes. } \label{figure: dimensionality reduction data}
\end{figure}

\begin{algorithm}
	\caption{{ Data Processing}}
\begin{algorithmic}[1]\label{Algorithm: lrData}
\Require  Far-field data $\{u^{\infty}(\hat{x}_m;\hat{\theta}_\ell;k): ~m = 1,2,\dots, N_1, ~\ell = 1,2,\dots, N_2\}$
\Ensure  {  Post-processed data} $U$
\State Compute  Gauss-Legendre quadrature nodes and weights $\{t_j,\omega_{t_j}\}_{j=0}^{T-1}$, and trapezoidal quadrature { nodes} and weights $\{\theta_i=\frac{2i\pi}{M},\omega_{\theta_i}=\frac{2\pi}{M}\}_{i=0}^{M-1}$. Identify $\{p_n\}_{n=1}^{TM}$ with $\big\{\sqrt{\frac{t_j+1}{2}} (\cos\theta_i,\sin \theta_i)^T \big\}_{j=0,~i=0}^{T-1,~M-1}$.
\State   Find $(\ell^*,m^*)$ such that $(\ell^*,m^*) =\mbox{argmin}_{\ell,m} \big\|p_n-\frac{\hat{\theta}_{\ell}-\hat{x}_{m}}{2} \big \|_2$

\State Compute approximate or mock-quadrature nodes $\widetilde{p}_n =\frac{\hat{\theta}_{\ell^*}-\hat{x}_{m^*}}{2}$ for $n=1,2,\dots,TM$. Identify $\{\widetilde{p}_n\}_{n=1}^{TM}$ with $\{\sqrt{\frac{\widetilde{t}_j+1}{2}} (\cos\widetilde{\theta}_i,\sin \widetilde{\theta}_i)^T \}_{j=0,~i=0}^{T-1,~M-1}$.

\State Evaluate the post-processed data $u(\widetilde{p}_n) = \frac{1}{k^2}u^\infty(\hat{x}_{m^*},\hat{\theta}_{\ell^*};k)$.
\State Formulate the matrix $U_{ji} = u(\widetilde{p}_n)$ by identifying $\{n: 1\le n \le {TM}\}$ with $\{ (j,i): j=0, 1 \dots T-1, ~i = 0, 1, \dots, M-1 \}$.
\end{algorithmic}
\end{algorithm}

\subsection{{  Low-rank reconstruction}}
\begin{algorithm}
	\caption{Low Dimensional Projection}\label{Algorithm: lrProj}
\begin{algorithmic}[1]
\Require  Matrix $U = (U_{ji})_{j=0, 1 \dots T-1, ~i = 0, 1, \dots, M-1}$ defined by \eqref{DataMatrix}, spectral cutoff regularization parameter $\epsilon$
\Ensure  The low dimensional vector $\{q_{m,n,\ell}\}_{(m,n,\ell) \in J_\epsilon}$ with $J_{\epsilon}:=\{ (m,n,\ell): { |\alpha_{m,n} |}> \epsilon \}$

\State
For all $(m,n,\ell) \in J_\epsilon$, apply the quadrature to compute
$$
    u_{m,n,l}  \approx\frac{1}{4}  \sum_{i=0}^{M-1} \sum_{j=0}^{T-1} U_{ji}\psi_{m,n,l}(\sqrt{(1+t_j)/2},\theta_i)\omega_{t_j}\omega_{\theta_i},
$$
where  $\{t_j,\omega_{t_j}\}_{j=0}^{T-1}$ and $\{\theta_i=\frac{2i\pi}{M},\omega_{\theta_i}=\frac{2\pi}{M}\}_{i=0}^{M-1}$ are quadrature nodes and weights given by \eqref{GL+trapezoidal quadrature}.

\State Compute $q_{m,n,l}=\frac{u_{m,n,l}}{\alpha_{m,n}}$ for all $(m,n,\ell) \in J_\epsilon$.
\end{algorithmic}
\end{algorithm}

In this section, we suppose that we are given the matrix $U = (U_{ji})$ with $j=0, 1 \dots T-1$ and $i = 0, 1, \dots, M-1$ defined by \eqref{DataMatrix}. Our goal is to project these data onto a { low-rank} space spanned by the disk PSWFs.

To begin with, consider the post-processed problem to determine the unknown $q\in L^2(B(0,1))$ from the post-processed data $u(x;c)$ by
\begin{equation}\label{inverse problem}
    u(x;c)=\int_{B(0,1)} e^{icx\cdot y}  q(y) \ind y, \quad x\in B(0,1),
\end{equation}
where $c>0$ is the parameter given by $c=2k$.

Given the disk PSWFs $\{\psi_{m,n,l}\}_{m,n\in\mathbb{N}}^{l\in\mathbb{I}(m)}$ which form a complete  and orthonormal basis of $L^2(B)$, we expand the unknown function $q$ and post-processed data $u$ in a series,
\begin{align*}
    q(x)=\sum_{m,n\in\mathbb{N},l\in\mathbb{I}(m)} q_{m,n,l} \psi_{m,n,l} (x;c),\qquad
    u(x;c)=\sum_{m,n\in\mathbb{N},l\in\mathbb{I}(m)} u_{m,n,l} \psi_{m,n,l} (x;c)
\end{align*}
where $u_{m,n,l}$ and $q_{m,n,l} $  are the projections of $u(x;c)$ and $q(x)$ on each  $\psi_{m,n,l} (x;c)$ respectively, i.e.,
$$
u_{m,n,l} = \langle u (\cdot;c),  \psi_{m,n,l} (\cdot;c)\rangle, \quad q_{m,n,l} = \langle q,  \psi_{m,n,l} (\cdot;c)\rangle,
$$
here $\langle \cdot, \cdot \rangle$ denotes the $L^2(B)$ inner product. From the property \eqref{eigen_R_Fourier}, one directly { finds} that
\begin{equation*}
    q_{m,n,l}=\frac{u_{m,n,l}}{\alpha_{m,n}}.
\end{equation*}

\begin{figure}[htbp]
\centering
\subfloat[]{ \includegraphics[width=0.45\linewidth]{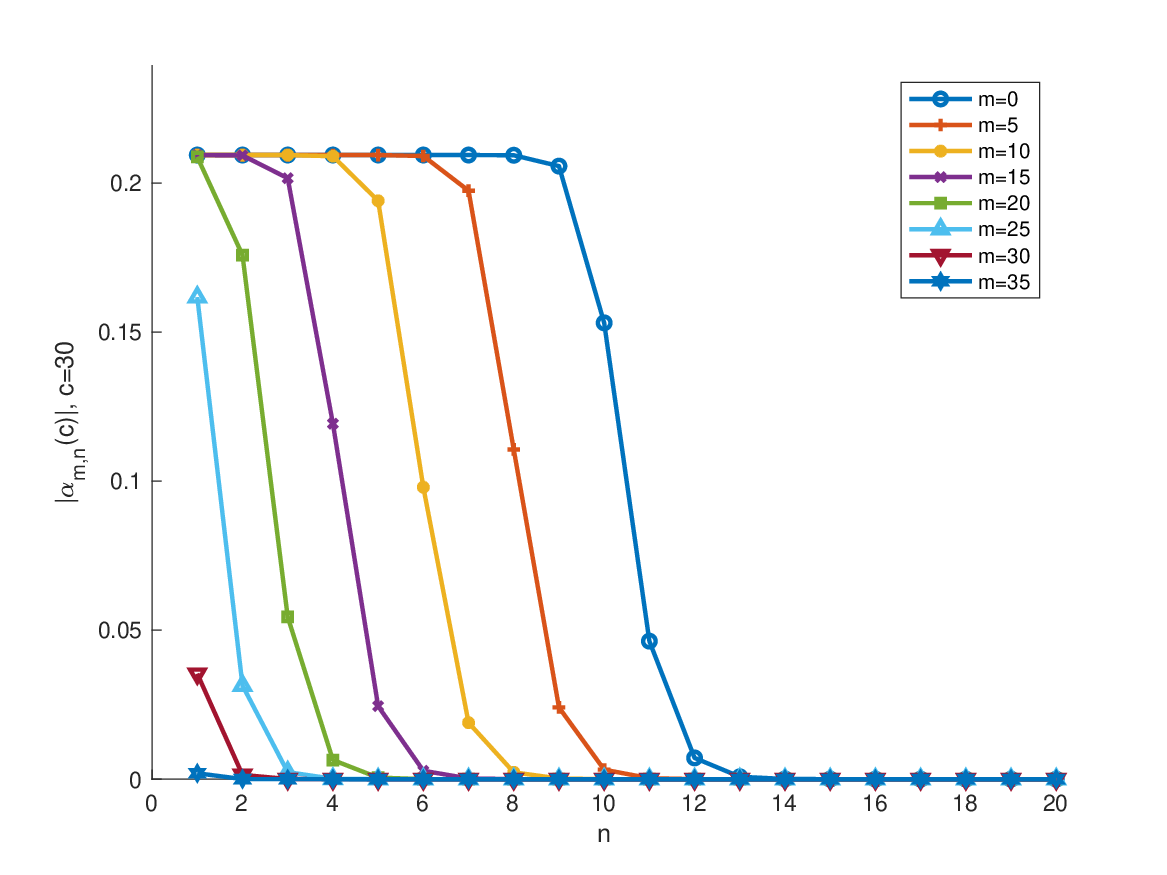} }
\subfloat[]{ \includegraphics[width=0.45\linewidth]{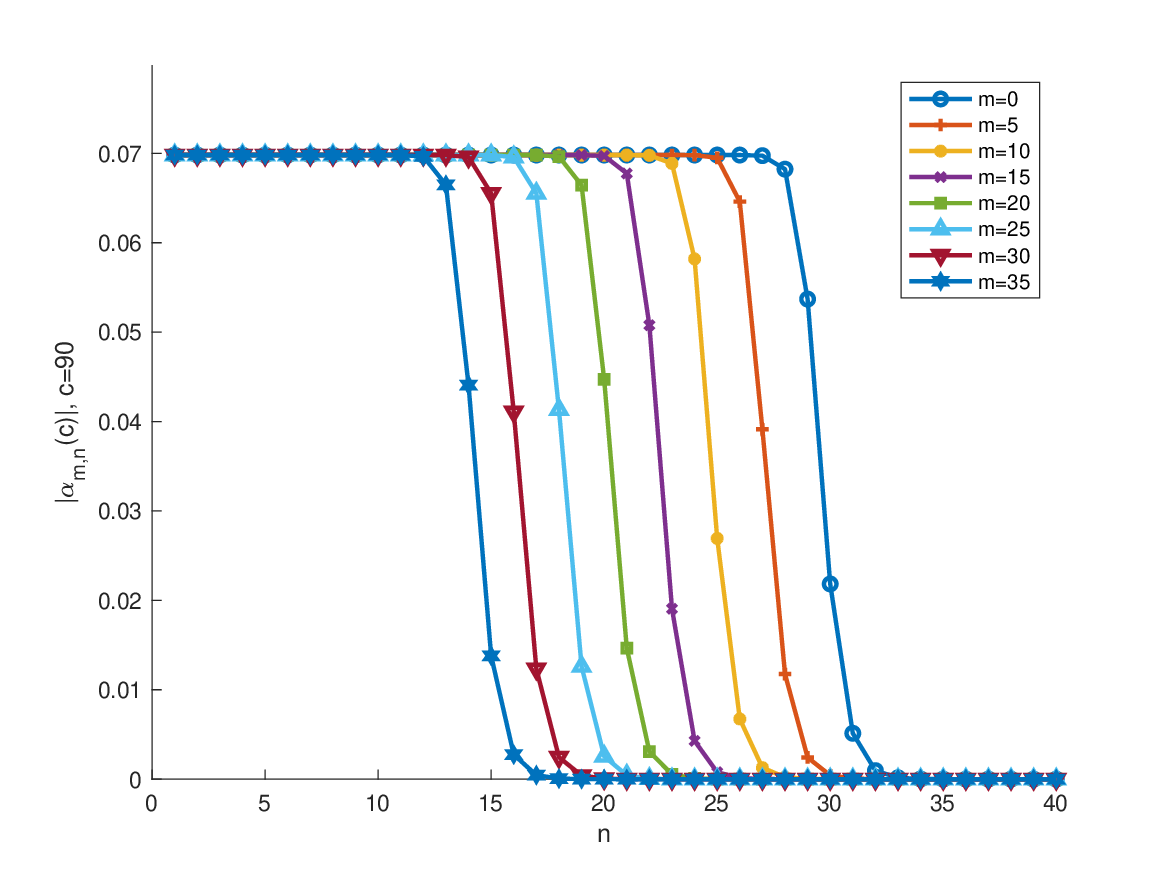} }\qquad

\caption{Distributions of the absolute value of eigenvalues $\alpha_{m,n}(c)$ for $m=0,~5,~10,~15,~20,~25,~30,~35$. (a)~$c=30$, (b) $c=90$.
}\label{figure: prolate eigs}
\end{figure}

The prolate eigenvalues $\alpha_{m,n}$ decay to zero exponentially fast and the dominant prolate eigenvalues are numerically  the same. %
See \Cref{figure: prolate eigs} for an illustration. Therefore we can choose the spectral cutoff regularization parameter according to the noises or perturbation errors. The strategy to choose the regularization parameter will be discussed in detail in the numerical experiments.
Recall that $J_{\epsilon}=\{ (m,n,\ell): { |\alpha_{m,n}| }> \epsilon \}$, we apply the spectral cutoff regularization to obtain that
\begin{align*}
    \tilde{q}^\epsilon (x)=\sum_{(m,n,\ell) \in J_{\epsilon}} q_{m,n,l} \psi_{m,n,l} (x;c).
\end{align*}
We summarize the low dimensional projection in \Cref{Algorithm: lrProj}.

{ 
In summary, given far-field data $\{u^{\infty}(\hat{x}_m;\hat{\theta}_\ell;k)\}_{m = 1,  ~\ell =1}^{N_1,N_2}$, we first apply \Cref{Algorithm: lrData} to obtain the matrix $U = (U_{ji})$ with $j=0, 1 \dots T-1$ and $i = 0, 1, \dots, M-1$, then we apply \Cref{Algorithm: lrProj} to compute the coefficients $\{u_{m,n,l}\}$ and $\{q_{m,n,l}\}$.  This allows us to reconstruct the { unknown} $q$. The algorithm is summarized in \Cref{Algorithm: lrInvScatt}.
}

For completeness we discuss the mathematical theory that supports our  method.
  In particular, suppose $u^\delta$ is a perturbation of $u$ in \eqref{inverse problem} such that $\|{   u^{\delta}} - {u }\|_{L^2(B)} \le \delta$. Let
 \begin{eqnarray*}
\quad q^{\delta,\alpha}
= \sum_{\chi_{m,n}(c)<\alpha^{-1}}   \frac{1}{ \alpha_{m,n}(c) } \left\langle {   u^{\delta}},  \psi_{m,n,\ell}(\cdot;c) \right\rangle_{B}  \psi_{m,n,\ell}(\cdot;c),
\end{eqnarray*}
and $\beta(\alpha)= \min_{ \chi_{m,n}(c)<\alpha^{-1}} \left\{   \left|\alpha_{m,n}(c)\right| \right\}$. Then the following follows from \cite{meng23data}.
\begin{lemma}
Let  $\|{   u^{\delta}} - {u }\|_{L^2(B)} \le \delta$.  If  $q \in  {H}^s(B)$ with $0<s<1/2$,  then
\begin{eqnarray*}
\|{q^{\delta,\alpha}} - q\|_{L^2(B)}  \le  { \frac{\delta}{{\beta(\alpha)}}} + { (4\alpha C) ^{s/2} (1+c^2)^{s/2}   \| q \|_{H^s\left(B\right)}},
\end{eqnarray*}
where $C\ge\sqrt{3}$ is a positive constant independent of $\delta$, $\alpha$, $s$ and $c$.
\end{lemma}
The first error term $ { \frac{\delta}{{\beta(\alpha)}}}$ is primarily due to the perturbation noises, and the second error term ${ (4\alpha C) ^{s/2} (1+c^2)^{s/2}   \| q \|_{H^s\left(B\right)}}$ is due to truncation error of finite prolate series expansion of $q$. In particular, if $q$ belongs to the span of finitely many disk PSWFs, then  the stability estimate is of Lipschitz. {  We can prove the following.
\begin{theorem}
Let  $\|{   u^{\delta}} - {u }\|_{L^2(B)} \le \delta$.  If  $q \in \mbox{span}\{ \psi_{m,n,\ell}(\cdot;c): { |\alpha_{m,n}|} > \eta\}$, { let
$$
q^\eta = \sum_{|\alpha_{m,n}|>\eta}   \frac{1}{ \alpha_{m,n}(c) } \left\langle {   u^{\delta}},  \psi_{m,n,\ell}(\cdot;c) \right\rangle_{B}  \psi_{m,n,\ell}(\cdot;c),
$$
}
then
\begin{eqnarray*}
\|{  q^\eta} - q\|_{L^2(B)}  \le  { \frac{\delta}{\eta}}.
\end{eqnarray*}
\end{theorem}
\begin{proof}
Due to the assumption that $q \in \mbox{span}\{ \psi_{m,n,\ell}(\cdot;c): { |\alpha_{m,n}| > \eta} \}$, then
\begin{eqnarray*}
\|{  q^\eta} - q\|_{L^2(B)}  \le \frac{1}{\mbox{min}\{ { |\alpha_{m,n}|: |\alpha_{m,n}| > \eta }\}} \|{u^{\delta}} - u\|_{L^2(B)} <  { \frac{\delta}{\eta}} .
\end{eqnarray*}
This completes the proof.
\end{proof}
}%

\begin{algorithm}
	\caption{Low-Rank Inverse Scattering}
\begin{algorithmic}[1]\label{Algorithm: lrInvScatt}
\Require   Far-field data $\{u^{\infty}(\hat{x}_m;\hat{\theta}_\ell;k): ~m = 1,2,\dots, N_1, ~\ell = 1,2,\dots, N_2\}$, spectral cutoff parameter $\epsilon$
\Ensure  Reconstruction $\tilde{q}^\epsilon$ of the unknown $q$
\State Precomputing: set $c=2k$, apply \Cref{Algorithm: PSWFs} to evaluate $\psi_{m,n,l}(x;c)$ and prolate eigenvalue $\alpha_{m,n}(c)$.
\State {  Data processing}: apply  \Cref{Algorithm: lrData} to obtain the matrix $U$ given by \eqref{DataMatrix}.
\State Low dimensional projection: apply Algorithm \ref{Algorithm: lrProj} to obtain the low dimensional vector $\{q_{m,n,\ell}\}_{(m,n,\ell) \in J_\epsilon}$ with $J_{\epsilon}:=\{ (m,n,\ell): { |\alpha_{m,n}(c)| }> \epsilon \}$.
\State Evaluate $\tilde{q}^\epsilon(x)=\sum_{(m,n,\ell) \in J_\epsilon} q_{m,n,l} \psi_{m,n,l} (x;c)$.
\end{algorithmic}
\end{algorithm}

\section{Numerical experiments} \label{section: numerical experiments}
In this section, we perform numerical experiments for both the Born far-field data and the full far-field data. We aim to test resolution capability, robustness against randomly added noise and modeling errors, and numerical evidence of increasing stability. Furthermore, we compare its performance against standard fft and iterative methods.

\subsection{Born model}
We begin with the Born far-field data. Recall from the Born model \eqref{Born fourier} that
\begin{align*}
    u_b^\infty (\hat{x};\hat{\theta};k)&=k^2\int_\Omega e^{-ik\hat{x}\cdot y}q(y) e^{iky\cdot \hat{\theta}}dy.
\end{align*}
One finds that
$$
  u_b^\infty (\hat{x};\hat{\theta};k) = k^2 \int_B e^{ic \frac{\hat{\theta} - \hat{x}}{2}  \cdot y}q(y)  \ind y,
$$
and we set
\begin{equation} \label{section: numerical : data generation}
u(p;c)  = \int_B e^{ic p  \cdot y}q(y)  \ind y, \quad p= \frac{\hat{\theta} - \hat{x}}{2}, \quad c = 2k.
\end{equation}

\subsubsection{Data generation}
 In this section, we discuss the Born data generated in our algorithm.
We generate the far-field data by evaluating the integral \eqref{section: numerical : data generation} analytically {  with the following explict formulas}.
\begin{itemize}
    \item Eigenfunction. Let $q(x)=\psi_{m,n,l}(x)$ be one of the eigenfunctions, then $u(x;c)= \alpha_{m,n}(c) \psi_{m,n,l}(x)$.

\item Constant $q$ in a disk/rectangle.
\begin{itemize}
    \item Constant in a disk: $q(x)=1_{|x| < a}(x)$. This leads to
    $$
    u(x;c)=\frac{2\pi a}{c|x|}J_1(ac|x|),
    $$ where $J_1$ is the Bessel function of the first kind of order  $\nu=1$, see for instance \cite{Abramowitz64}.

\item Constant in a rectangle: $q(x)=1_{x\in \Omega}(x)$, $\Omega=\{x\in \mathbb{R}^2:  |x_1|,|x_2| < 1/2\}$. This leads to
\begin{equation*}
    u(x;c)=\frac{4\sin{(\frac{c}{2}x_1)}\sin{(\frac{c}{2}x_2)}}{c^2x_1 x_2}.
\end{equation*}
\end{itemize}

\item  Oscillatory contrast: $
    q(x_1,x_2)= \sin m\pi x_1 \cdot 1_\Omega(x)$, $\Omega=\{x\in \mathbb{R}^2:  |x_1|,|x_2| < 1/2\}$, $m\in \mathbb{Z}$.
This leads to
\begin{equation*}
    u(x;c)= (-i)\frac{2 \sin \frac{c x_2}{2}}{c x_2}\Big(\frac{\sin{\frac{cx_1+m\pi}{2}}}{cx_1+m\pi}-\frac{\sin{\frac{cx_1-m\pi}{2}}}{cx_1-m\pi} \Big).
\end{equation*}

\item  Constant in three  rectangles: $
\cup_{j=1}^3 \Omega_j$ where each $\Omega_j$ is a rectangle.

We give the formula for generating the data for a fixed rectangle, then for three different rectangles, it is sufficient to apply the superposition.

For a rectangular region $\Omega=\{(x,y)\in \mathbb{R}^2|a_1 < x_1 < a_2, b_1 < x_2 < b_2 \}$ that supports constant $q=1$, we can compute
\begin{equation*}
    u(x;c)=4\frac{\sin(\frac{a_2-a_1}{2}cx_1)\sin(\frac{b_2-b_1}{2}cx_2)}{c^2 x_1 x_2}e^{i (\frac{a_2+a_1}{2}c x_1+\frac{b_1+b_2}{2}c x_2)}.
\end{equation*}

\end{itemize}

In the following, we test our algorithm step by step:  first with the exact post-processed data and then with the far-field data.

\subsubsection{Noiseless and noisy post-processed data}

In this section, we first work with the exact post-processed data. This is to implement the \Cref{Algorithm: lrInvScatt} by skipping the data processing \Cref{Algorithm: lrData}. To be more precise, we consider the following discrete data $\{\tilde{u}(t_j,\theta_i)|0\leq j\leq T-1,0\leq i\leq M-1\}$ in the unit disk, i.e.,
\begin{equation*}
    U =
    \begin{pmatrix}
        \tilde{u}(t_0,\theta_0;c) & \tilde{u}(t_0,\theta_1;c) & \cdots & \tilde{u}(t_0,\theta_{M-1};c) \\
        \tilde{u}(t_1,\theta_0;c) & \tilde{u}(t_1,\theta_1;c) & \cdots & \tilde{u}(t_1,\theta_{M-1};c) \\
                   \vdots & \vdots & \ddots & \vdots\\
        \tilde{u}(t_{T-1},\theta_0;c) & \tilde{u}(t_{T-1},\theta_1;c) & \cdots & \tilde{u}(t_{T-1},\theta_{M-1};c)
    \end{pmatrix}.
\end{equation*}
where $\{t_j,\omega_{t_j}\}_{j=0}^{T-1}$ are the  Gauss-Legendre quadrature nodes and weights,  $\{\theta_i=\frac{2i\pi}{M},\omega_{\theta_i}=\frac{2\pi}{M}\}_{i=0}^{M-1}$ are the trapezoidal quadrature nodes and weights, and
$$
\tilde{u}(t_j,\theta_i;c) = u(p_n;c),  \quad \{p_n\}_{n=1}^{TM}= \big\{\sqrt{(t_j+1)/2} (\cos\theta_i,\sin \theta_i)^T \big\}_{j=0,~i=0}^{T-1,~M-1}.
$$
Here $u(x;c)$ is given by evaluating \eqref{section: numerical : data generation} analytically.
In \Cref{Algorithm: PSWFs} we set $K = 146$ in \eqref{section: computing disk PSWFs: polynomial expansion} which is sufficient for all of our experiments.

We add random uniformly distributed noise of noise level $\delta$ to each entry of $U$ to get the noisy matrix $U^\delta$, i.e.,
{
\begin{equation}\label{add noise}
\tilde{u}^\delta(t_j,\theta_i)= \tilde{u}(t_j,\theta_i) \left(1+ \delta \xi_{ji} \right), \, 0\leq j\leq T-1,0\leq i\leq M-1
\end{equation}
where { $\xi_{ji}\sim U(-1,1)$, indicating that $\xi_{ji}$  is a uniformly distributed random number between $-1$ and $1$}.
}

It is observed that the prolate eigenvalues $\alpha_{m,n}$ decay to zero exponentially fast and the dominant prolate eigenvalues have {  the same numerical magnitude}. See \Cref{figure: prolate eigs} for an illustration. We {  choose} to implement \Cref{Algorithm: lrInvScatt} by choosing a spectral cutoff $\epsilon$ as
{  
\begin{equation*}
    \epsilon=\begin{cases}
    0.1|\alpha_{0,0}(c)|,~\delta= 0, &\qquad \mbox{(Case: Born data or analytic data)}\\\delta|\alpha_{0,0}(c)|,~~~\delta\neq 0, &\qquad \mbox{(Case: Born data or analytic data)}\\0.9|\alpha_{0,0}(c)|,
    &\qquad \mbox{(Case: full simulation  data)}
\end{cases},
\end{equation*}
}
where $\delta$ is the noise level. %

First we test the algorithm for $q$ being one of the disk PSWFs and we gradually perform more complicated examples. In \Cref{figure: psi} we plot the ground-truth of $q$ (left) and reconstruction of $q$ (right). The reconstruction is under our expectation.
\begin{figure}[htbp]
\centering
\includegraphics[width=0.45\linewidth]{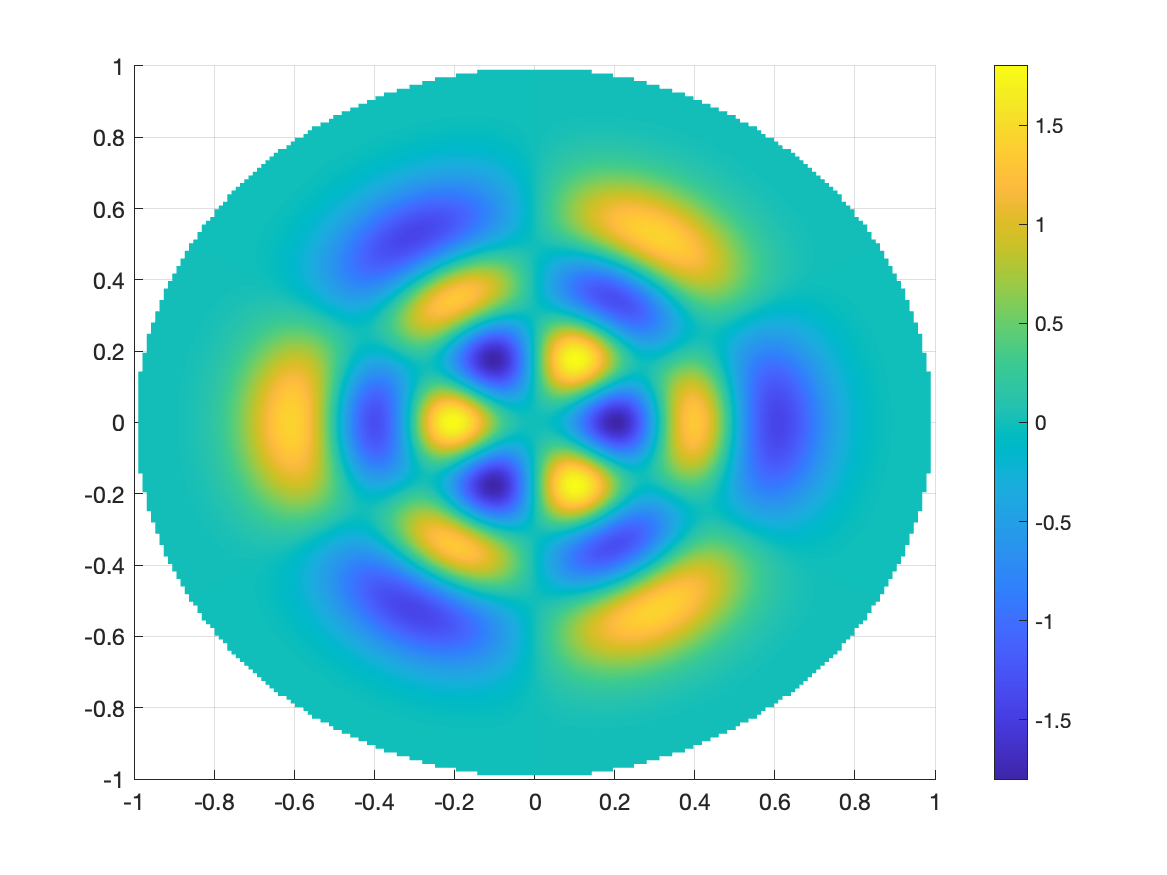}
\includegraphics[width=0.45\linewidth]{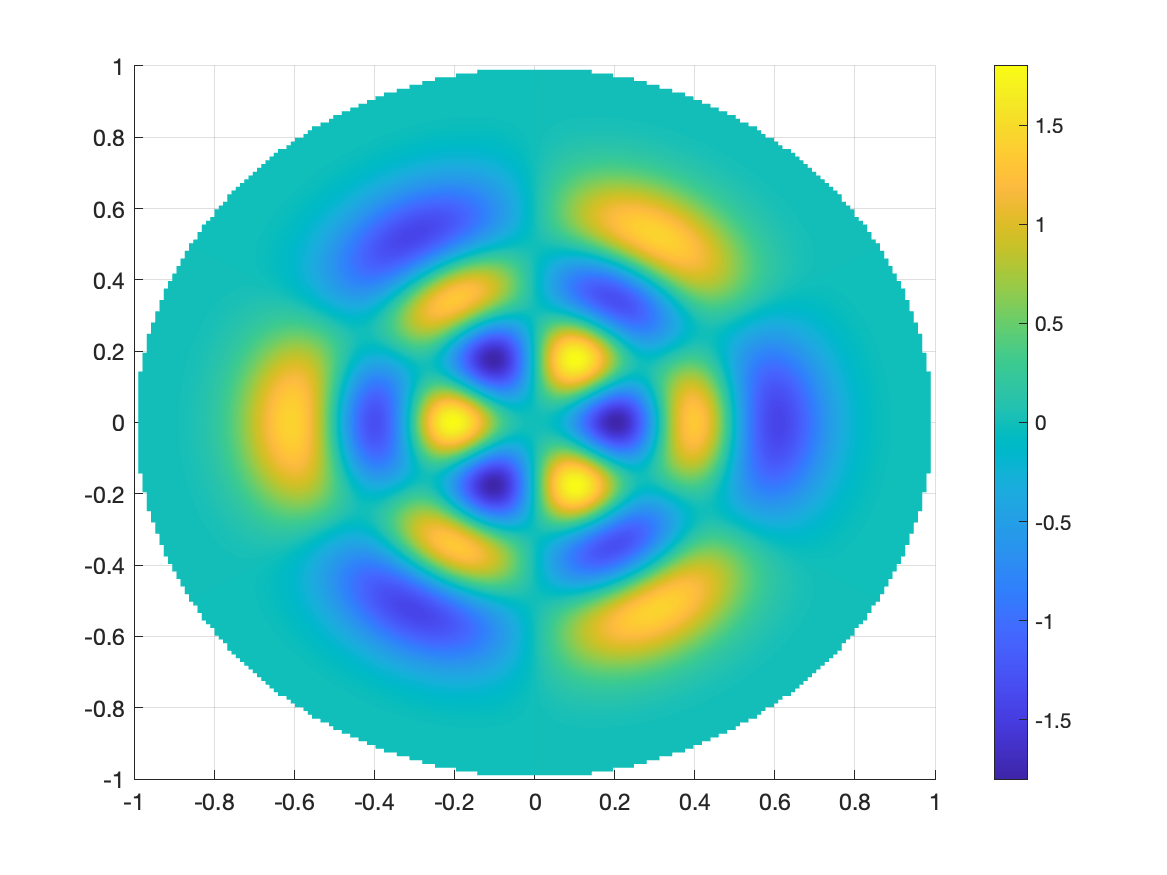}
\caption{Reconstruction of a disk PSWF. Left: $\psi_{3,2,2}(x;30)$, right: reconstruction. } \label{figure: psi}
\end{figure}
We further test the performance of our algorithm for the three different types of unknowns: constant $q$ in a disk, constant $q$ in a rectangle, and oscillatory function $q$. Here the disk is given by
$$
\{(x_1, x_2)\in\mathbb{R}^2: \sqrt{x_1^2+x_2^2} < 1/2\},
$$
the rectangle is given by
$$
\{(x_1, x_2)\in\mathbb{R}^2: -1/2< x_1, x_2 < 1/2\},
$$
and the oscillatory function is given by
$$
\sin(8\pi x_1) \mbox{ supported in } \{(x_1, x_2)\in\mathbb{R}^2: -1/2< x_1, x_2 < 1/2\}.
$$
Numerical experiments in  \Cref{example: U noiseless noise} demonstrate the capability of our algorithm.

\begin{figure}[htbp]
\centering

\includegraphics[width=0.3\linewidth]{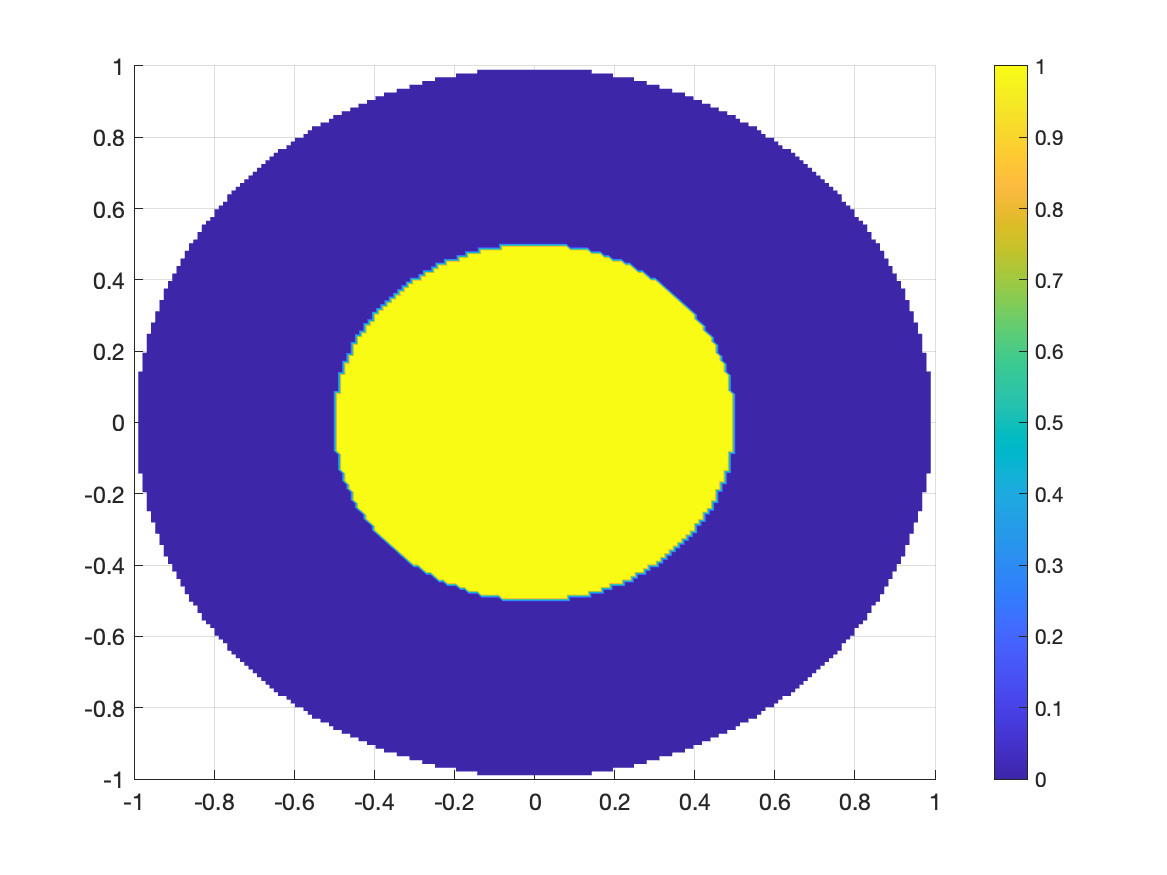}
\includegraphics[width=0.3\linewidth]{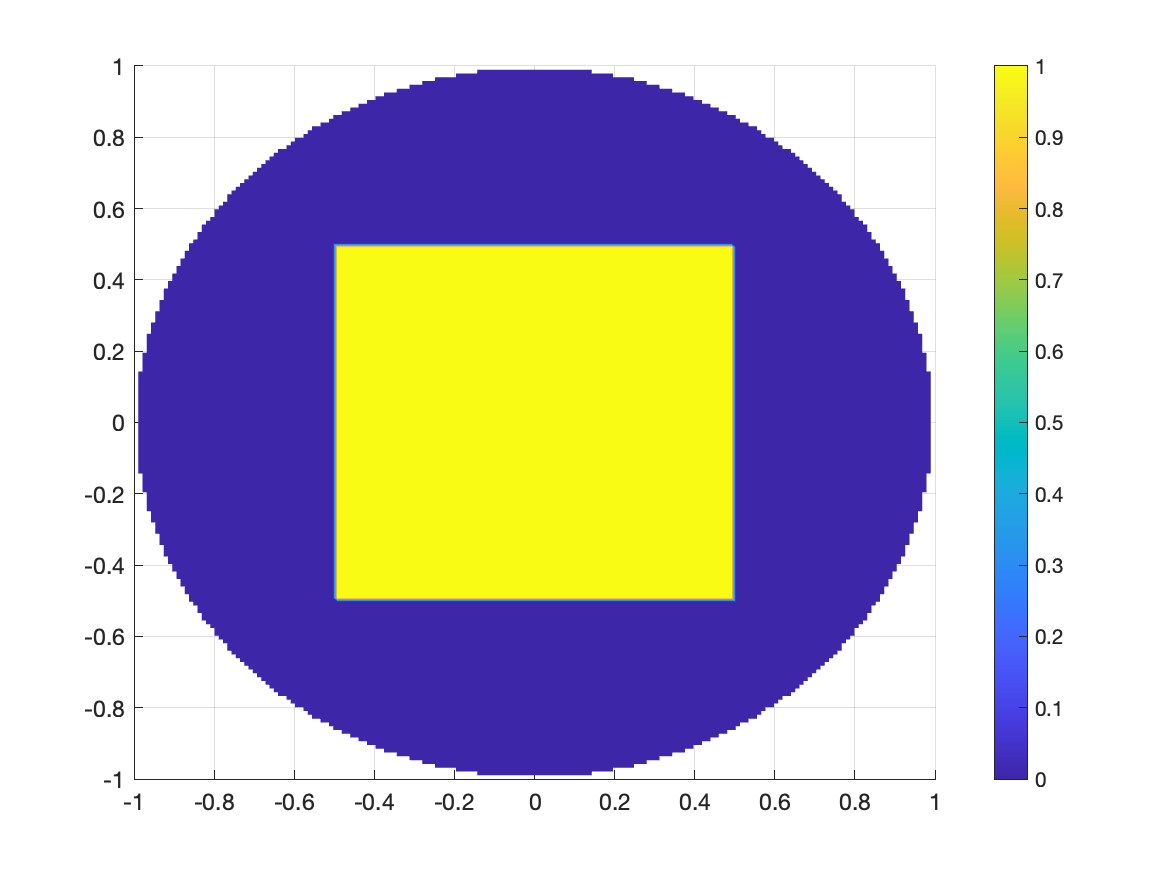}
\includegraphics[width=0.3\linewidth]{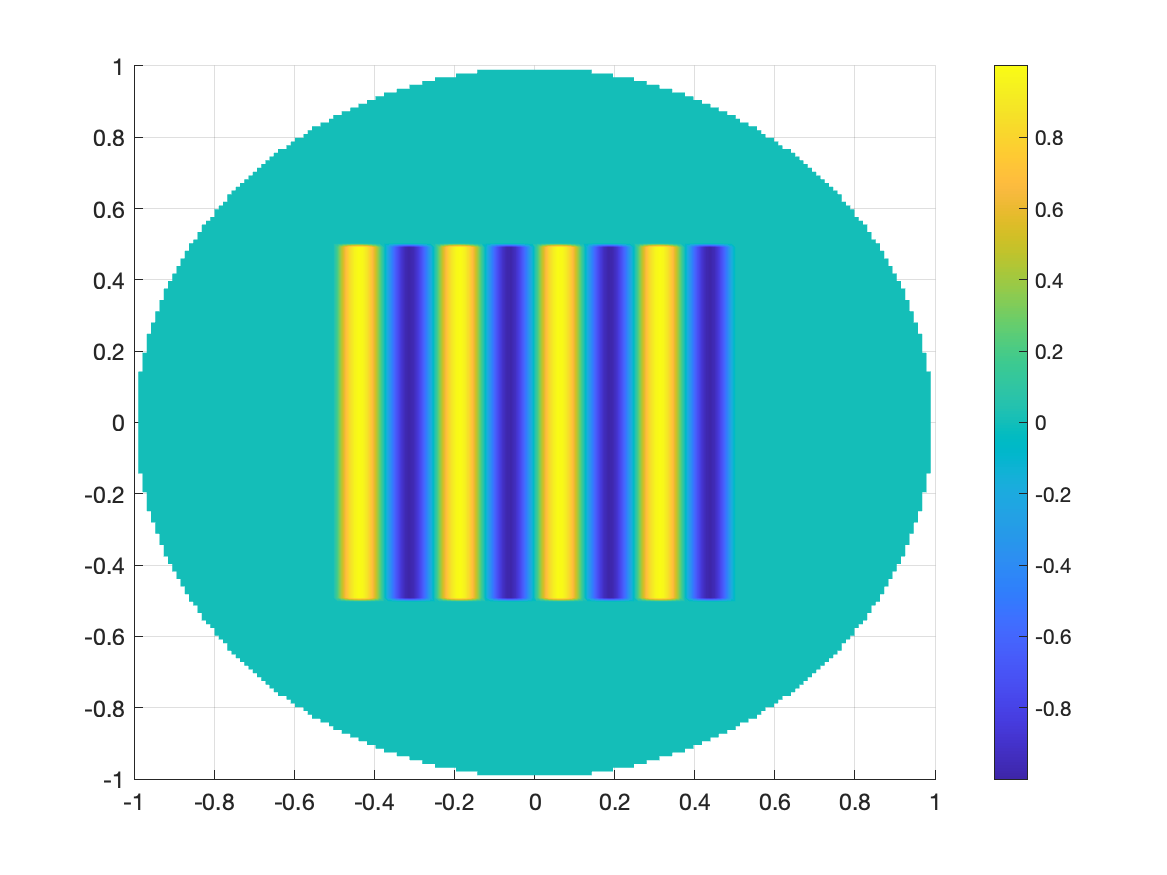}
\\
\includegraphics[width=0.3\linewidth]{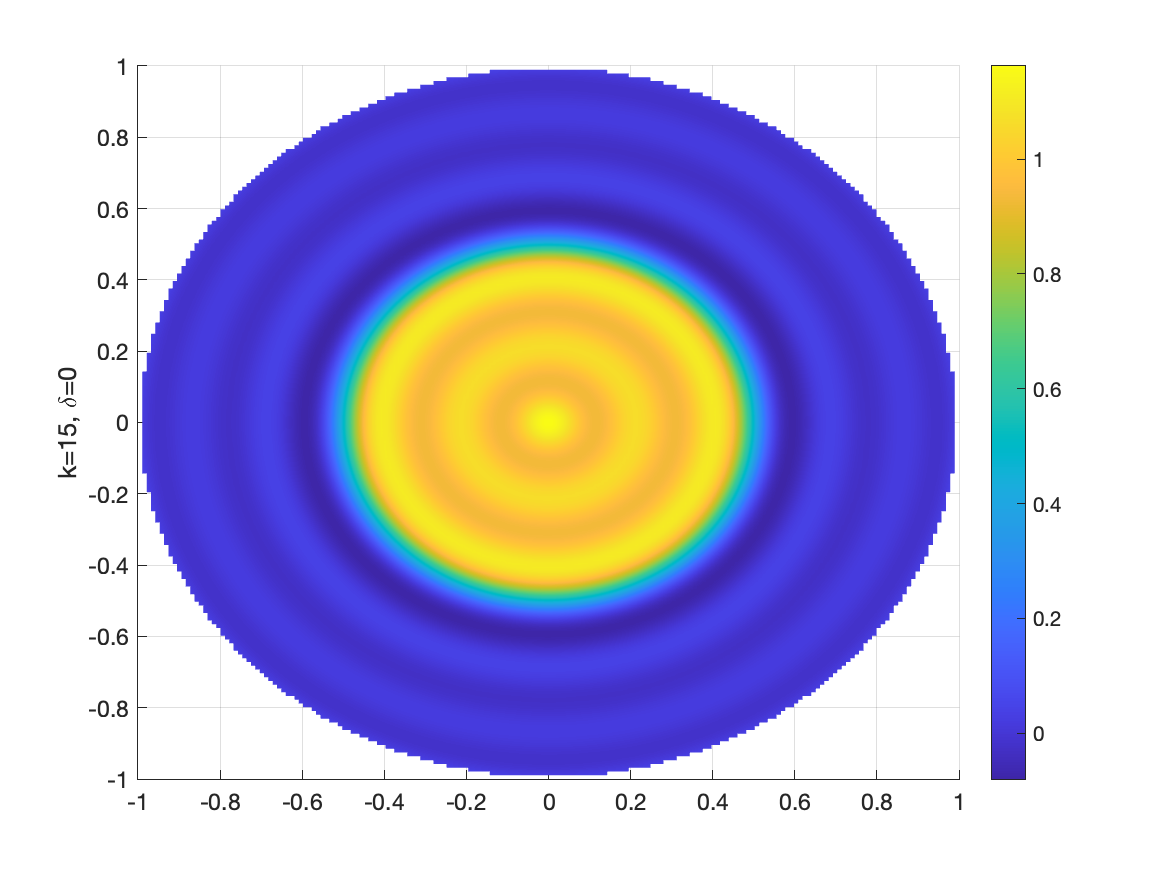}
\includegraphics[width=0.3\linewidth]{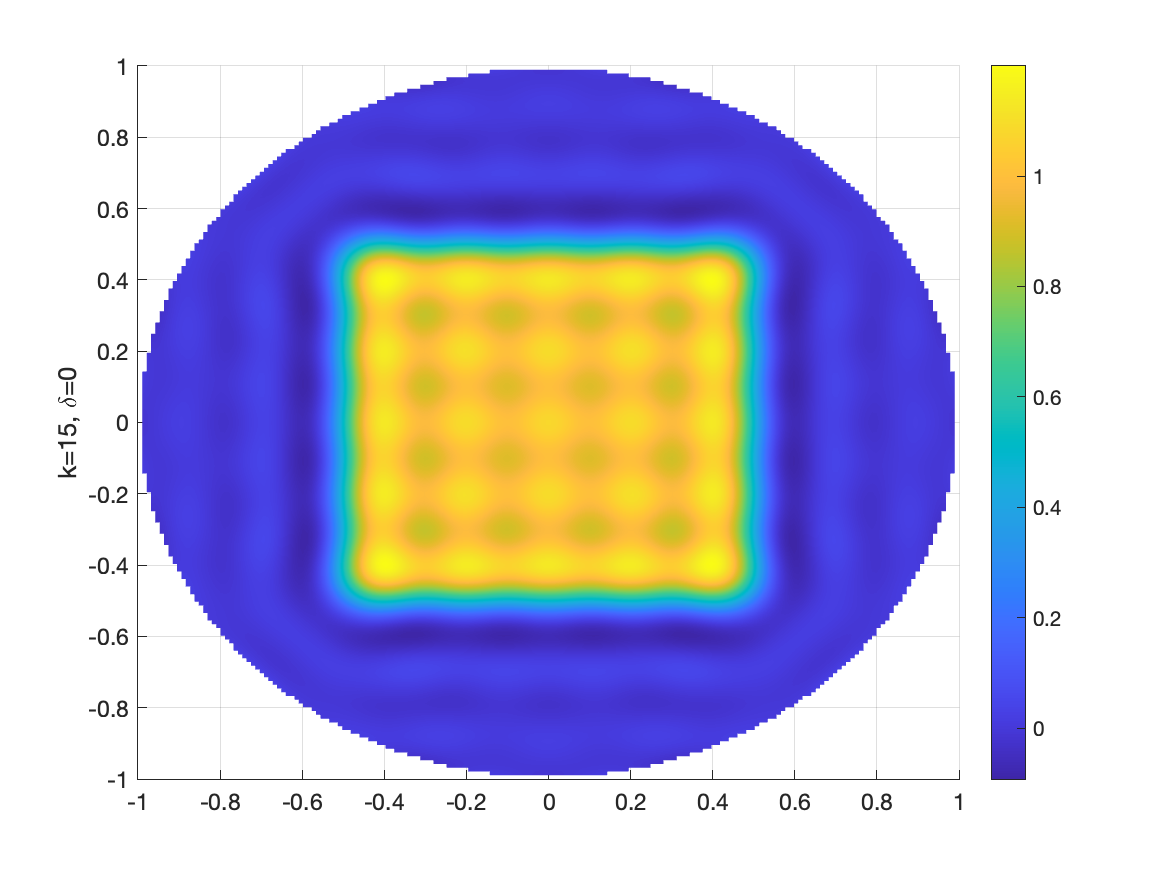}
\includegraphics[width=0.3\linewidth]{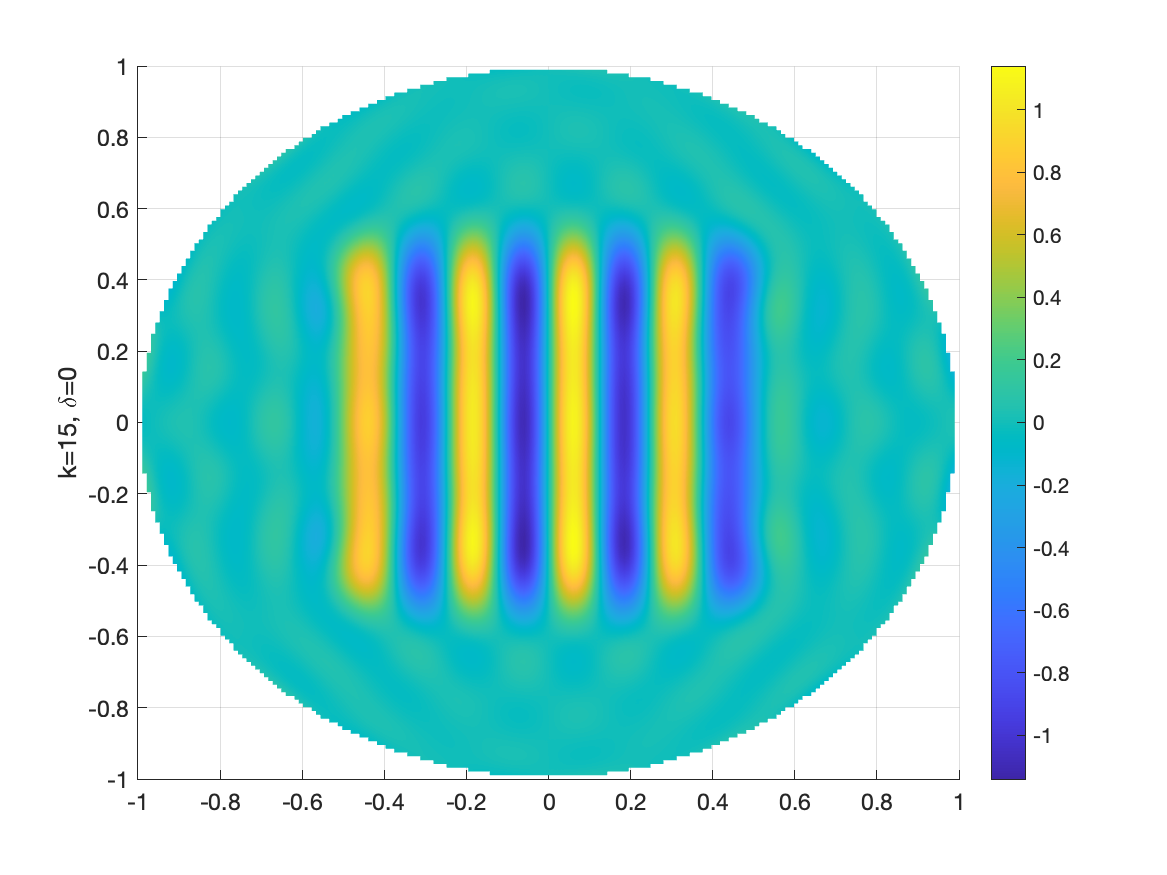}
\\
\includegraphics[width=0.3\linewidth]{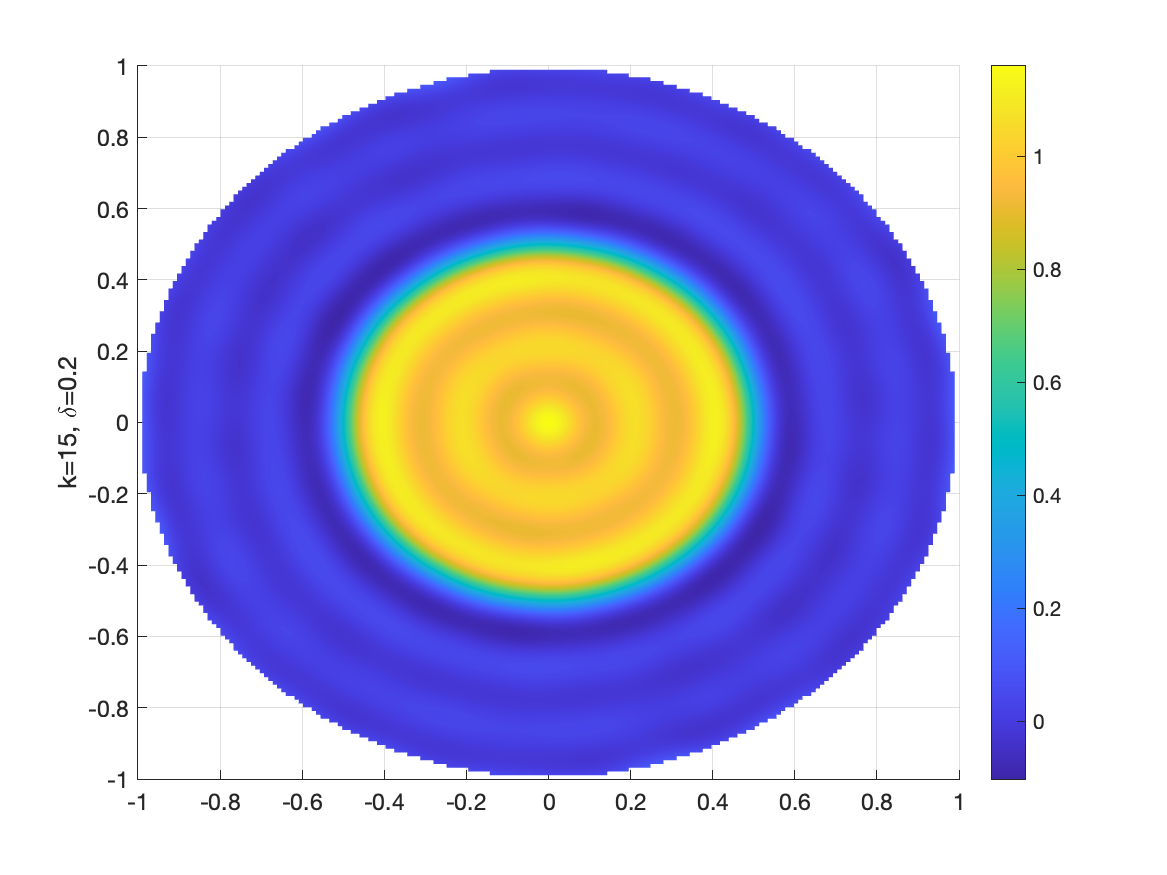}
\includegraphics[width=0.3\linewidth]{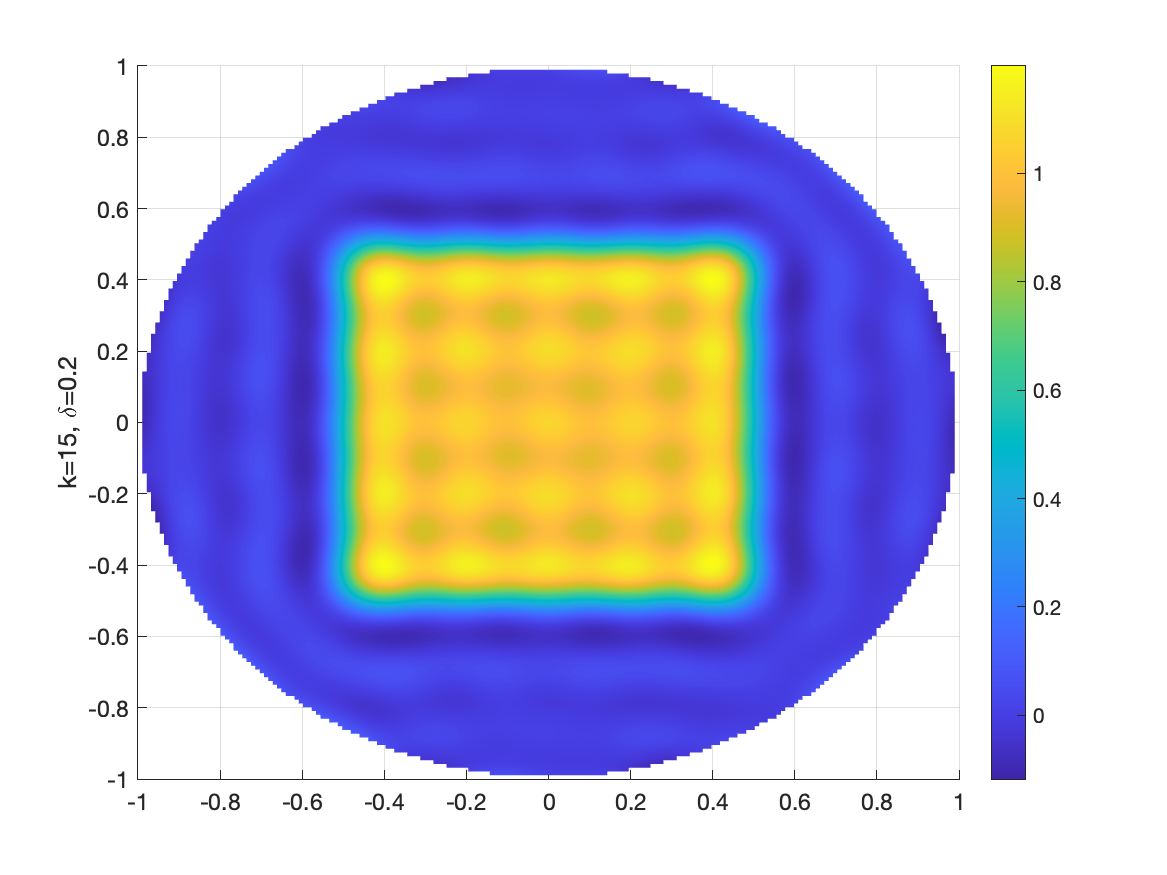}
\includegraphics[width=0.3\linewidth]{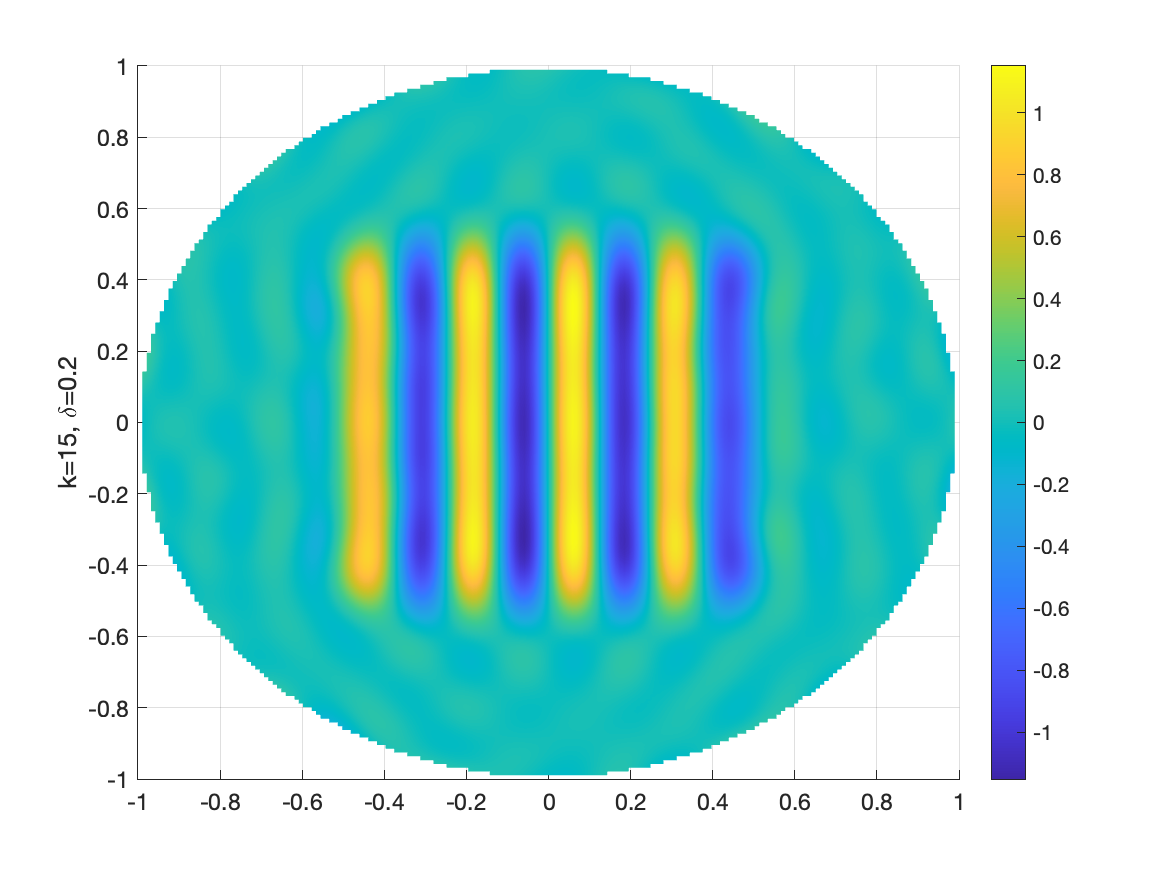}

  \caption{Reconstruction of three different types of unknowns.  First row: exact. Second row:  Noiseless data. Third row: $20\%$ noisy data.}  \label{example: U noiseless noise}

\end{figure}

Next, we consider perturbed noisy data given by \eqref{add noise} where we add   $\delta=20\%$   to the exact data $U$. The reconstruction is plotted in the bottom row of \Cref{example: U noiseless noise} which shows that the algorithm is robust to the added uniform noises. We report that we also tried Gaussian noises {  $\xi_{ji}\sim \mathcal{N}(0,1)$} and observed similar performance.

{ 
It is pedagogical to observe the changes in the reconstruction by gradually increasing the dimension of the low-rank space. In \Cref{example: Different numbers of PSWFs} for the rectangular contrast, we plot the reconstructions for different dimensions of the low-rank space. The leading disk PSWFs capture the principal support of the contrast; as the dimension of the low-rank space increases, the reconstruction becomes better.}

\begin{figure}[htbp]
\centering

\includegraphics[width=0.3\linewidth]{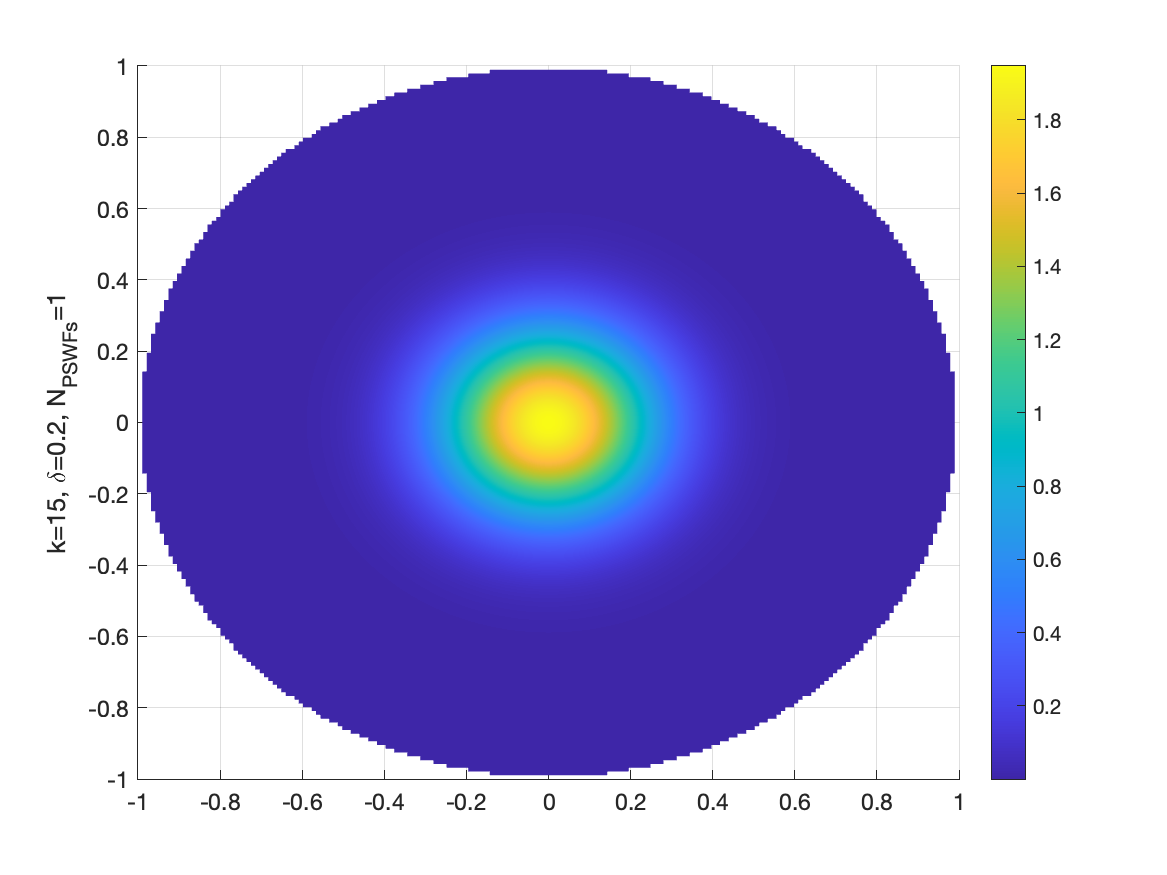}
\includegraphics[width=0.3\linewidth]{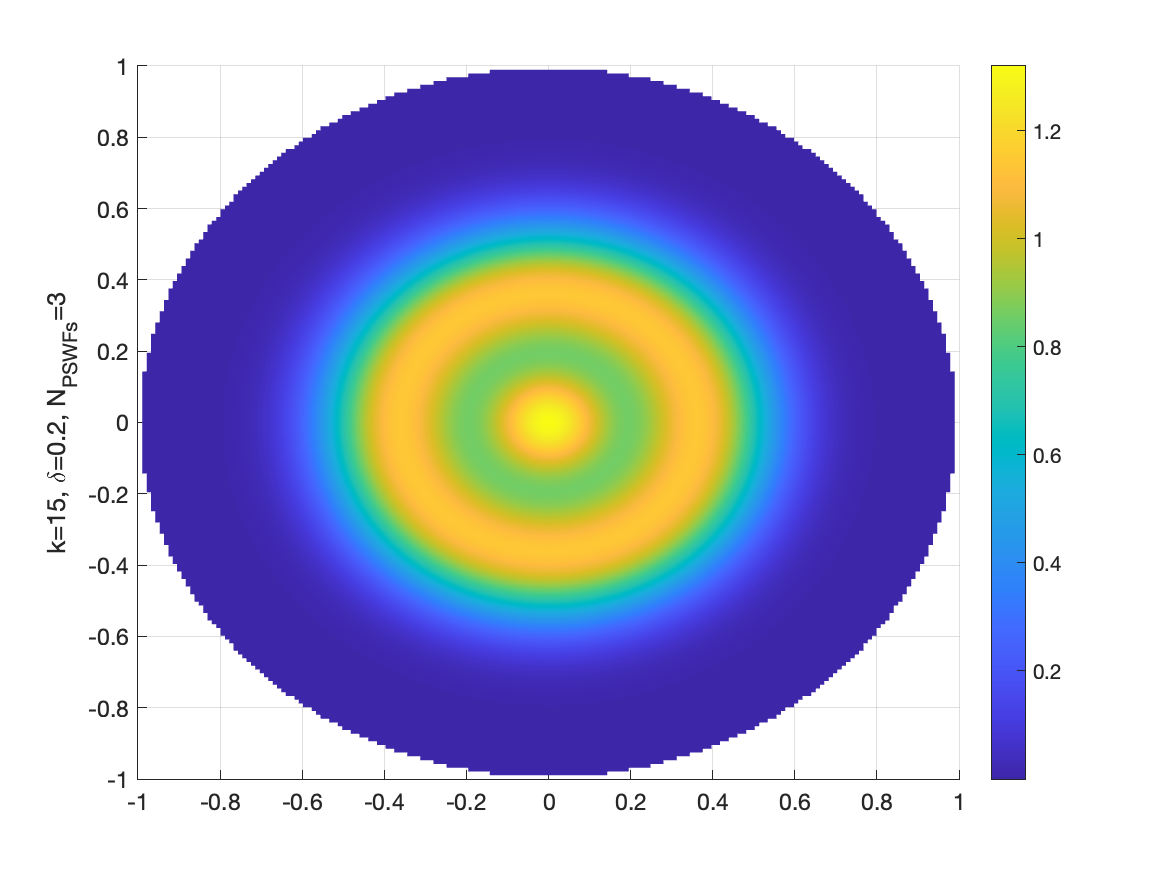}
\includegraphics[width=0.3\linewidth]{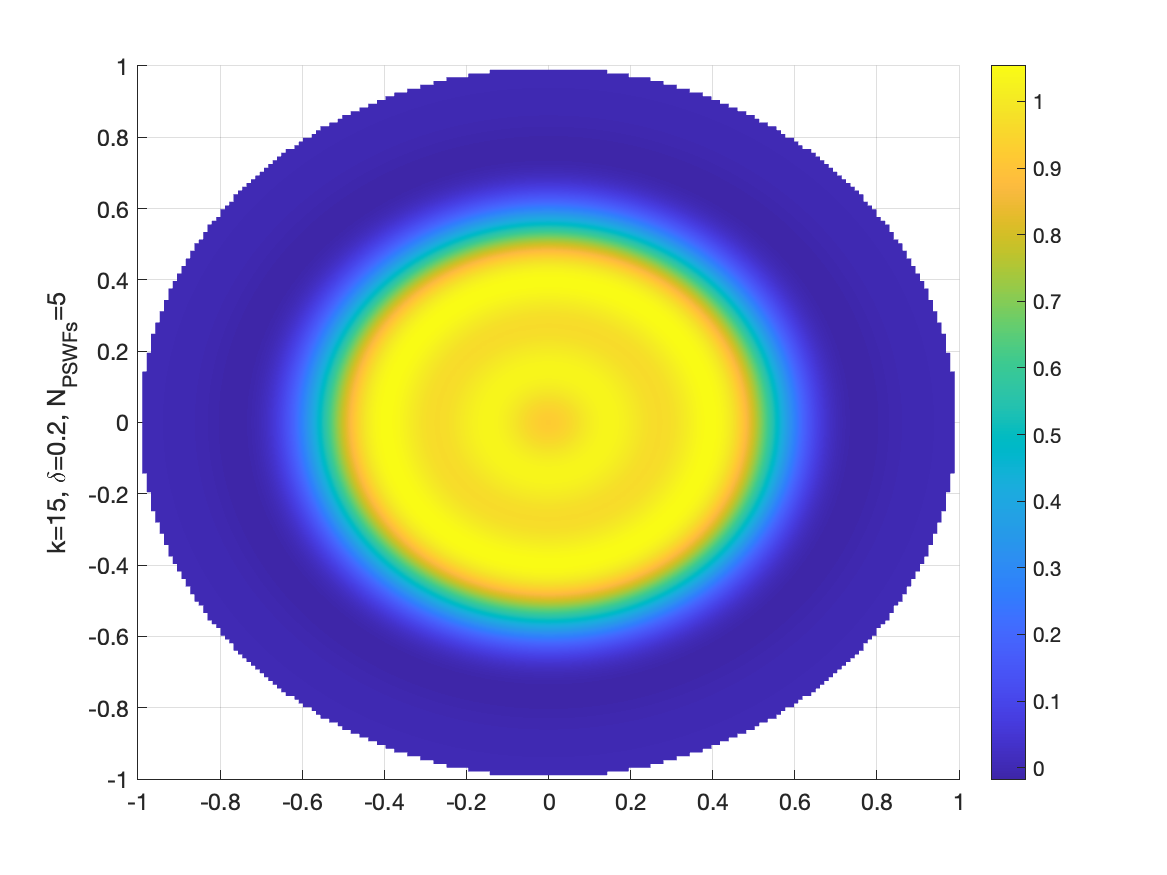}
\\
\includegraphics[width=0.3\linewidth]{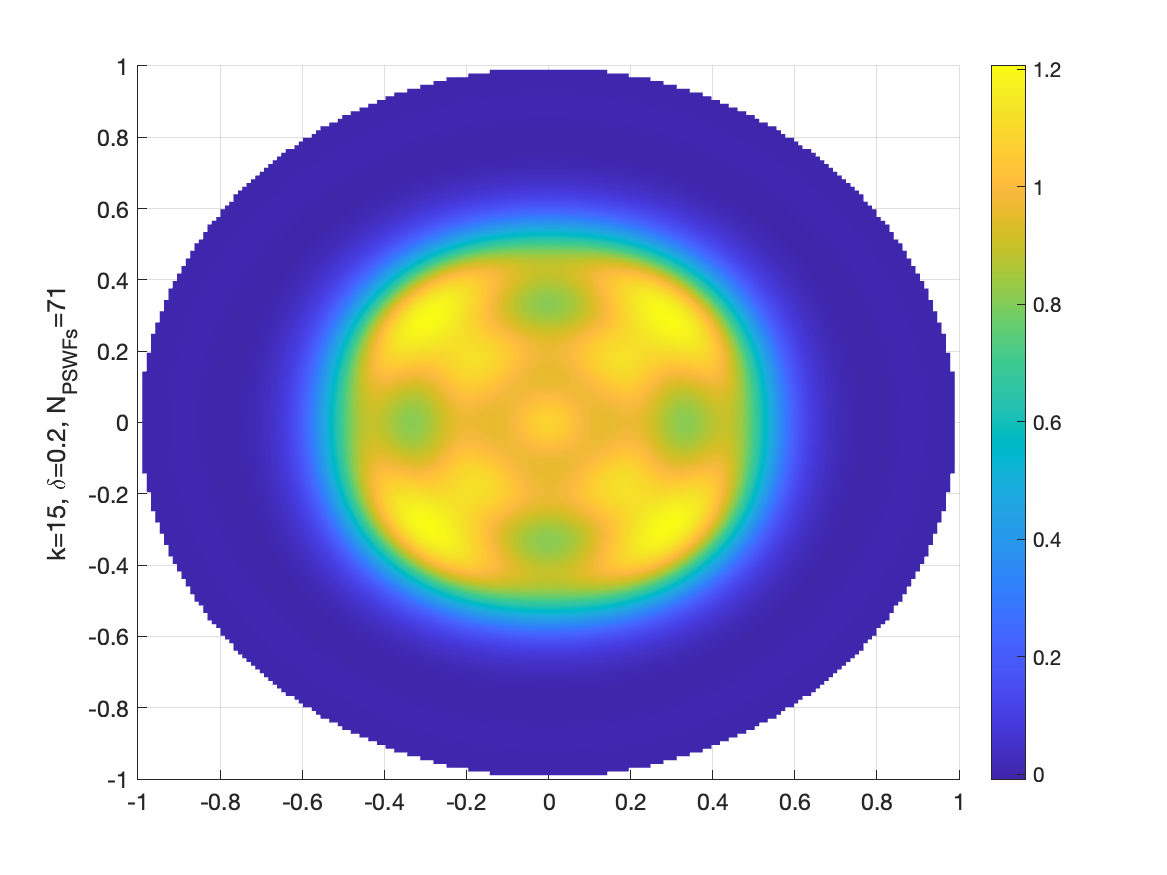}
\includegraphics[width=0.3\linewidth]{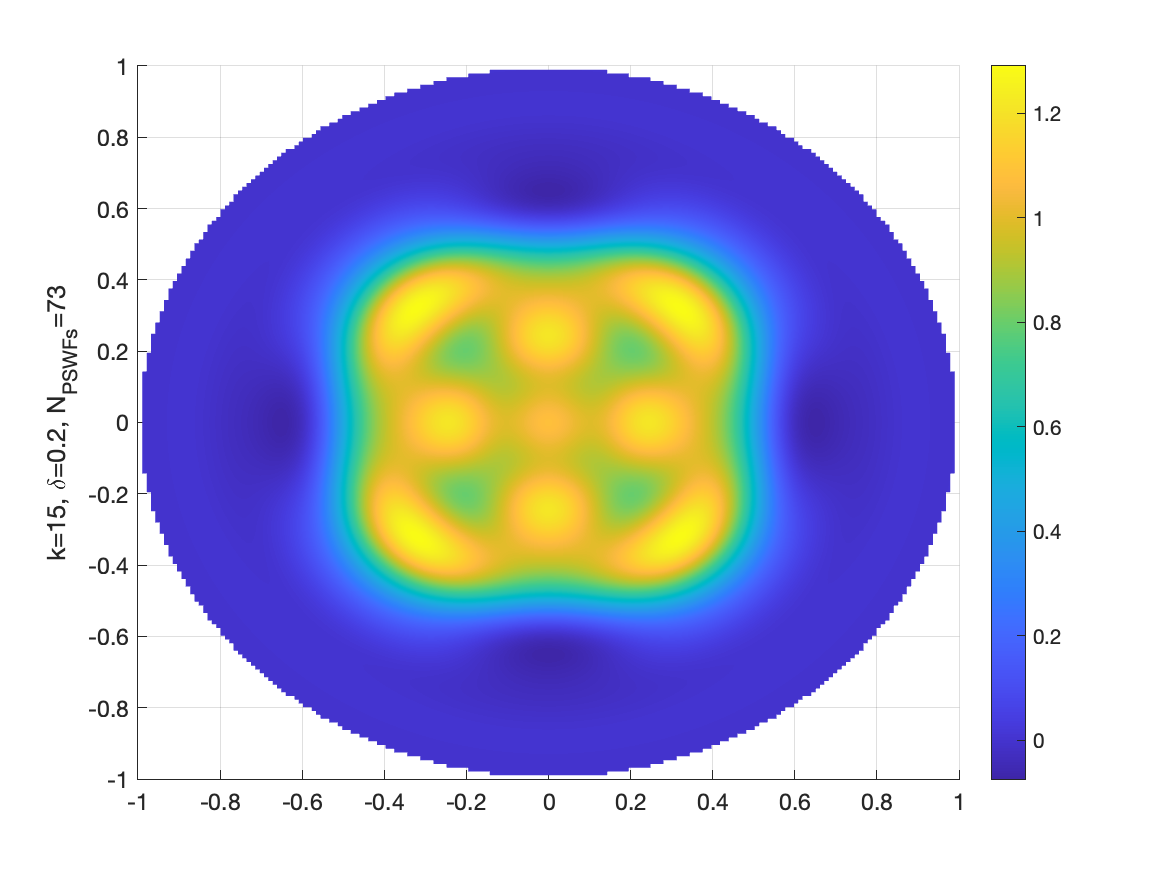}
\includegraphics[width=0.3\linewidth]{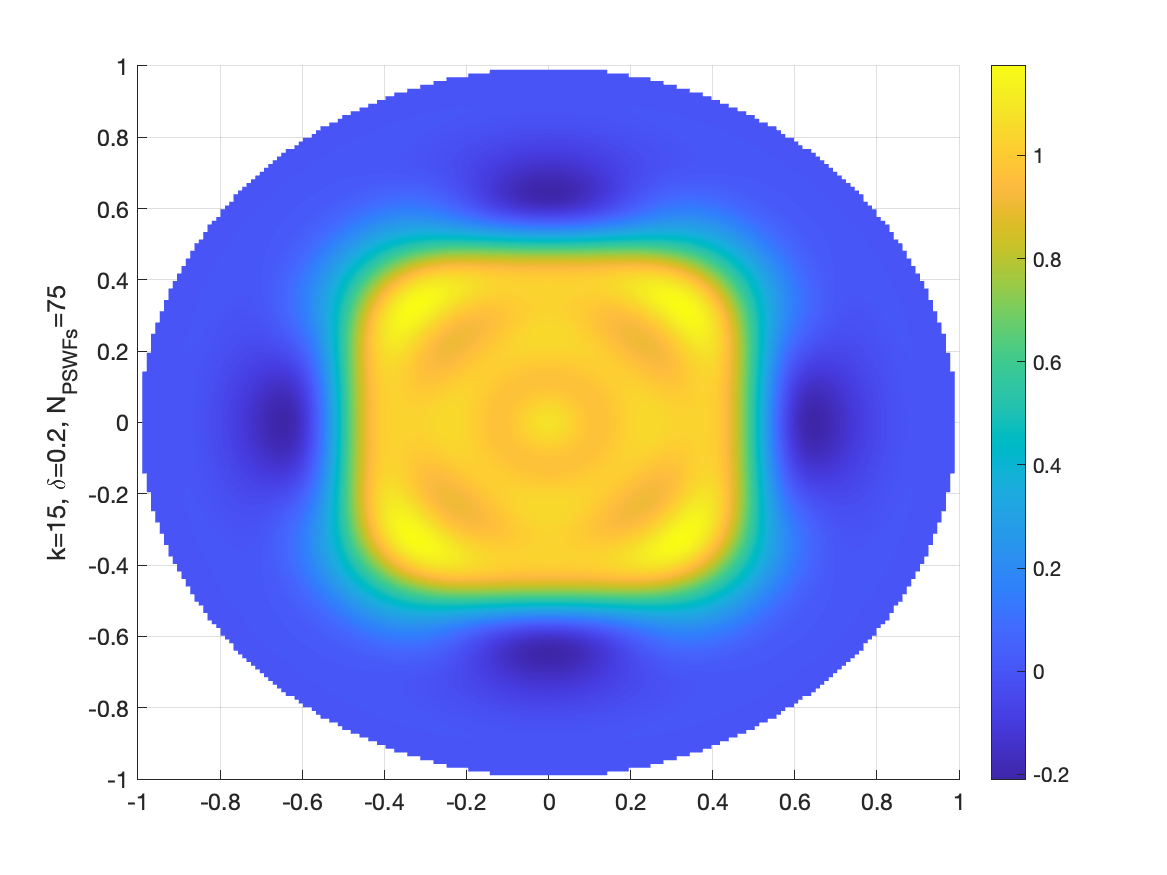}
\\
\includegraphics[width=0.3\linewidth]{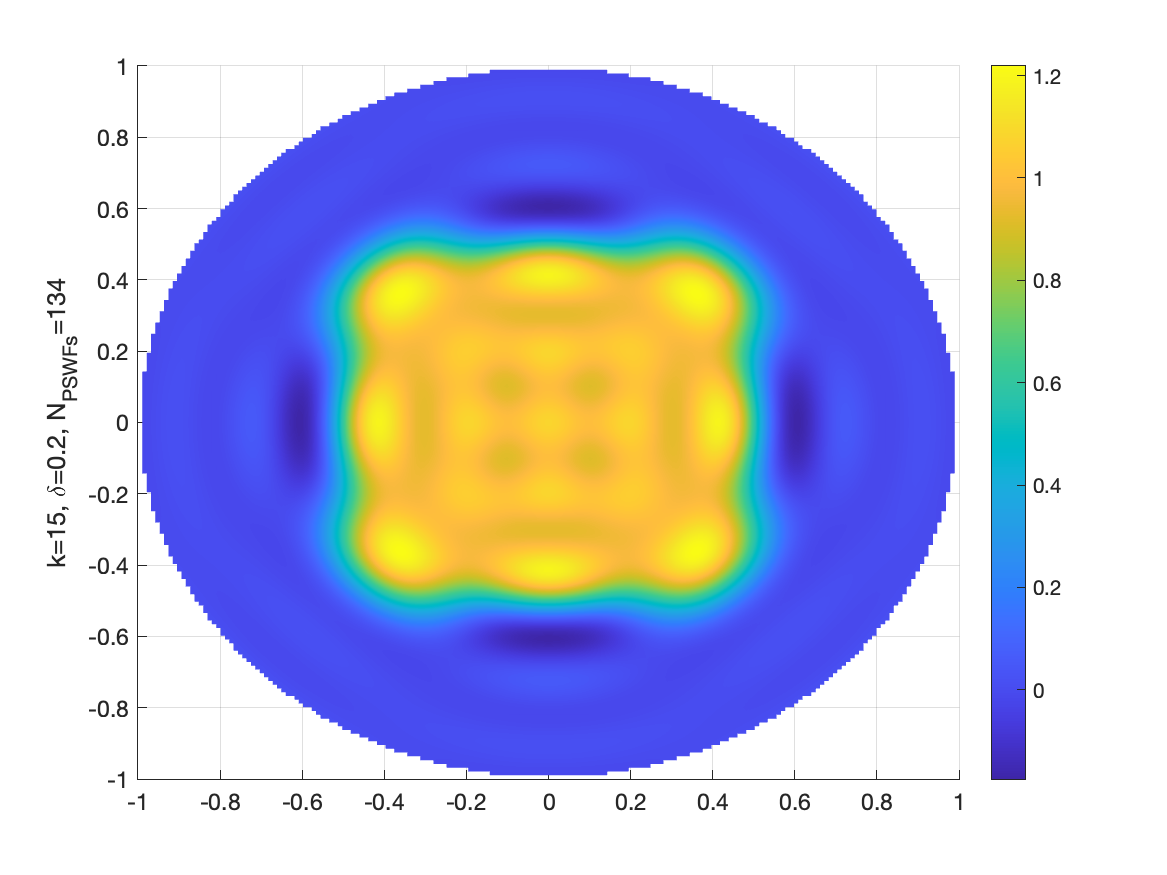}
\includegraphics[width=0.3\linewidth]{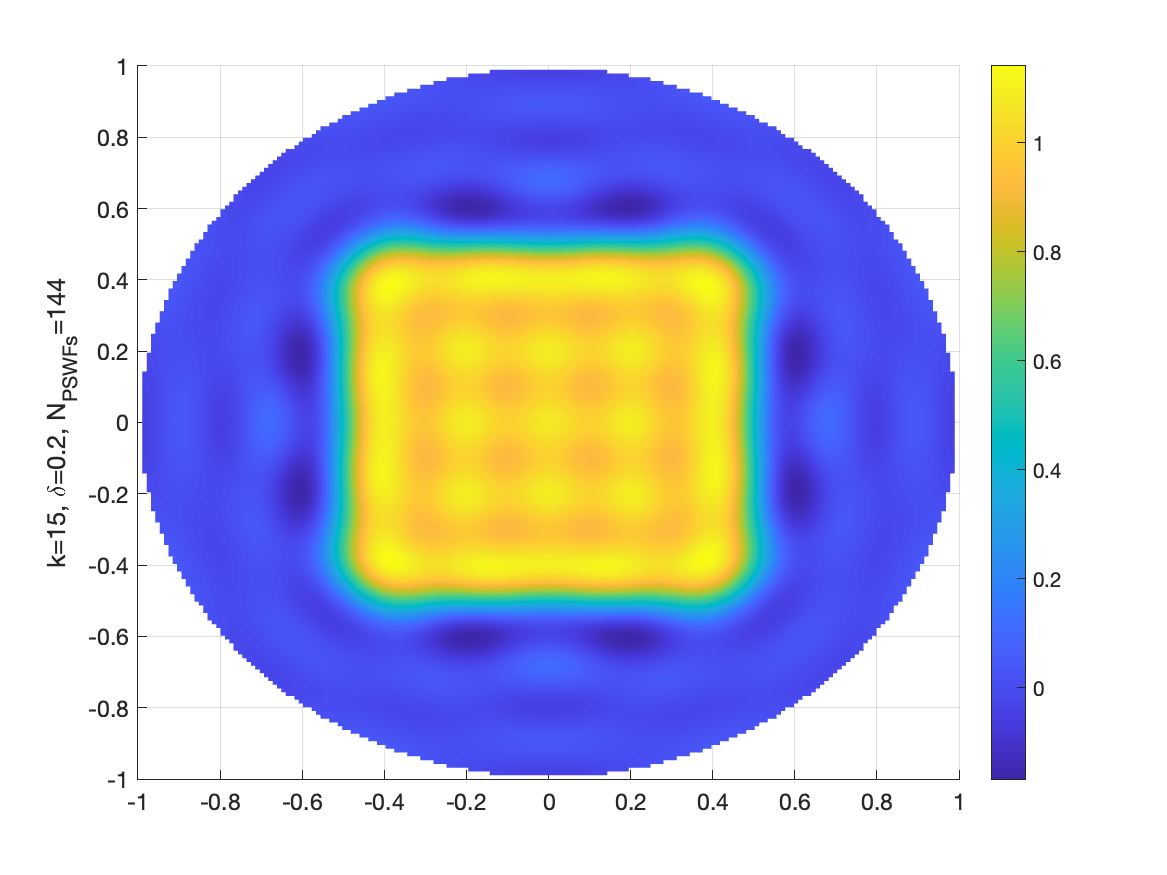}
\includegraphics[width=0.3\linewidth]{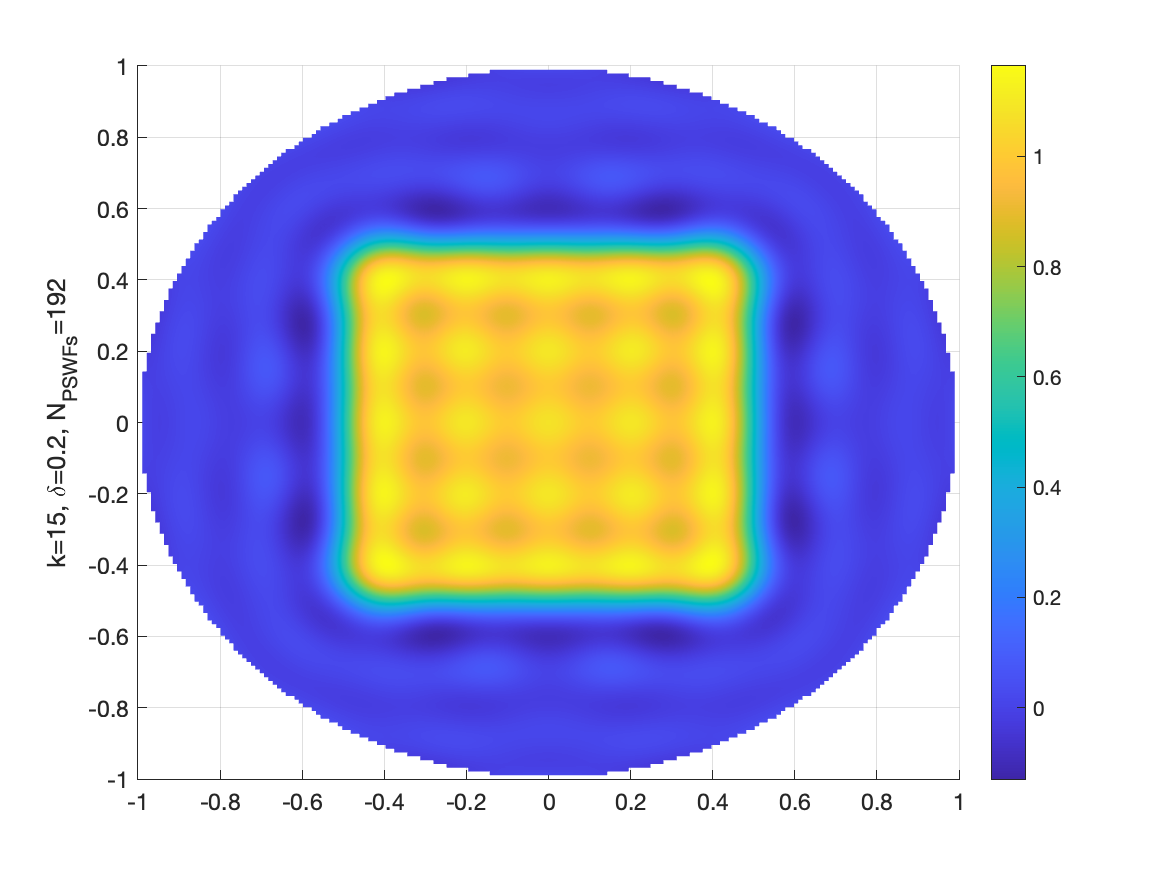}

  \caption{Reconstruction of the rectangle with the following dimensions of the low-rank space: $N_{\rm PSWFs}=1,~3,~5,~71,~73,~75,~134,~144,~192$. $k=15$ and $\delta=20\%$.}  \label{example: Different numbers of PSWFs}

\end{figure}

\subsubsection{Resolution}
We emphasize the case of three rectangles to demonstrate that our algorithm's capability in improved resolution. The three rectangles are given by
\begin{align*}
    \Omega_1&=\{(x_1, x_2)\in\mathbb{R}^2:-0.3 < x_1 <  -0.025,0.1 < x_2 <  0.3\},\\
    \Omega_2&=\{(x_1, x_2)\in\mathbb{R}^2:0.025 <  x_1 <  0.3,0.1 <  x_2 <  0.3\},\\
    \Omega_3&=\{(x_1, x_2)\in\mathbb{R}^2:-0.1 <  x_1 <  0.1,-0.2 <  x_2 <  0.025\},
\end{align*}
which gives the smallest distance between the rectangles  $0.05$.
The reconstruction in \Cref{example: U noisy super resolution} is for noisy data with $20\%$ noise level and $c=2k = 30$. The distance is less than half wavelength $\frac{\pi}{k}\approx 0.2$ and these objects can be clearly distinguished from the reconstruction. We point out that a similar improved resolution was also reported in \cite{novikov22a}.

\begin{figure}[htbp]
\centering

\includegraphics[width=0.45\linewidth]{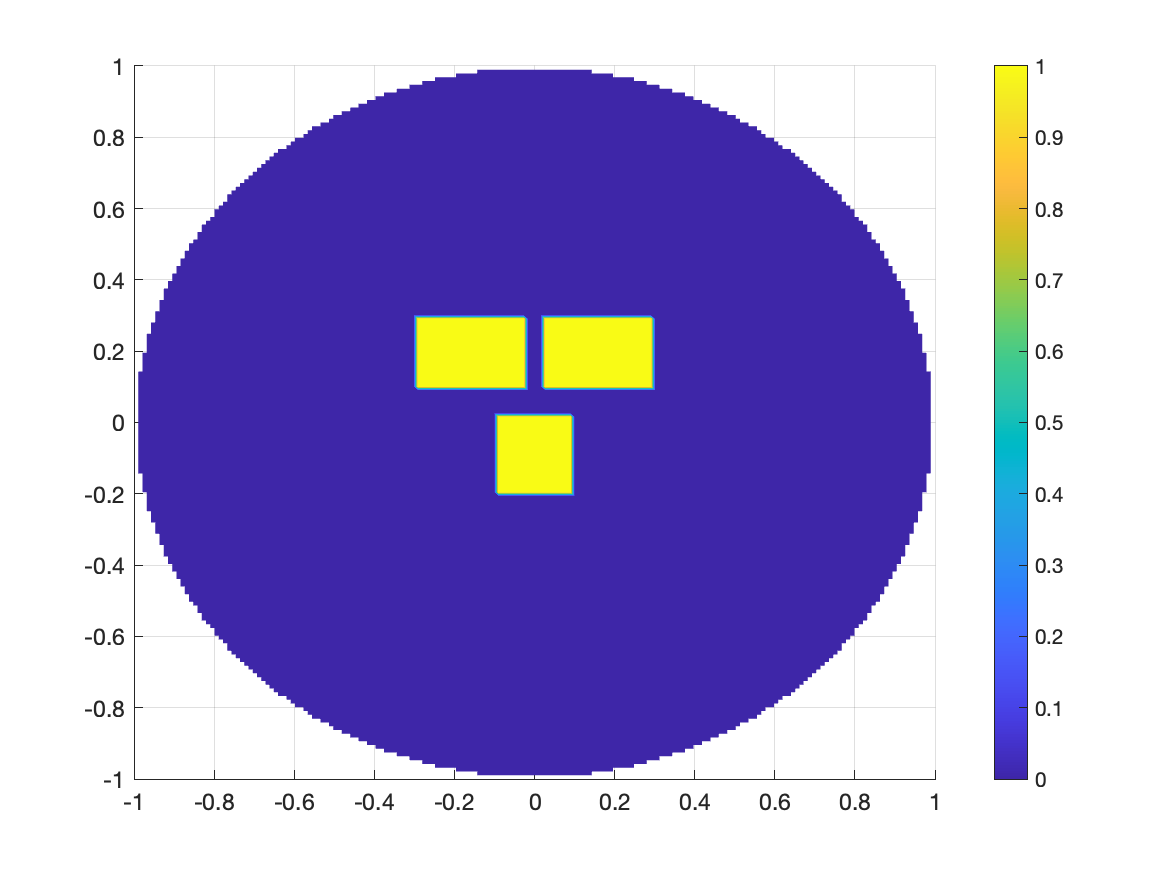}
\includegraphics[width=0.45\linewidth]{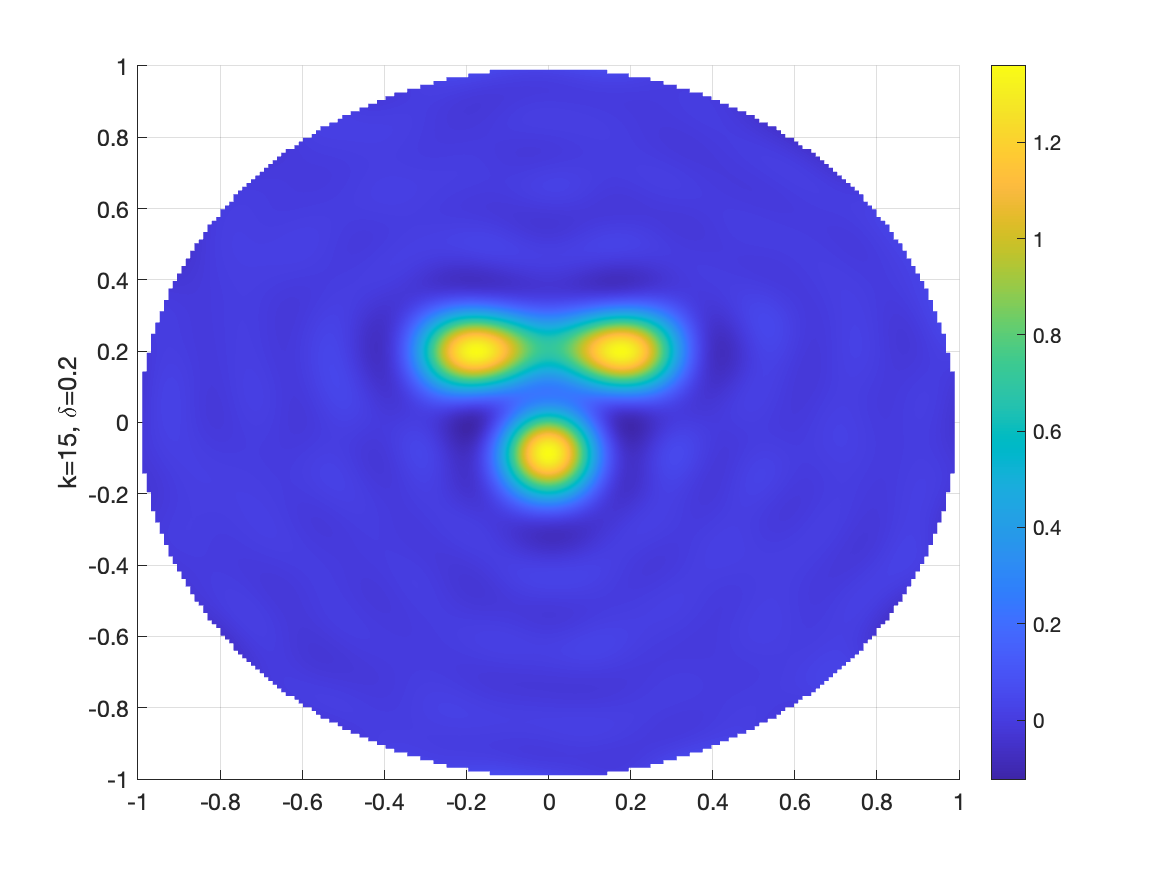}

  \caption{Reconstruction of three close rectangles. Left: exact contrast, right: reconstruction. $20\%$ noisy data.}
  \label{example: U noisy super resolution}
\end{figure}

\subsubsection{Far-field data}\label{far-field data }
Having successfully tested the algorithm using the post-processed data, in this section we implement the full algorithm of \Cref{Algorithm: lrInvScatt} with the Born far-field data  $\{u_b^{\infty}(\hat{x}_m;\hat{\theta}_\ell;k): ~m = 1,2,\dots, N_1, ~\ell = 1,2,\dots, N_2\}$. We follow the same rule as in \eqref{add noise} to add perturbed noise instead to the data $\{u_b^{\infty}(\hat{x}_m;\hat{\theta}_\ell;k): ~m = 1,2,\dots, N_1, ~\ell = 1,2,\dots, N_2\}$.

To illustrate the potential of our algorithm in dimensionality reduction, we work with far-field data of dimension  $N_1\times N_2 = 50\times 50, 100\times 100, 200\times 200, 500\times 500$. The test is done with $c = 30$.
Results in  \Cref{example: farfield noisy} show that our algorithm has the potential to reduce the dimensionality of large-scale data.

\begin{figure}[htbp]
\centering
\subfloat[$N_1=N_2 = 50$]{\includegraphics[width=0.45\linewidth]{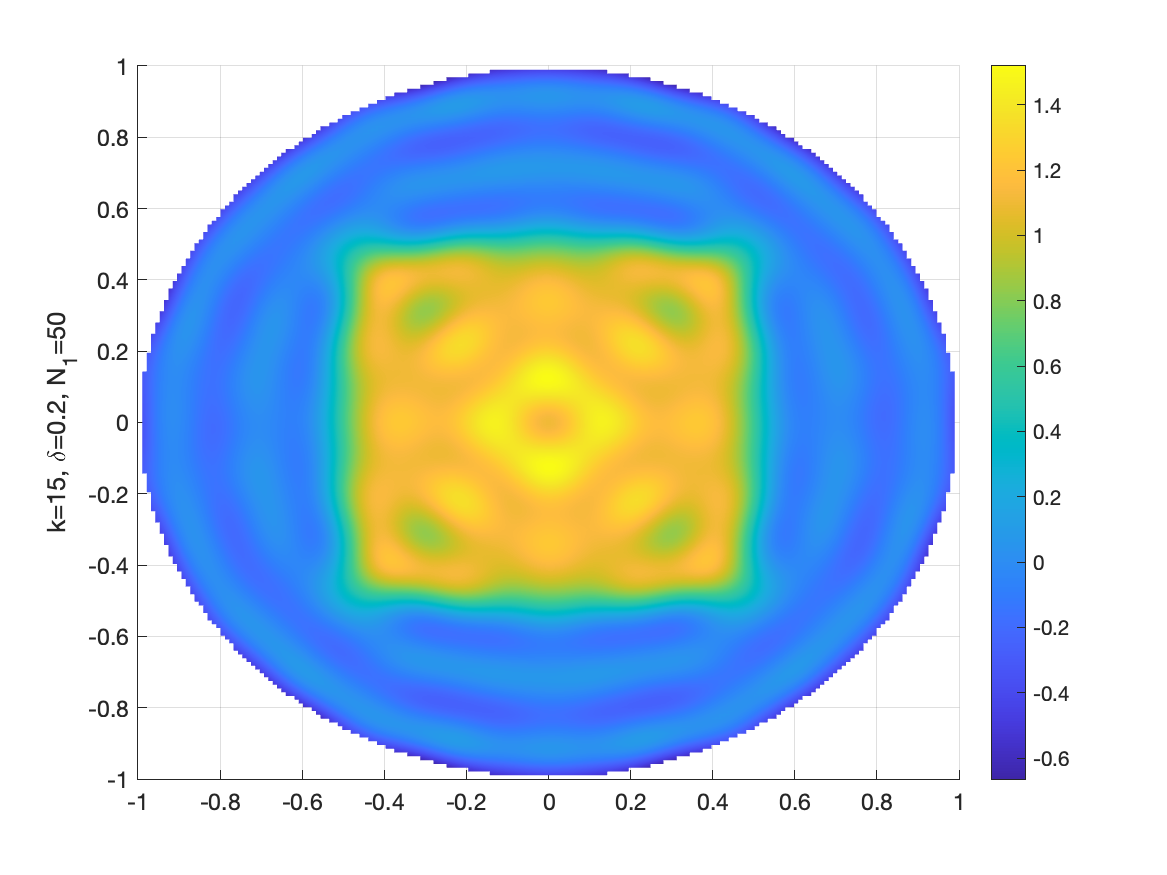}}
 \subfloat[$N_1=N_2 = 100$]
{\includegraphics[width=0.45\linewidth]{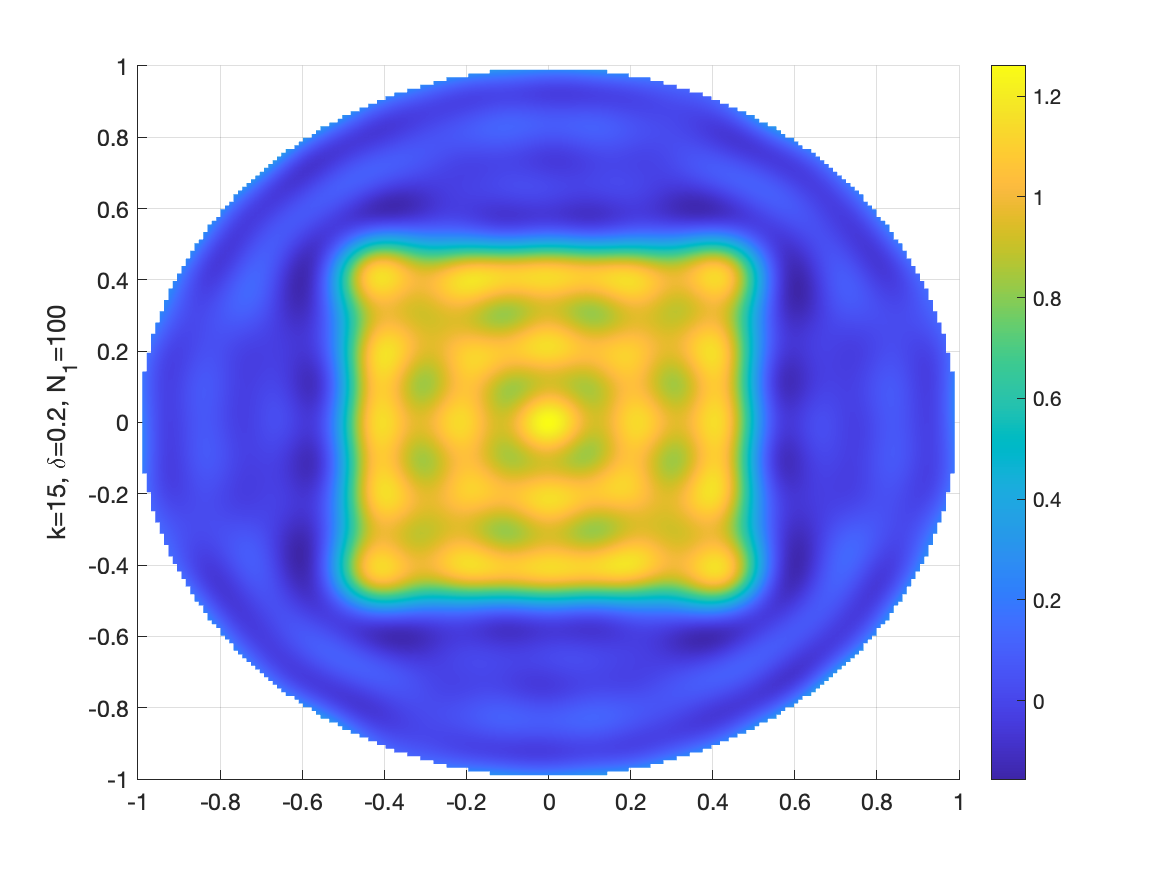}}
\quad
\subfloat[$N_1 = N_2 = 200$]
{\includegraphics[width=0.45\linewidth]{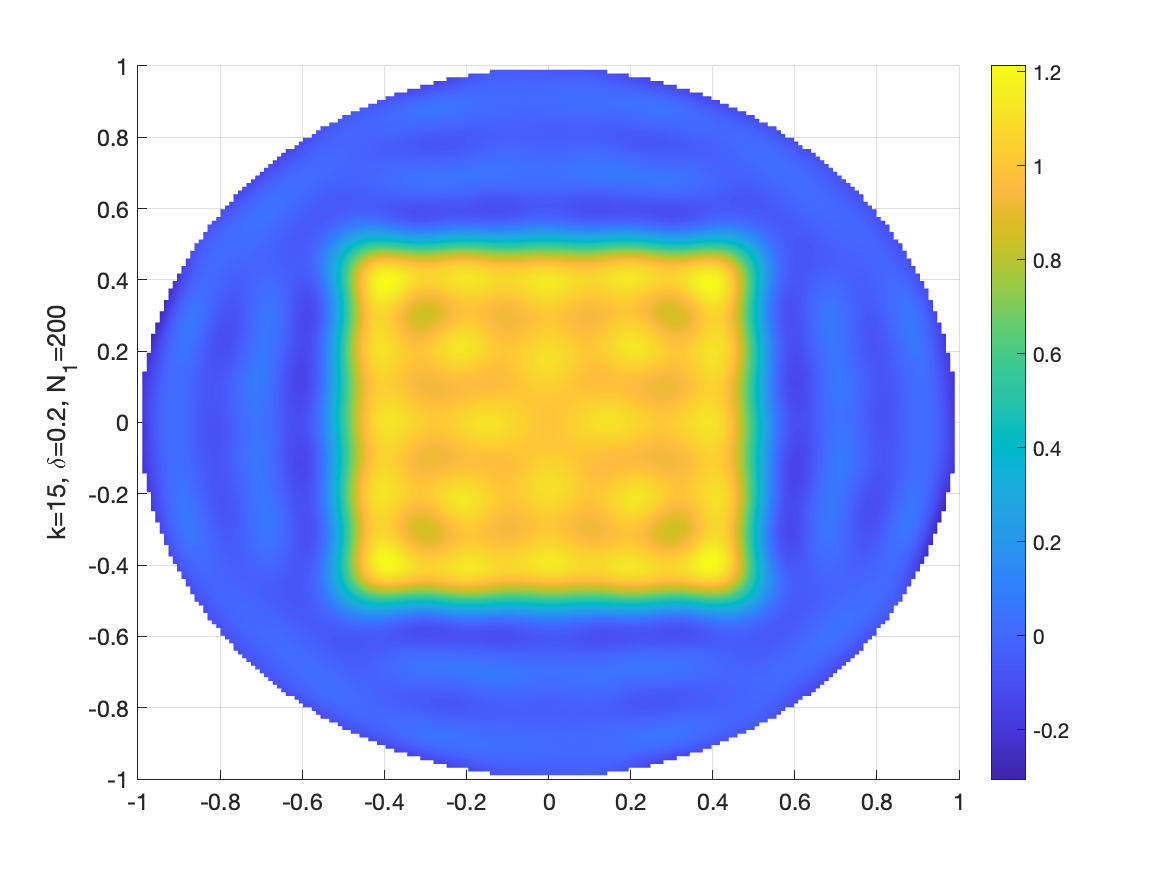}}
\subfloat[$N_1 =  N_2 = 500$]
{\includegraphics[width=0.45\linewidth]{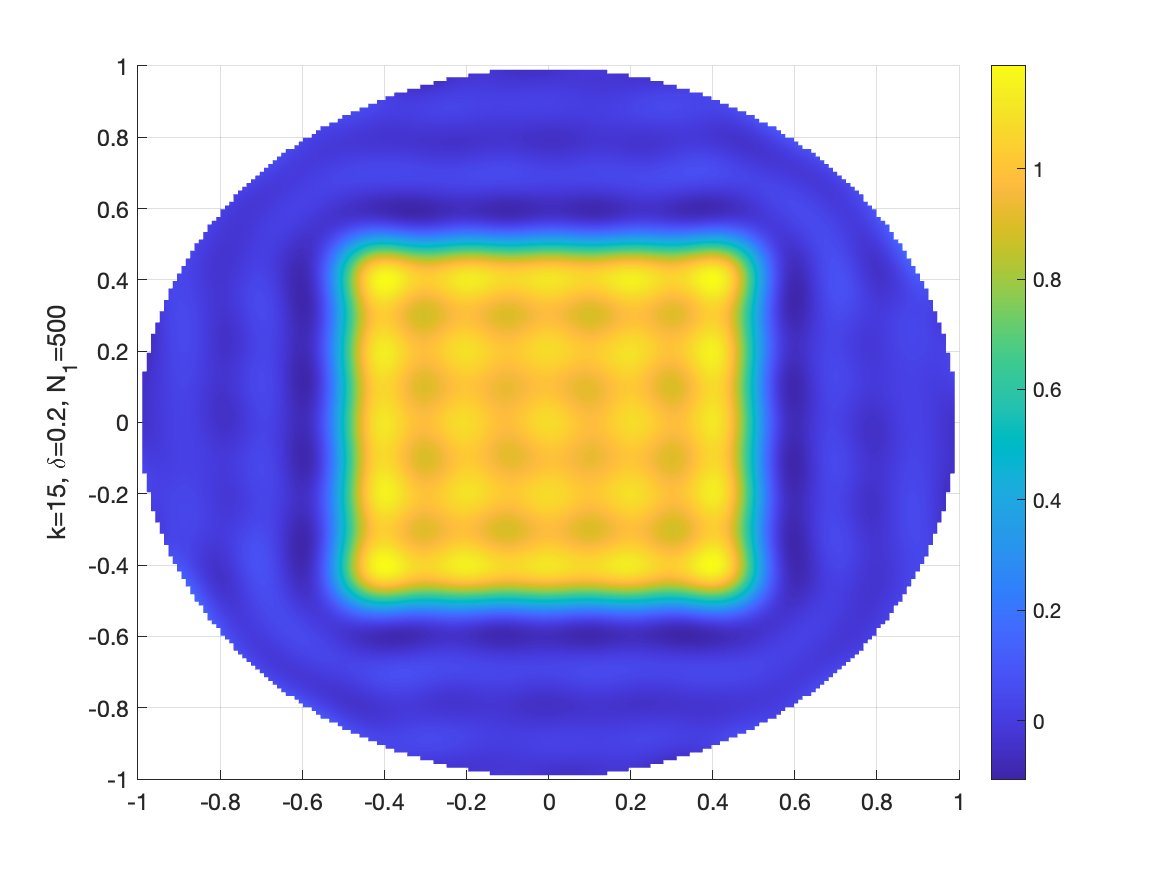}}

   \caption{Reconstruction of a rectangle with   $20\%$ noisy far-field data   of dimension  $N_1\times N_2 = 50\times 50, 100\times 100, 200\times 200, 500\times 500$. }    \label{example: farfield noisy}

\end{figure}

\subsubsection{Increasing stability} \label{section: increasing stability Born}

Another interesting aspect is the increasing stability (cf. \cite{HrycakIsakov04,SubbarayappaIsakov07}): the resolution becomes better as $c=2k$ gets larger. 
We first demonstrate this property by the Born far-field data and we will test this property using the far-field data obtained by the nonlinear model in Section \ref{section: increasing stability}.
\Cref{example: farfield increasing stability} plots the reconstruction of three rectangles using our method for $k=15,~20,~25,~30,~35,~40$. Clearly, we have observed the increasing stability. 
{  Moreover, in   \Cref{example: near boundary} we test the performance when there are three rectangles close to the boundary of the unit disk, it is observed that both the supports and the parameters of these three rectangles can be reconstructed well. }

\begin{figure}[htbp]
\centering
\subfloat[$k=15$]
{
\includegraphics[width=0.45\linewidth]{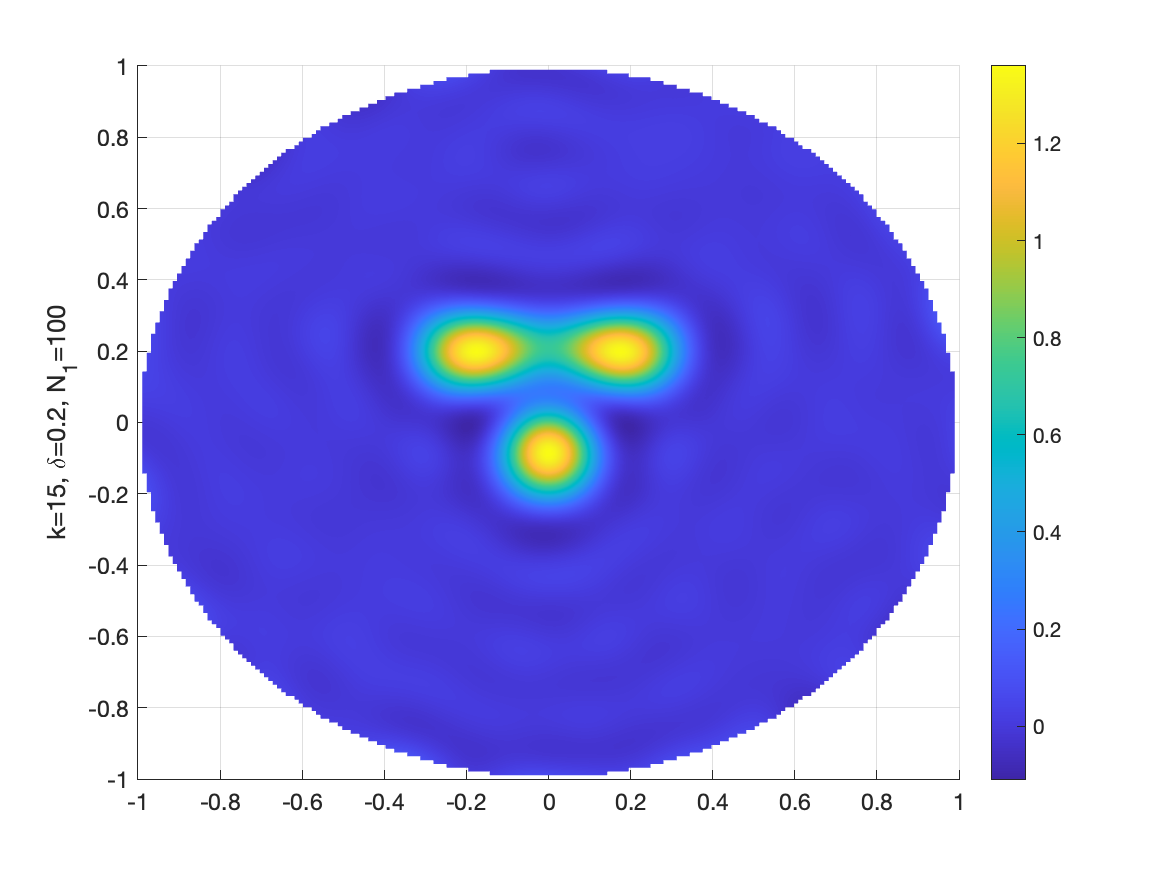}
}
\subfloat[$k=20$]
{
\includegraphics[width=0.45\linewidth]{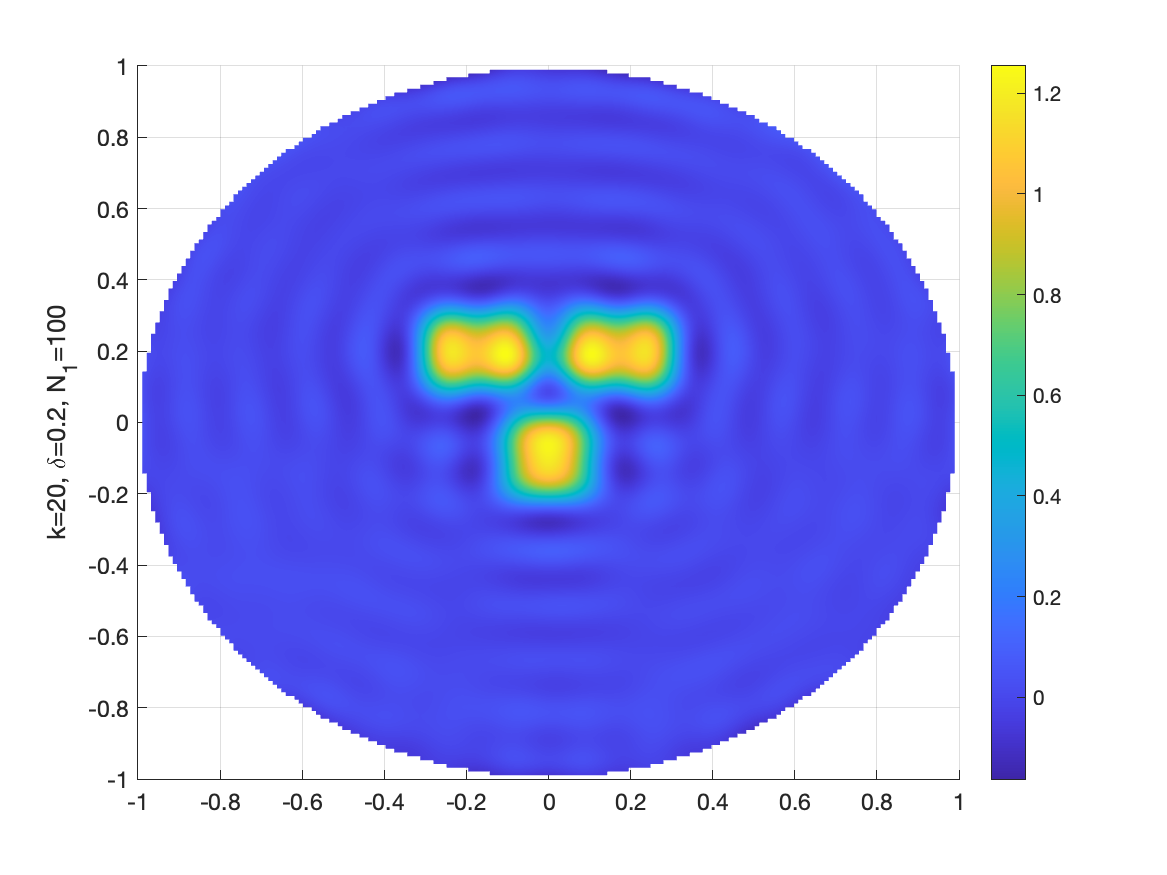}
}
\quad
\subfloat[$k=25$]
{
\includegraphics[width=0.45\linewidth]{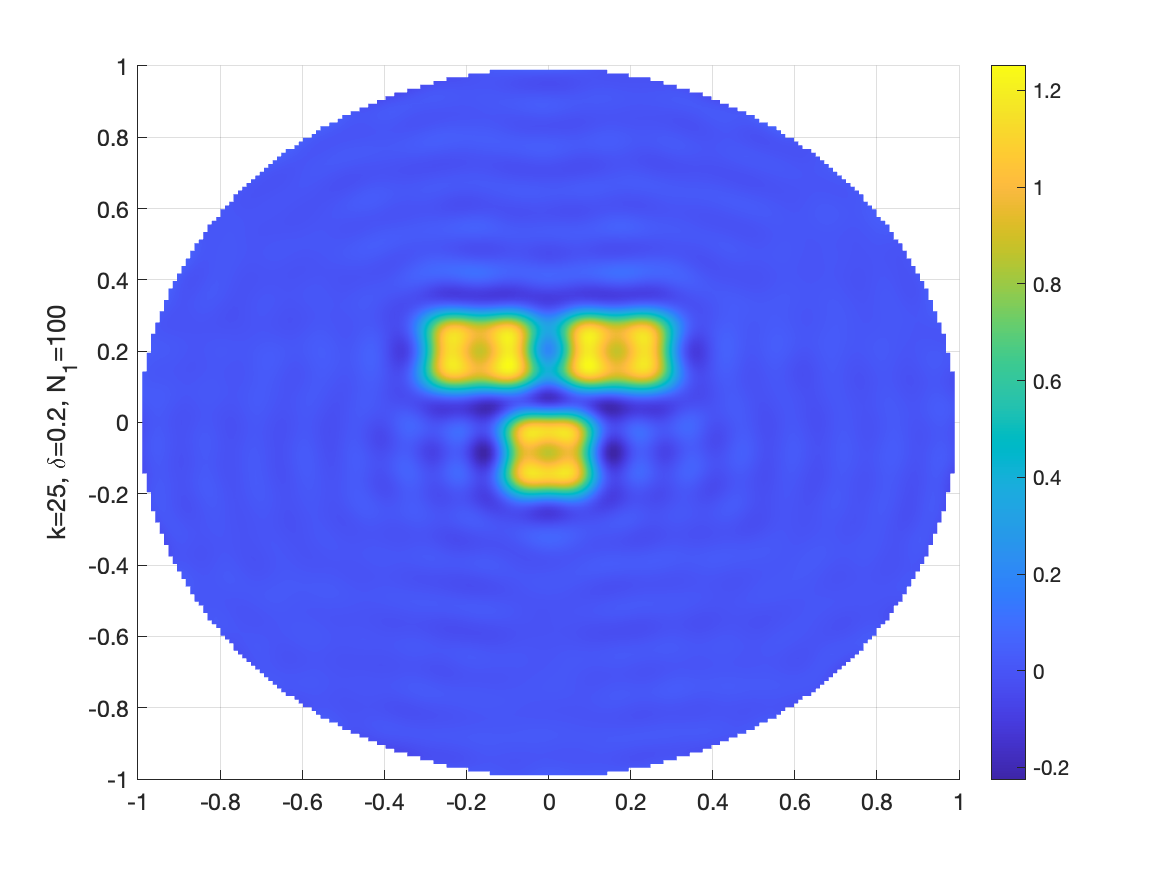}
}
\subfloat[$k=30$]
{
\includegraphics[width=0.45\linewidth]{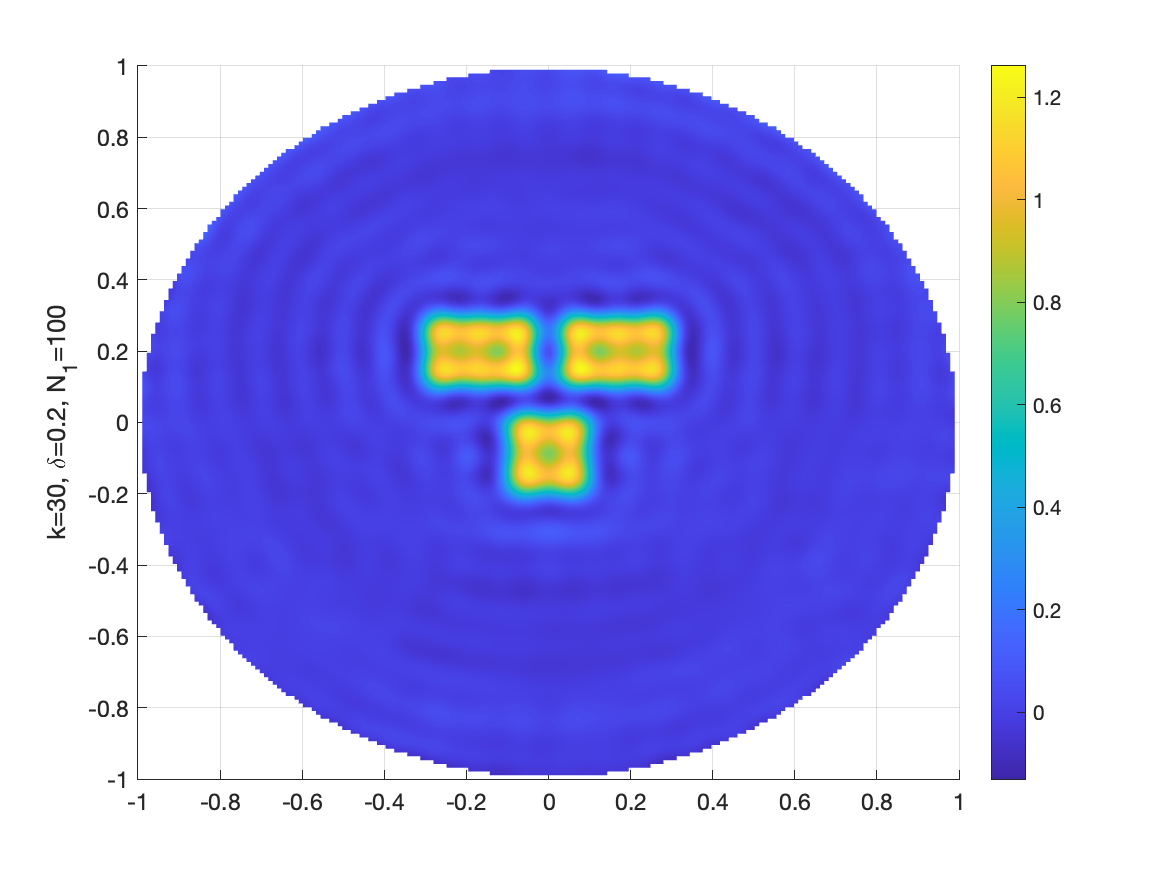}
}
\quad
\subfloat[$k=35$]
{
\includegraphics[width=0.45\linewidth]{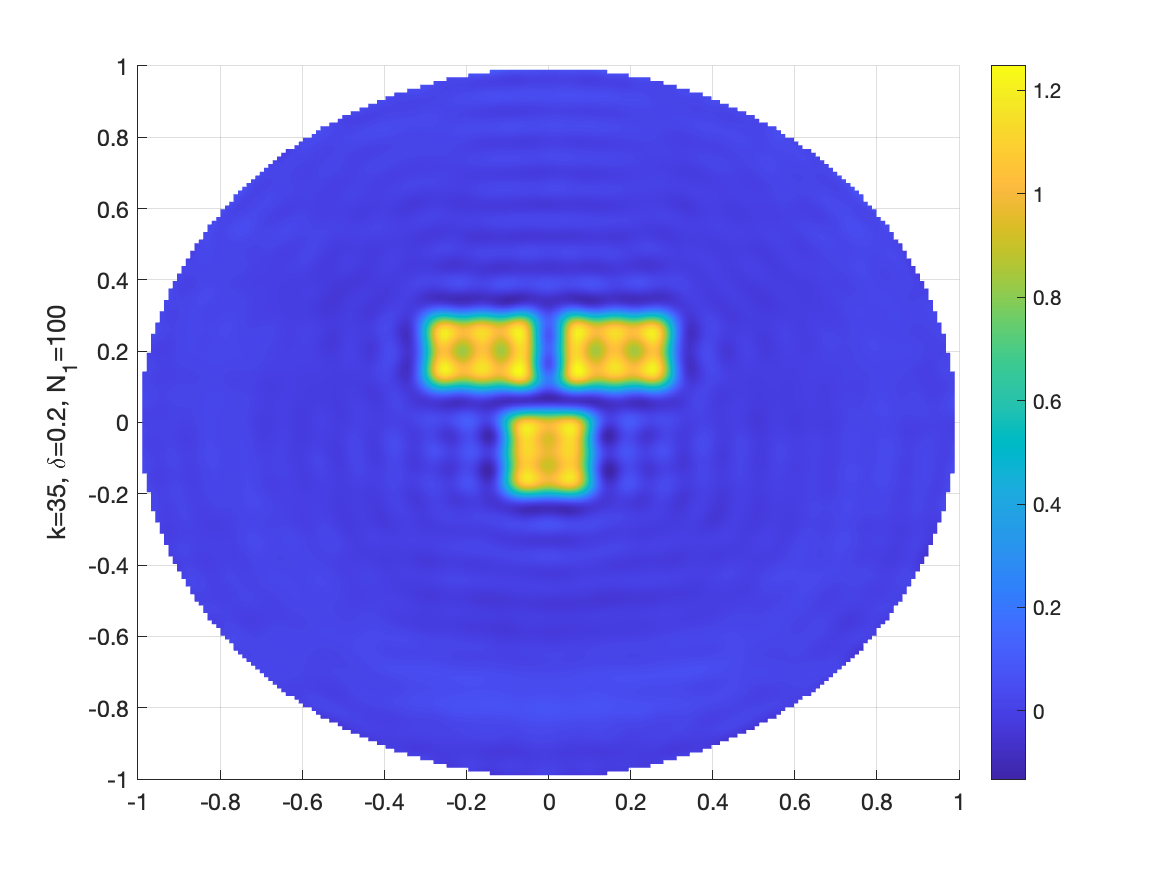}
}
\subfloat[$k=40$]
{
\includegraphics[width=0.45\linewidth]{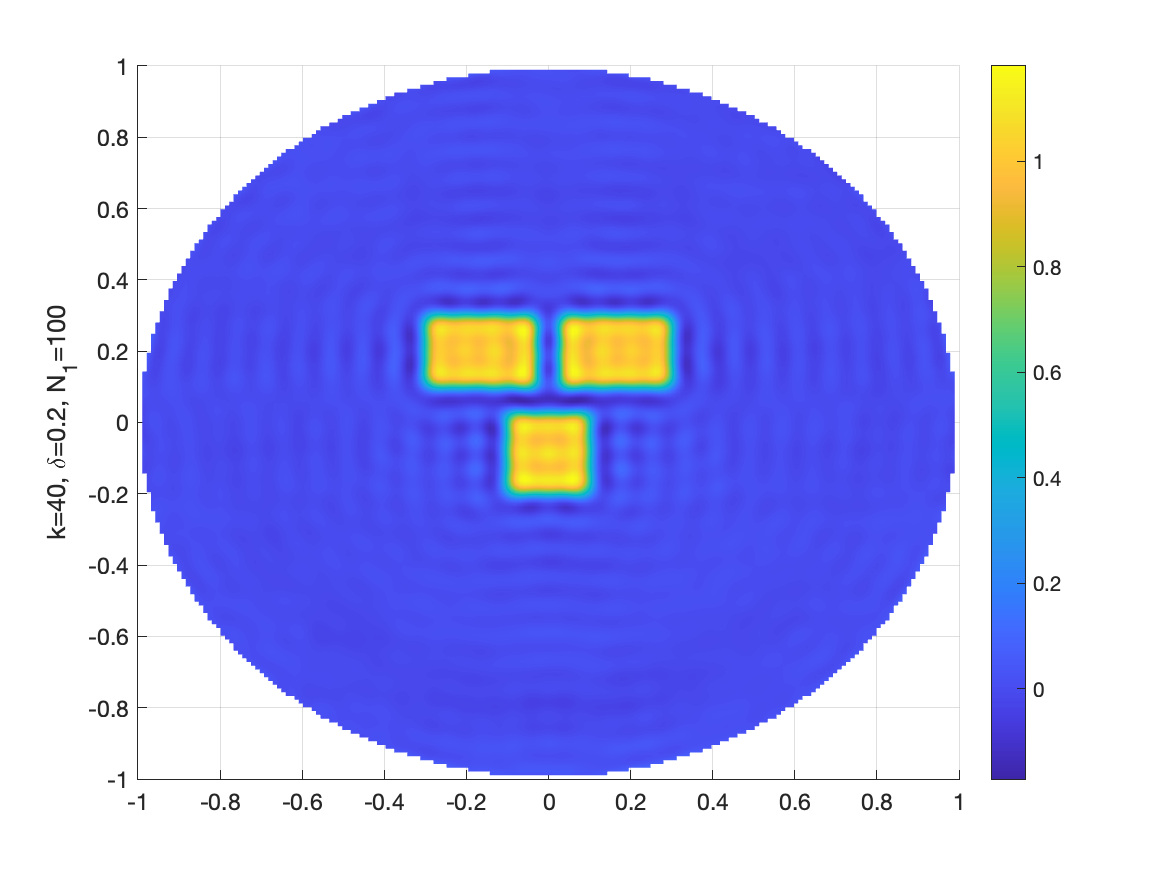}
}
  \caption{Reconstruction of three rectangles using our method for $k=15,~20,~25,~30,~35,~40$ using $20\%$ noisy data. }    \label{example: farfield increasing stability}

\end{figure}

\begin{figure}[htbp]
\centering
\subfloat[Ground truth]
{
\includegraphics[width=0.45\linewidth]{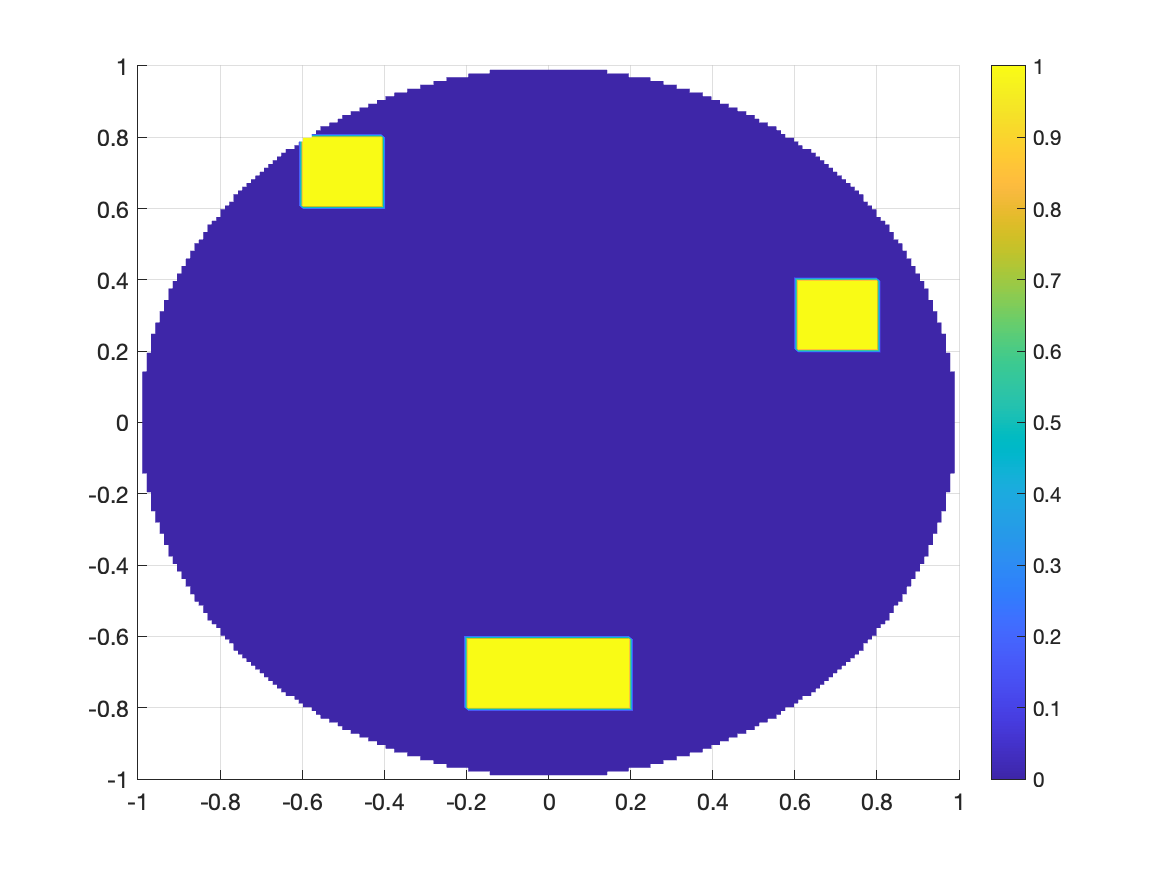}
}
\subfloat[$k=15$]
{
\includegraphics[width=0.45\linewidth]{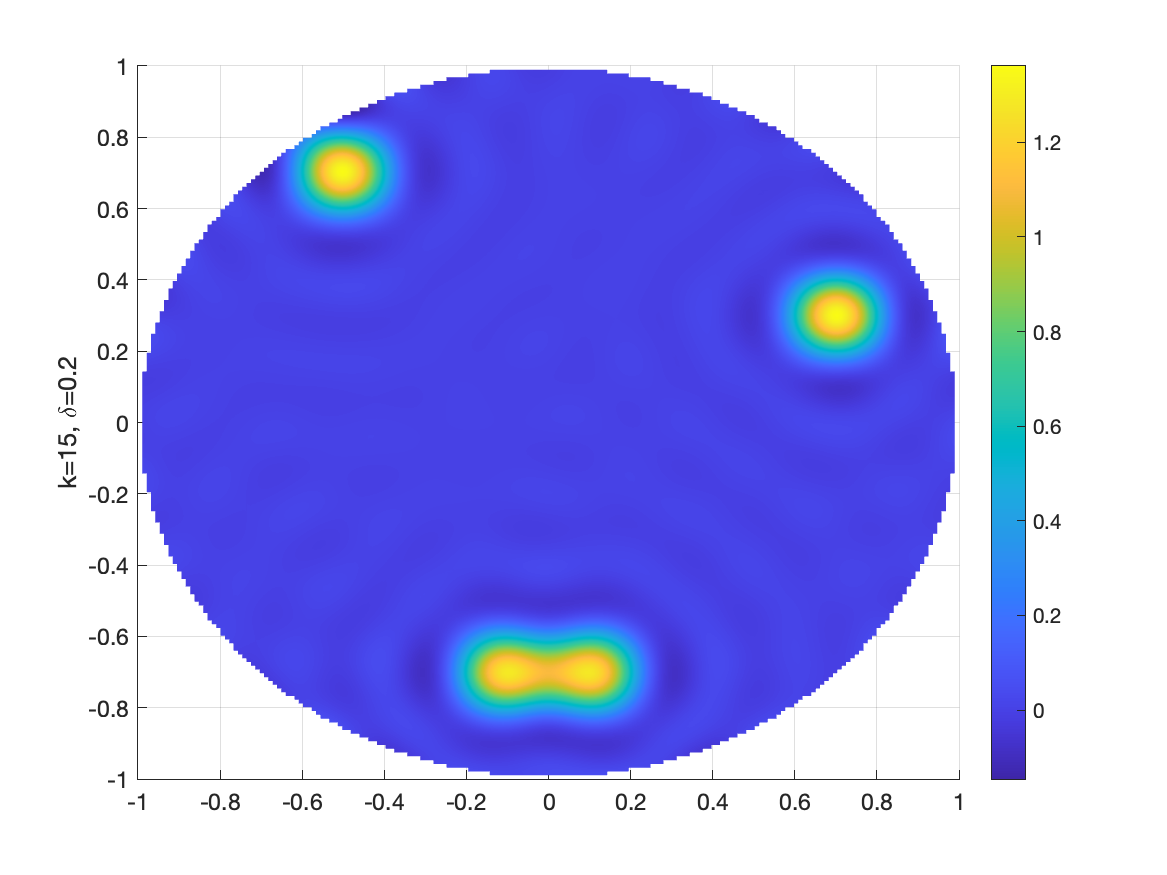}
}
\quad
\subfloat[$k=20$]
{
\includegraphics[width=0.45\linewidth]{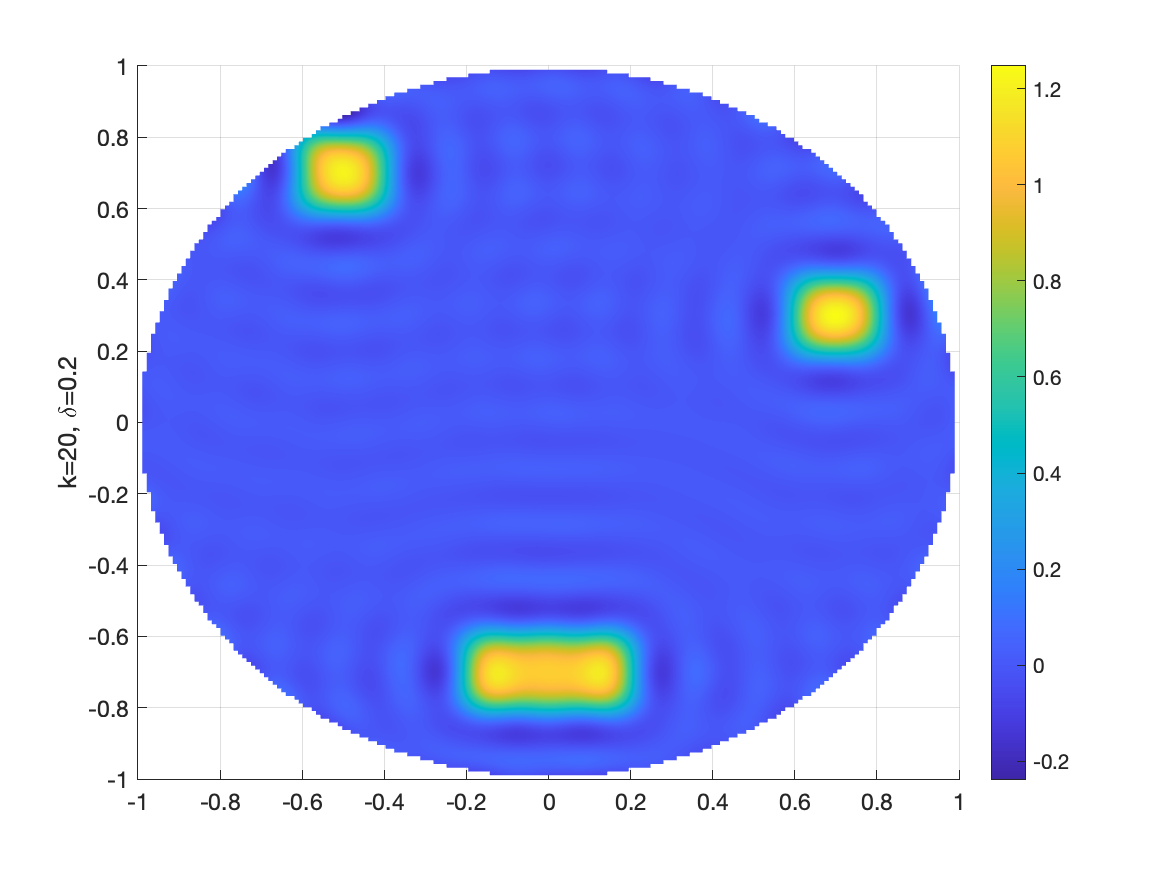}
}
\subfloat[$k=25$]
{
\includegraphics[width=0.45\linewidth]{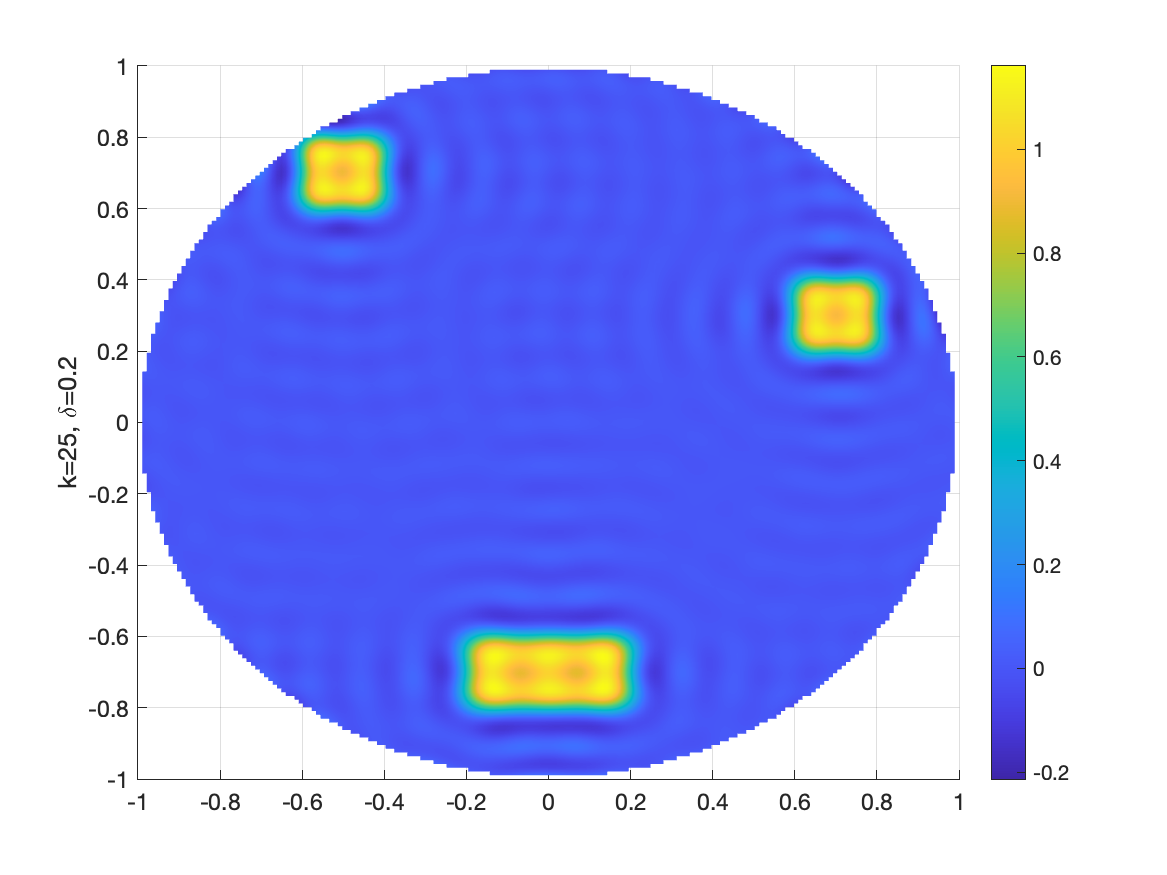}
}
\quad
\subfloat[$k=30$]
{
\includegraphics[width=0.45\linewidth]{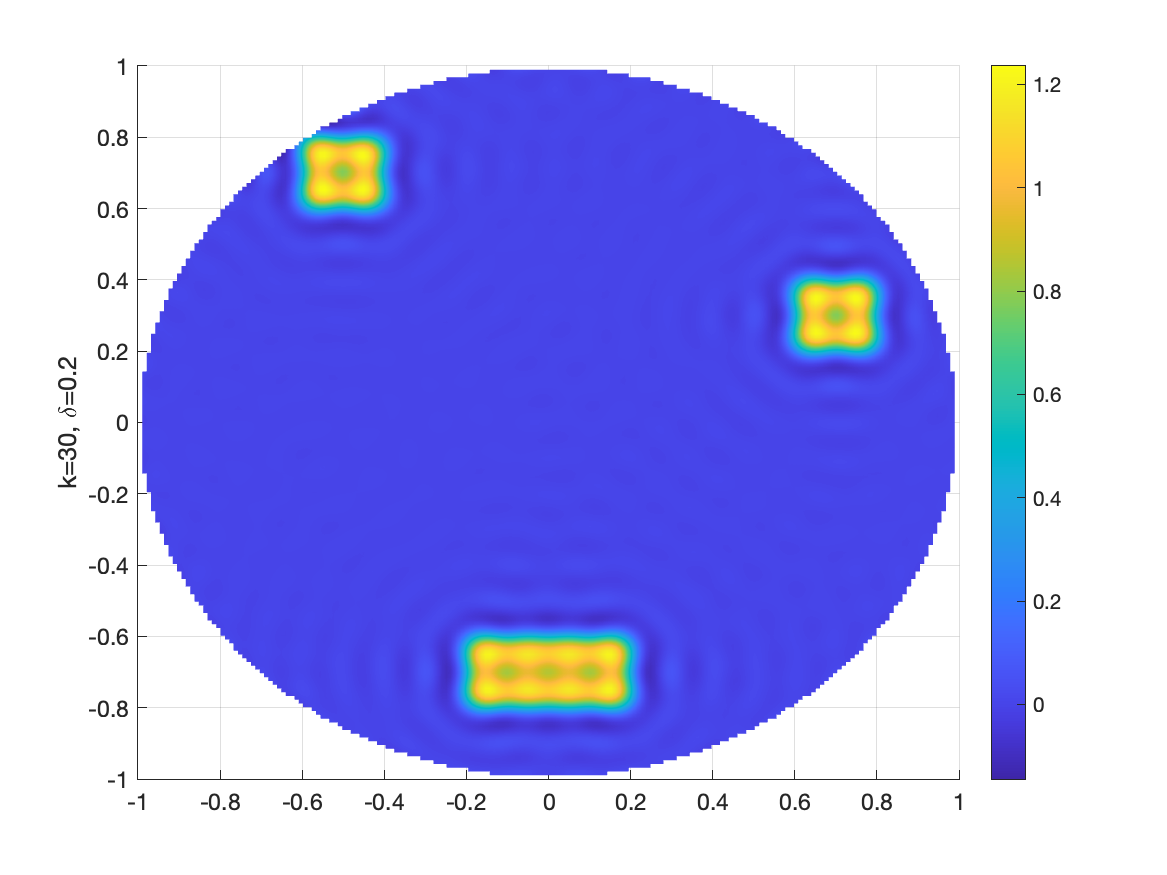}
}
\subfloat[$k=40$]
{
\includegraphics[width=0.45\linewidth]{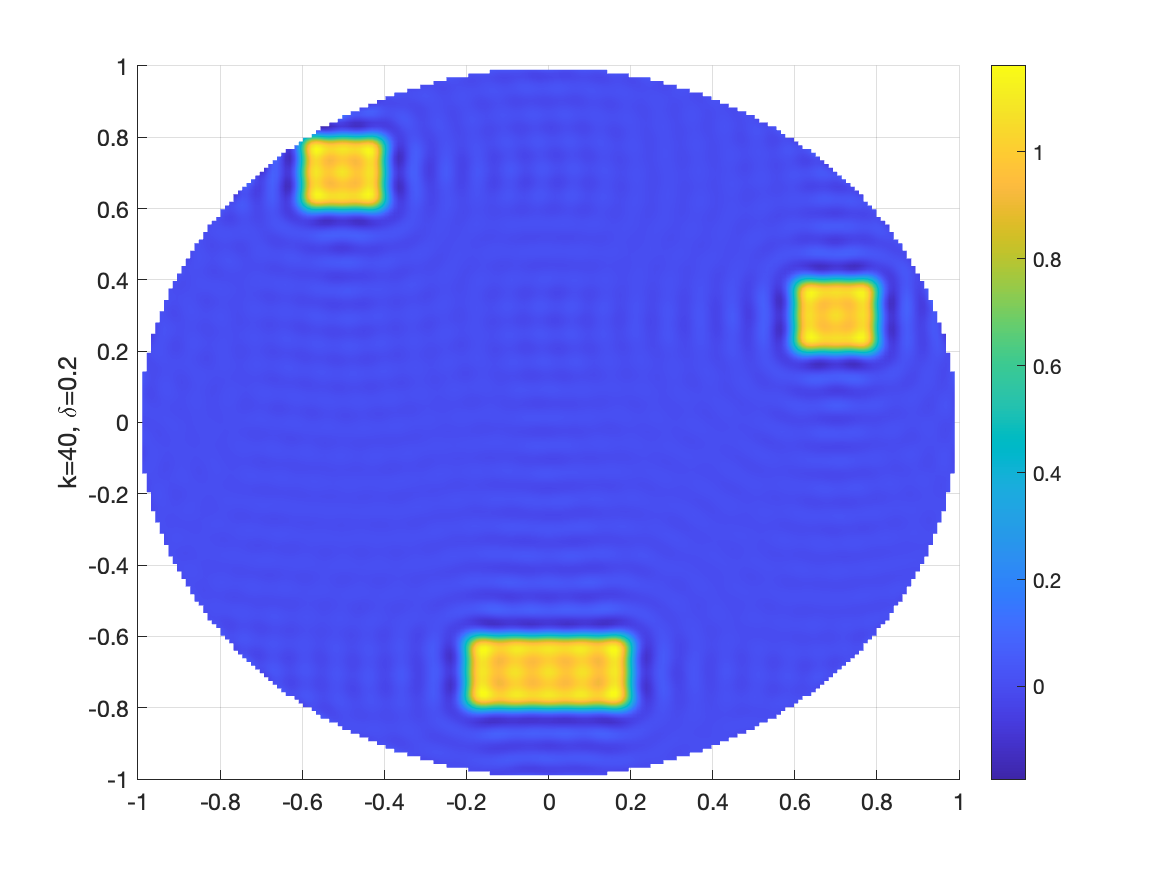}
}
  \caption{Reconstruction of three rectangles near the boundary using our method for $k=15,~20,~25,~30,~40$ using $20\%$ noisy data. }    \label{example: near boundary}

\end{figure}

\subsubsection{Comparison with Matlab built-in fft}
In this section, we compare our algorithm with the Matlab built-in function fft. The motivation is that the Born data is related to the restricted Fourier transform of the unknown $q$, therefore it makes sense to consider solving the unknown $q$ by the fast Fourier transform. To apply the Matlab built-in function fft, {  one needs to extend the post-processed data in the disk to a unit square by zero}; in addition, the number of  sampling points in the unit square $m$ is connected with wave number $k$, i.e., $m=\frac{4k}{\pi}$, where $k$ needs to be a multiple of $\pi$. We point the readers to the fractional Fourier transform  \cite{inverarity02} which overcame this issue but we only consider the Matlab built-in fft.

For both methods, we use noisy Born far-field data $\{u^{\infty,\delta}_b(\hat{x}_i;\hat{\theta}_j;k)\}_{i=1,j=1}^{N_1,N_2}$ with   $N_1=N_2={ 100}$, and noise level $\delta=20\%$.
For { the method based on the low-rank structure}, we directly apply \Cref{Algorithm: lrInvScatt}. For the Matlab fft algorithm, exact nodes are given by $\{((i-\frac{m}{2})\frac{2}{m},(j-\frac{m}{2})\frac{2}{m})\}_{i=0,j=0}^{m-1}$ where $m=\frac{4k}{\pi}$ is an even number; following the idea of \Cref{Algorithm: lrData}, we also find the approximate or mock-quadrature nodes that are approximations to the fft exact nodes.
 { 
It is observed from \Cref{FFT} that our proposed method has less blurring effects since it utilizes the post-processed data directly, as apposed to the fft whose blurring effect is possibly due to extending the post-processed data from the unit disk to the unit square by zero.}

\subsection{{ Full model}}
In practical application, full far-field  data $u^\infty(\hat{x};\hat{\theta};k)$ are measured rather than Born far-field data $u^\infty_b(\hat{x};\hat{\theta};k)$. This motivates us to test our algorithm with the full far-field data  $\{u^\infty(\hat{x}_i;\hat{\theta}_j;k)\}$ to understand the limitations and potentials.

\subsubsection{Data generation}

\begin{figure}[htbp]
\centering
  \subfloat[Proposed Method: $k=4\pi$.]
  {
      \includegraphics[width=0.45\linewidth]{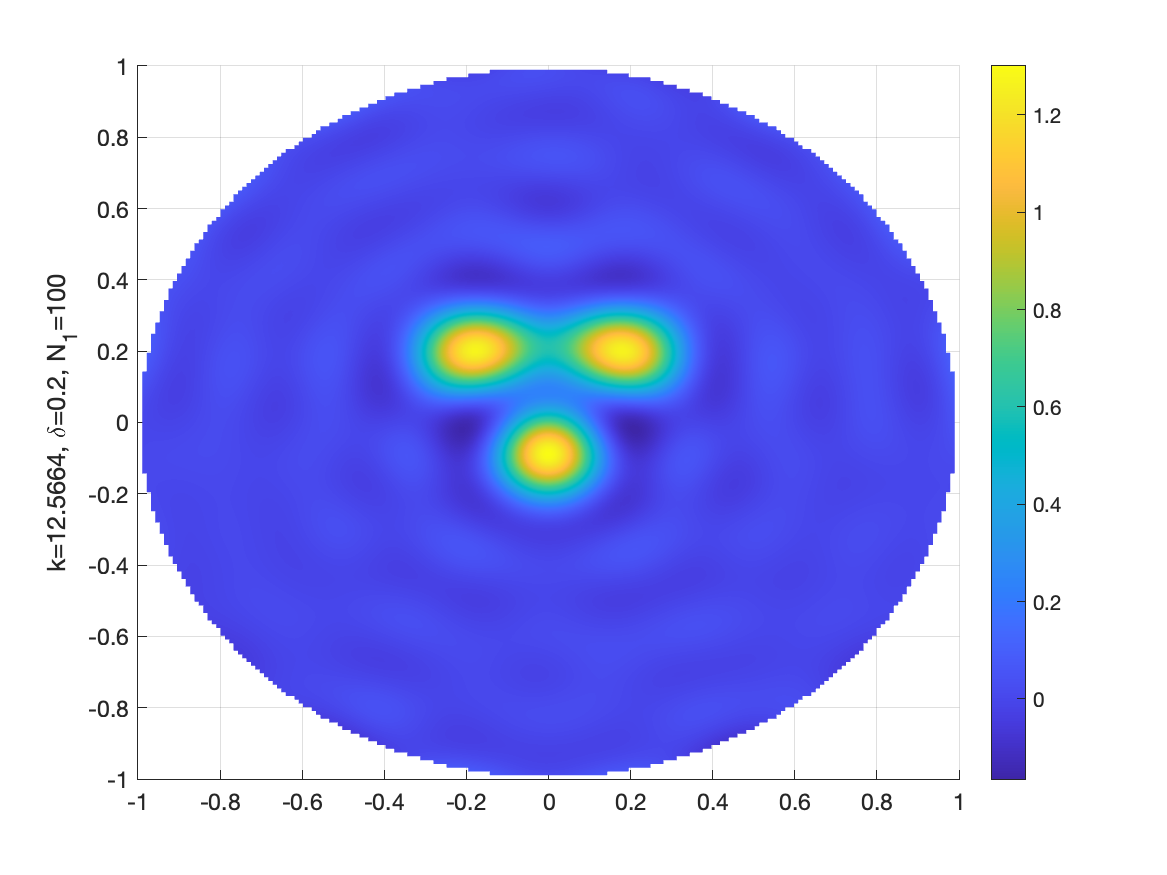}
  }
 \subfloat[Matlab Built-in fft: $k=4\pi$.]
  {
      \includegraphics[width=0.45\linewidth]{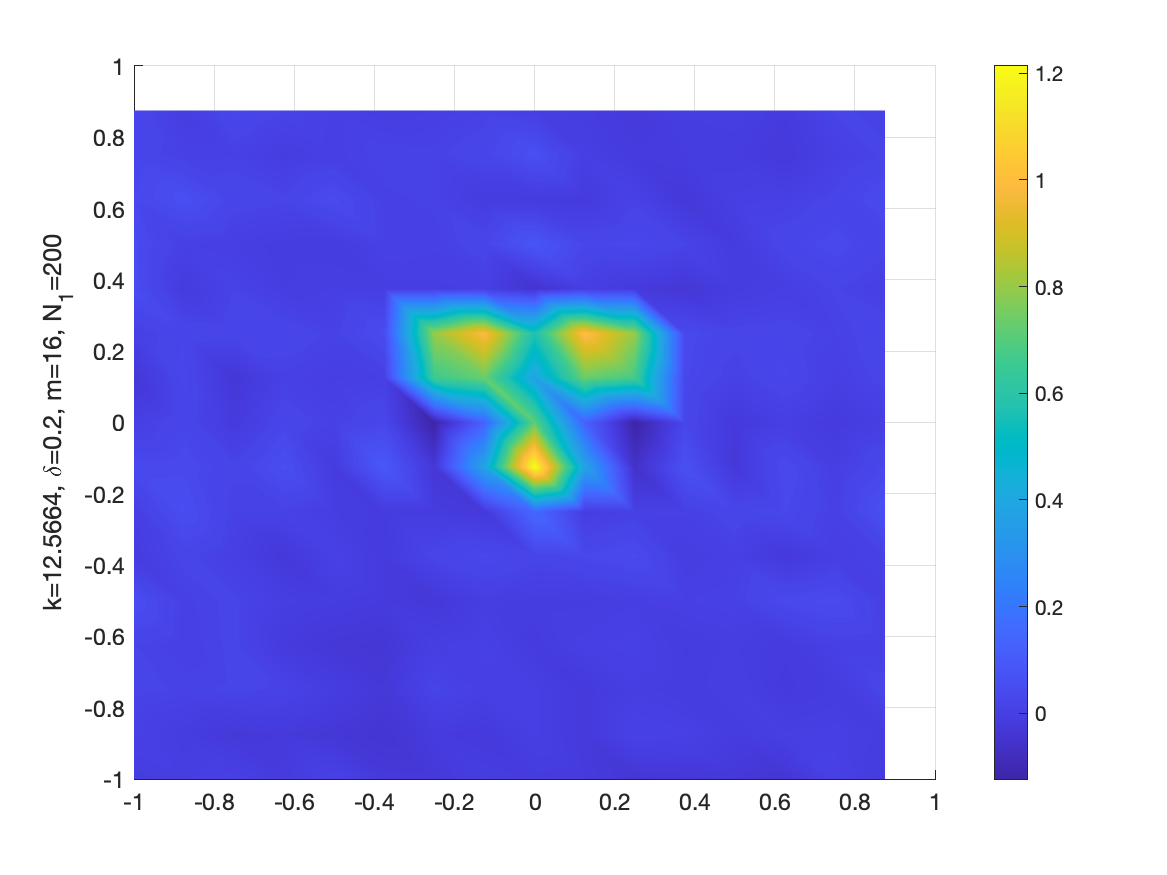}
  }\qquad
 \subfloat[Proposed Method:   $k=10\pi$.]
  {
      \includegraphics[width=0.45\linewidth]{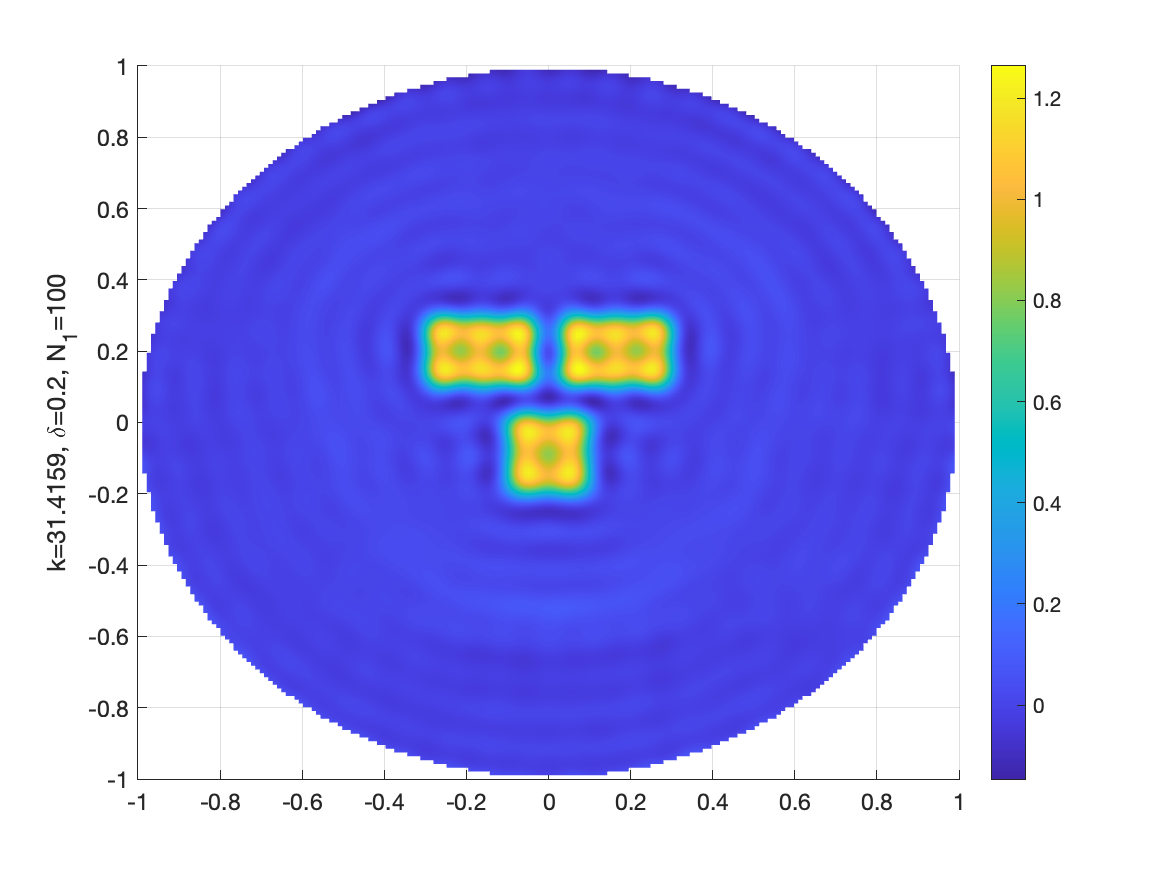}
  }
 \subfloat[Matlab Built-in fft: $k=10\pi$.]
  {
      \includegraphics[width=0.45\linewidth]{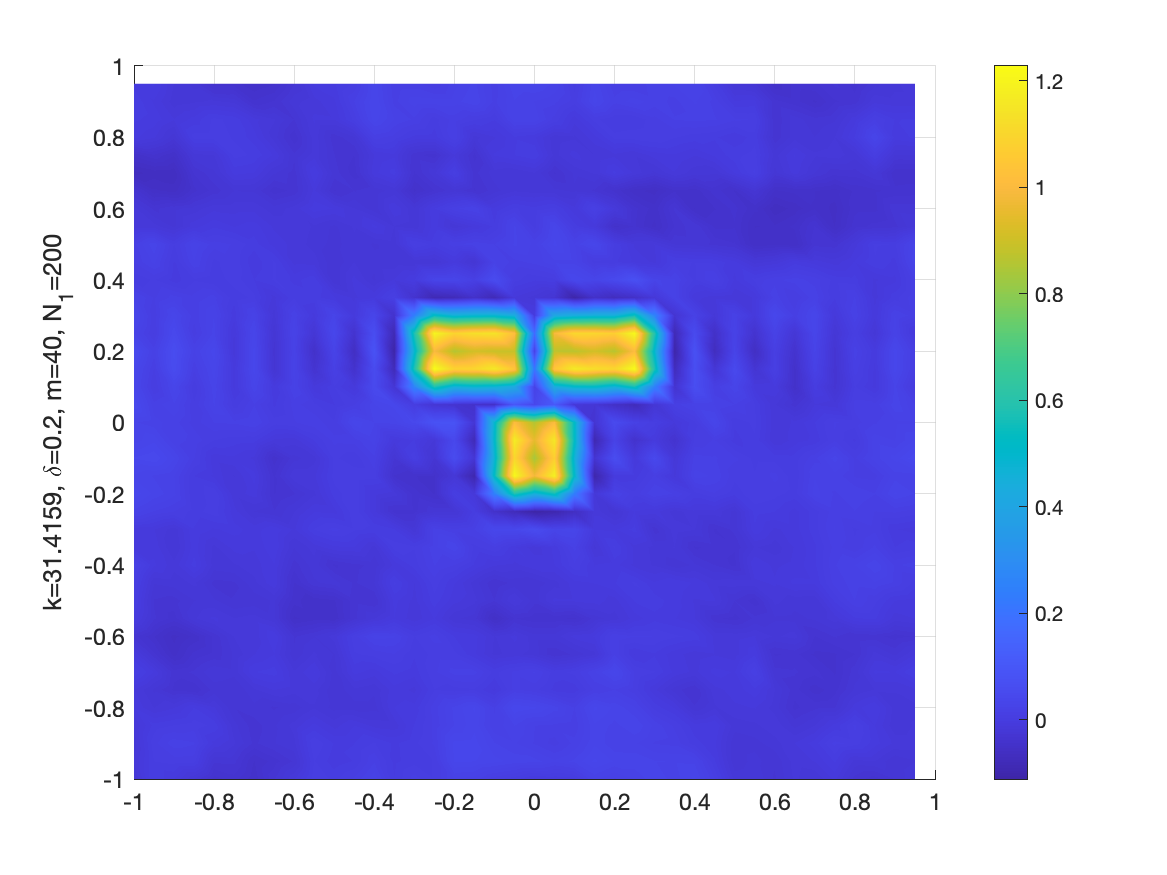}
  }
 \caption{Reconstruction of three rectangles. Noise level $20\%$.
 }
\label{FFT}
\end{figure}

For the full model, we use IPscatt \cite{Burgel19} to generate the exact far-field data $\{u^\infty(\hat{x}_i;\hat{\theta}_j;k)\}_{i=1,j=1}^{N}$, where  $N$ denotes the number of incident directions and observation directions.  The noisy data $\{u^{\infty,\delta}(\hat{x}_i;\hat{\theta}_j;k)\}_{i=1,j=1}^{N}$ are generated by  adding  noise with noise level $\delta$  to the simulated data $\{u^\infty(\hat{x}_i;\hat{\theta}_j;k)\}_{i=1,j=1}^{N}$  
following  \eqref{add noise}. According to \cite{Burgel19}, the relative error between the  Born model  and the full model  is given by
\(
    {\rm rel(k)}=\frac{\|{U}_b^s-{U}^s \|_2}{\|{U}^s\|_2}\),
where ${U}_b^s=(u_{b,j}^s(x_i))_{\rm IncNb\times Np}, {U}^s=(u_{j}^s(x_i))_{\rm IncNb\times Np}$ are two matrices: here $Np$ is the number of nodes in a fixed domain $\tilde{\Omega} = (- \sqrt{2}/2,\sqrt{2}/2)\times (-\sqrt{2}/2,\sqrt{2}/2)$ (here we  follow \cite{Burgel19} to choose $\tilde{\Omega}$)  and ${\rm IncNb}$ is the number of incident directions;  the subindex $j$ represents that the scattered wave is due to an incident wave at direction $\hat{\theta}_j$.

Note that the simulation of the far-field data of \cite{Burgel19} makes no use of disk PSWFs, therefore the possibility of inverse crimes is avoided. {In the following experiments, the spectral cutoff is set as $\epsilon=0.9 |\alpha_{0,0}(c)|$ to incorporate  that no a prior information on the relative modeling error is known.}

 \begin{figure}[htp]
\centering
  \subfloat[Ground-truth]
  {
      \includegraphics[width=0.45\linewidth]{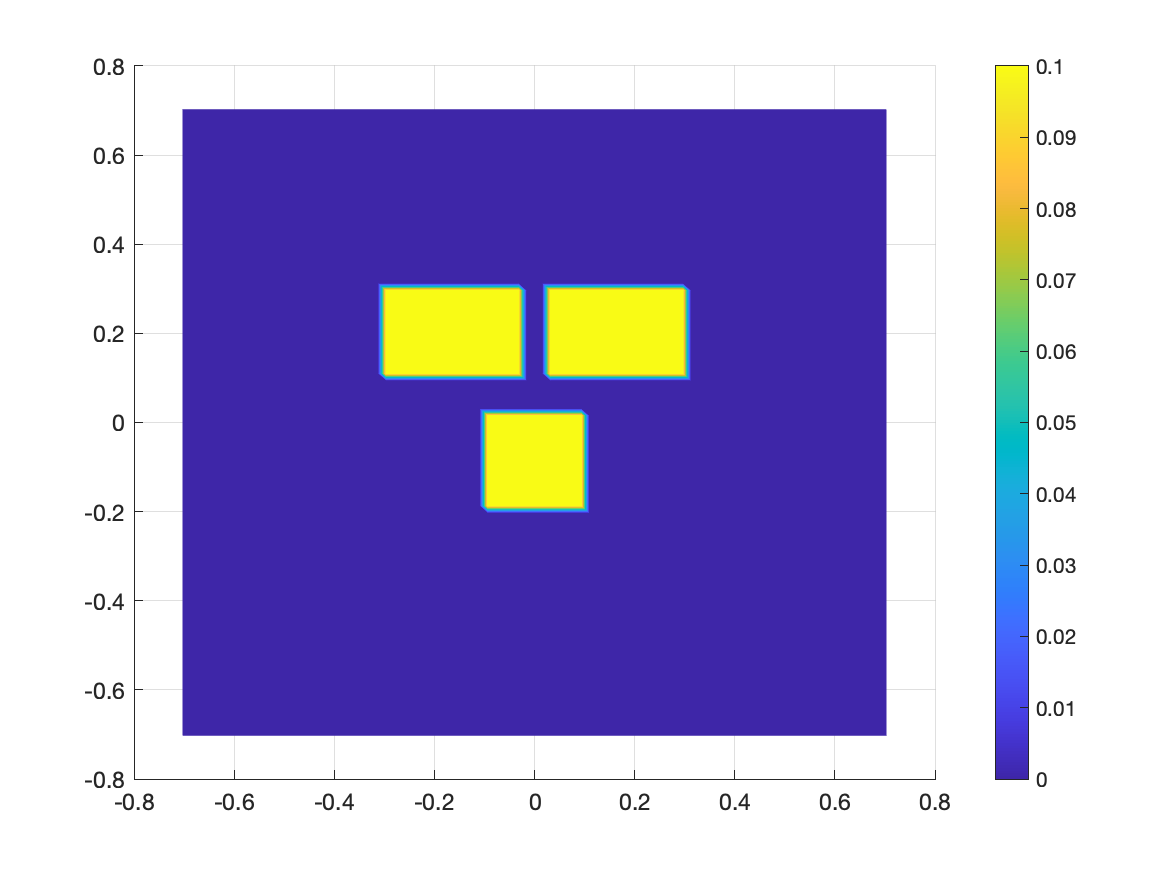}
  }
 \subfloat[k=15,  ${\rm rel(k)} = 0.17138$]
  {
      \includegraphics[width=0.45\linewidth]{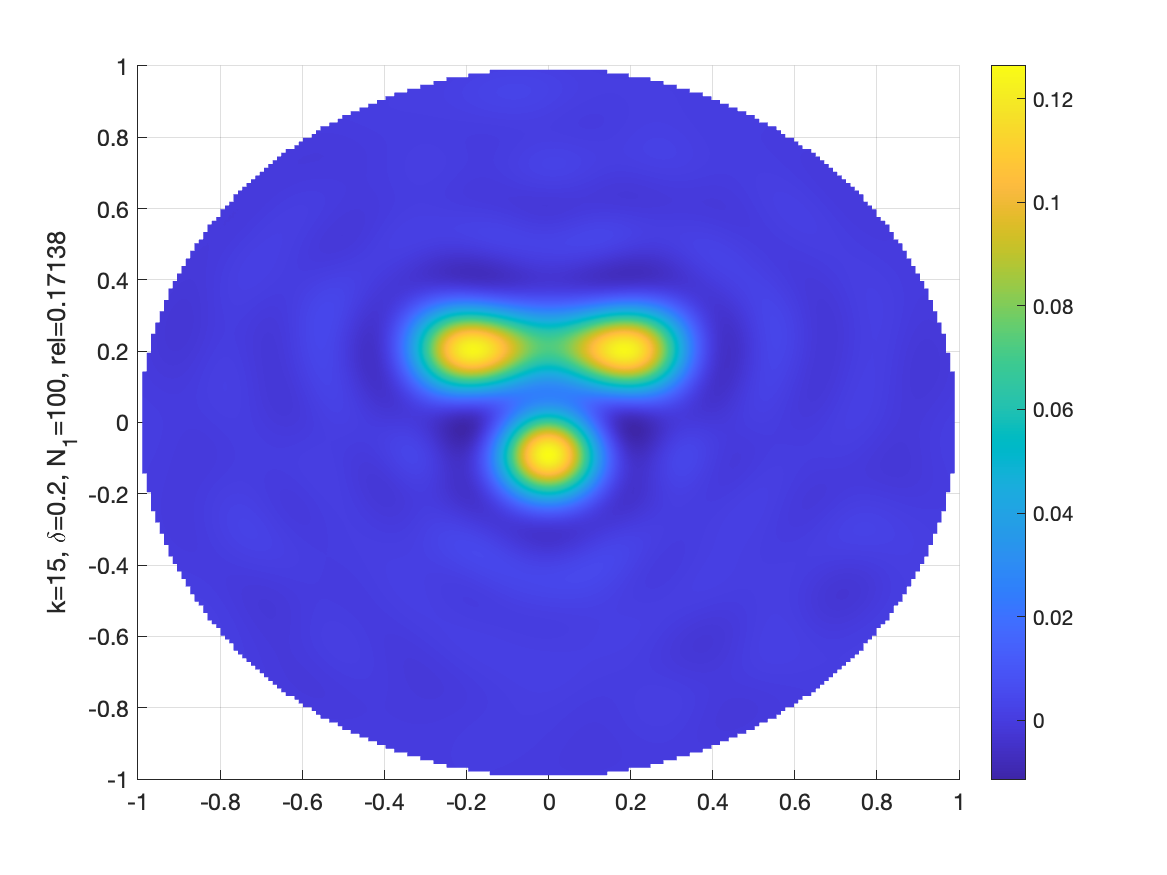}
  } \\
 \subfloat[k=30,  ${\rm rel(k)} = 0.39434$]
  {
      \includegraphics[width=0.45\linewidth]{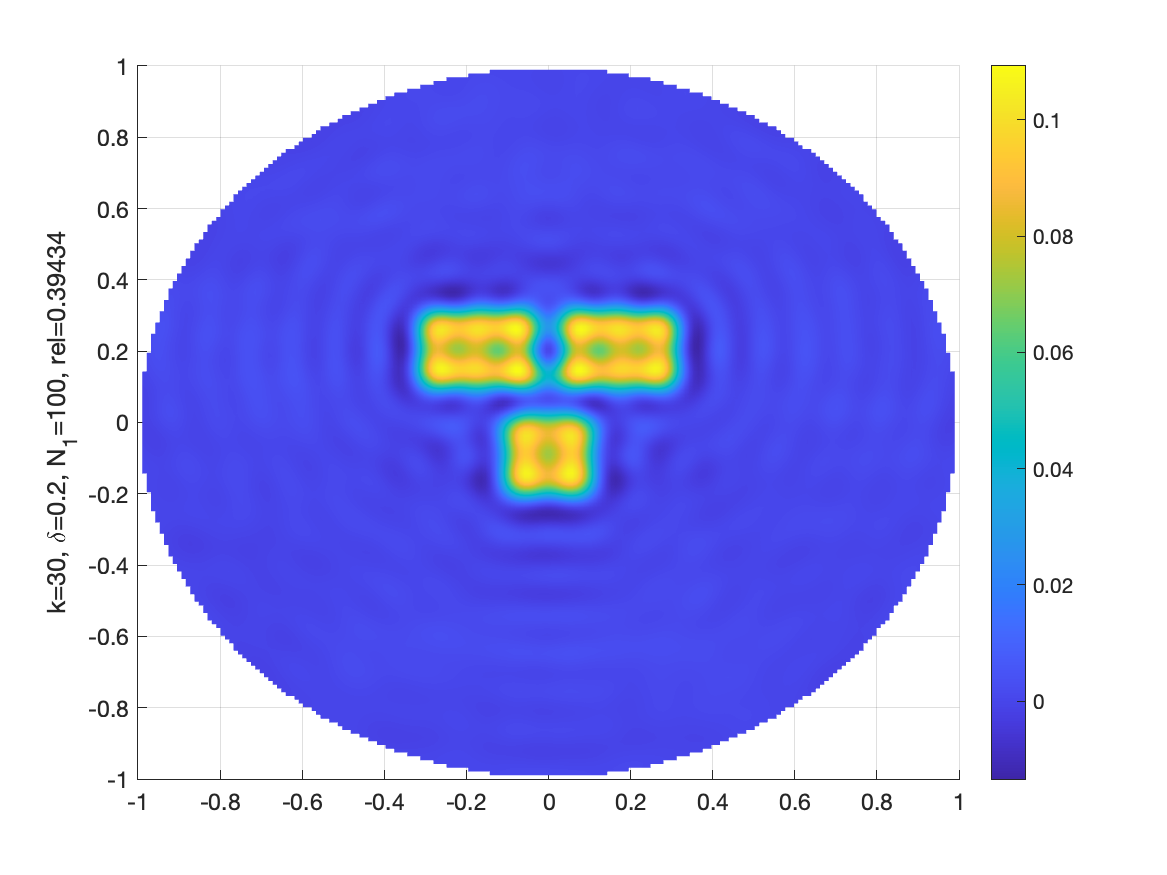}
  }
 \subfloat[k=45,   ${\rm rel(k)} = 0.55556$]
  {
      \includegraphics[width=0.45\linewidth]{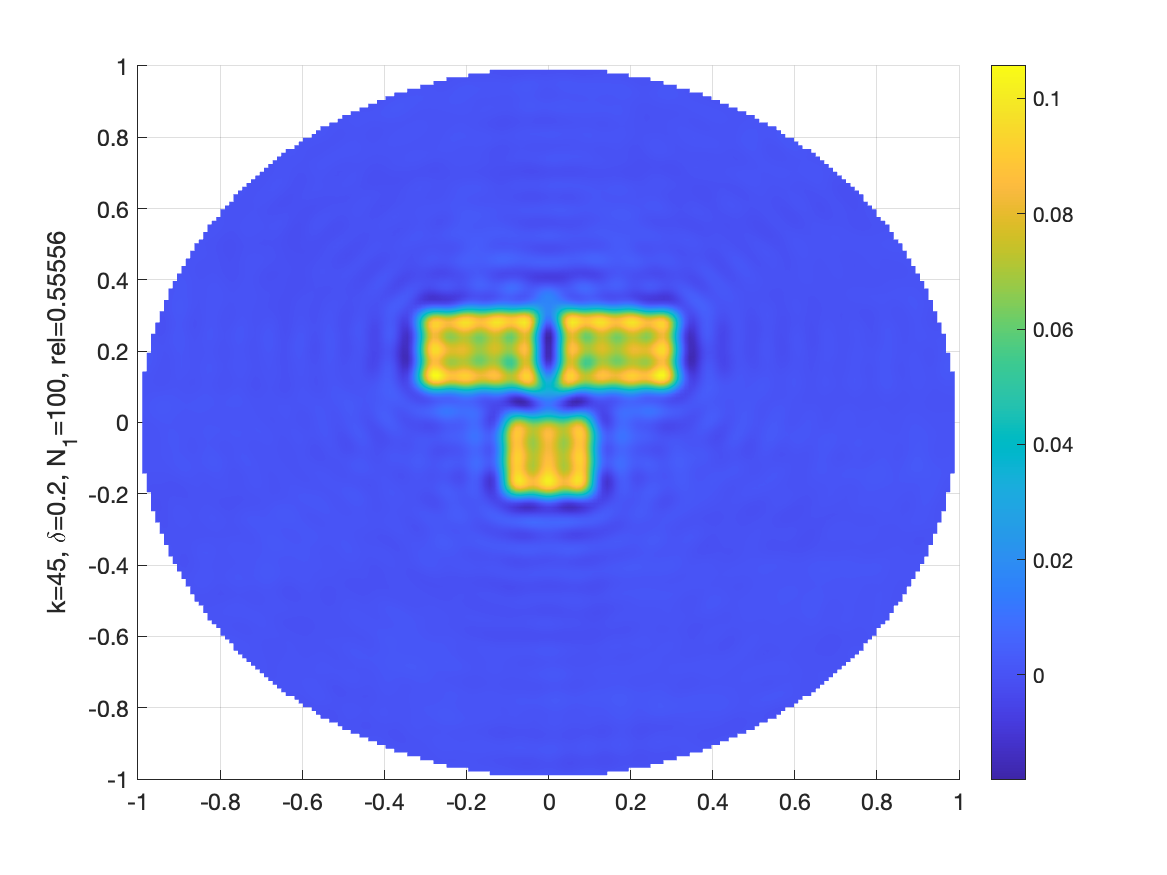}
  }

 \caption{Reconstruction of three rectangles using our proposed method and $20\%$ noisy full far-field data for $k=15,~30,~45$.}
 \label{increasing stability1}
\end{figure}

\subsubsection{Increasing stability} \label{section: increasing stability}
In this section, we test our algorithm by adding  $20\%$ uniform noises to the full far-field data simulated by 
IPscatt \cite{Burgel19}.  The number of incident and observation directions are fixed as $N_1=N_2={ 100}$. 
The modeling error is given by $ {\rm rel(k)}$ for different wavenumber $k$.

We begin with an experiment with $q$ supported in three rectangles. \Cref{increasing stability1} displays 
the reconstructions of the three rectangles at wavenumbers $k \in \{15, 30, 45\}$. Increasing the value of 
wave number $k$ leads to better reconstruction under the premise that the relative modeling error is not too large.

Next we consider a more complicated example where the ground-truth of the unknown  $q$ is given by \Cref{increasing stability2} and can be generated by  Matlab function ``ship2D.m'' in IPscatt \cite{Burgel19}. The distance between two small squares is $1/16$.    Again we have observed the resolution becomes better as the { wave number} $k$ gets larger.

\begin{figure}[htp]
\centering
  \subfloat[Ground-truth, real part]
  {
      \includegraphics[width=0.45\linewidth]{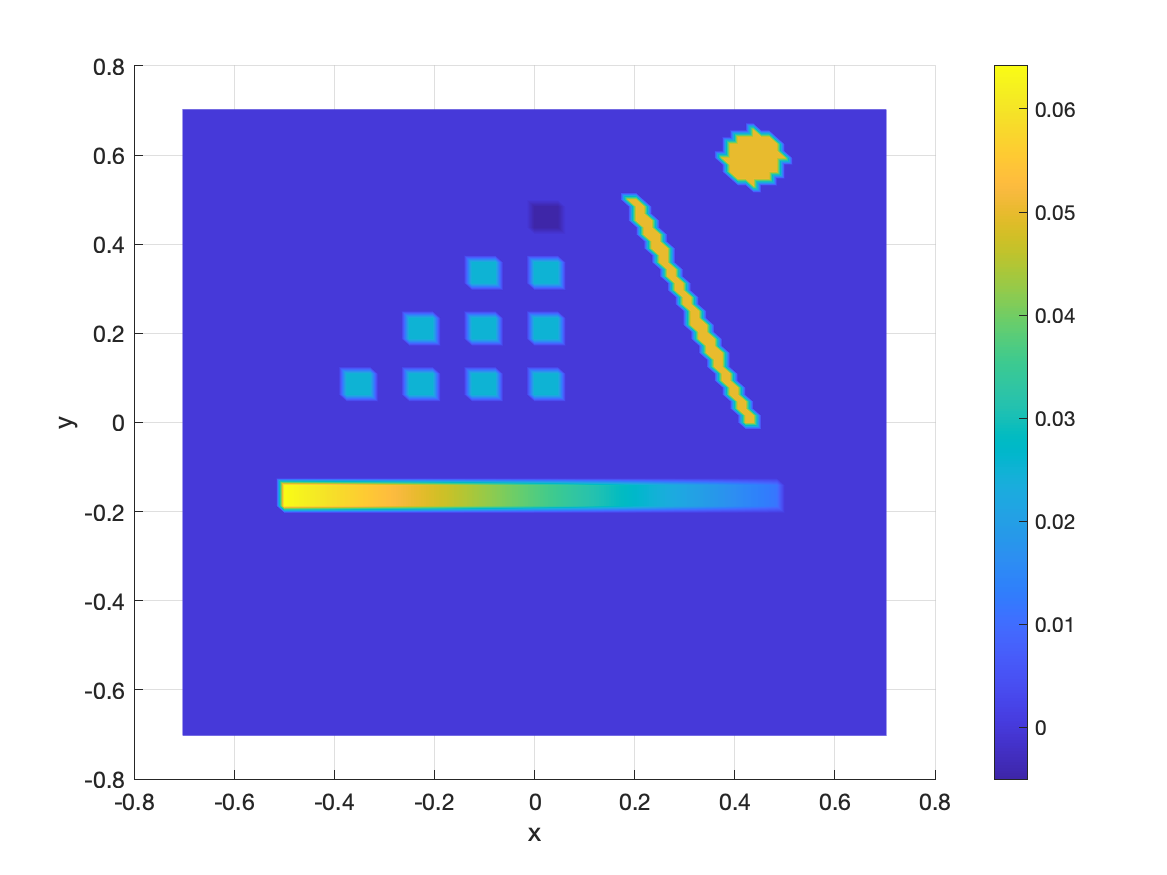}
  }
    \subfloat[Ground-truth, imaginary part]
  {
      \includegraphics[width=0.45\linewidth]{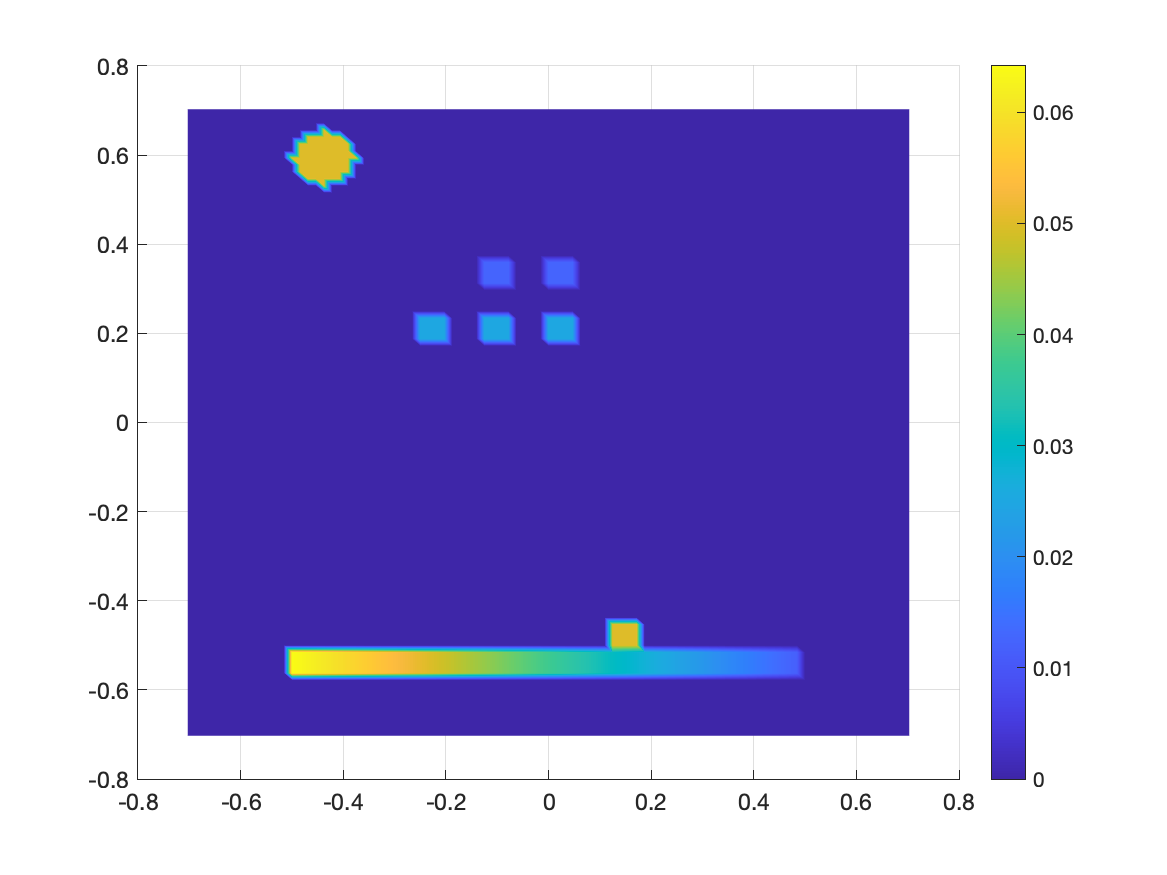}
  }\\
 \subfloat[k=15,  ${\rm rel(k)} = 0.046737$, real part]
  {
      \includegraphics[width=0.45\linewidth]{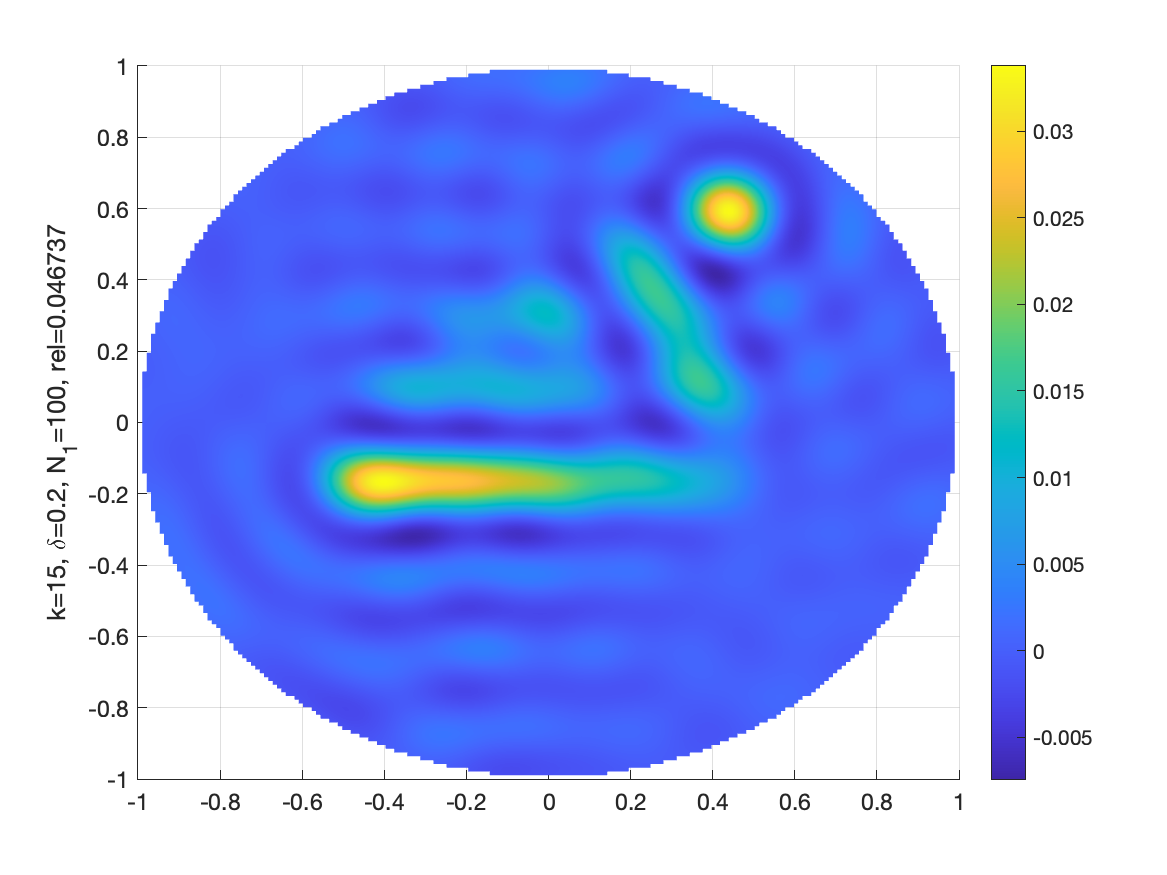}
  }
   \subfloat[k=15, imaginary part]
  {
      \includegraphics[width=0.45\linewidth]{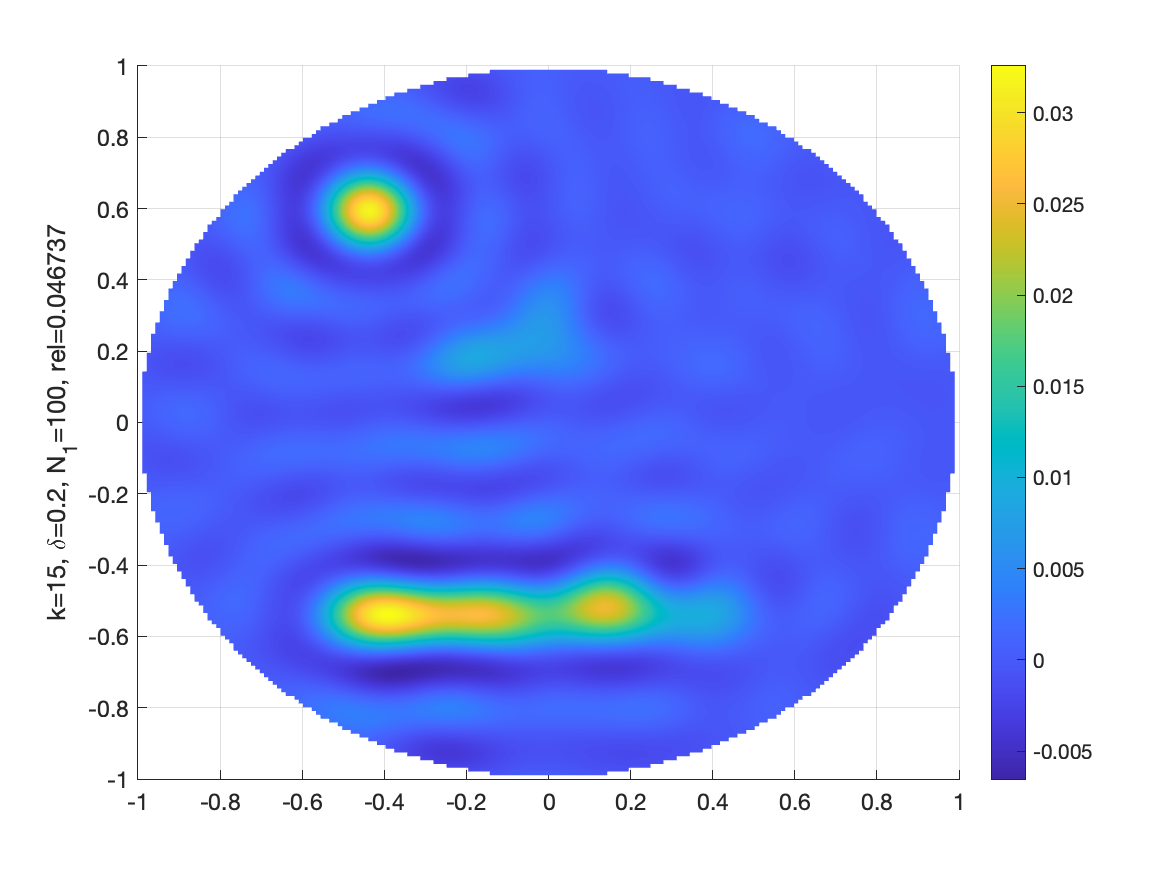}
  }\\
 \subfloat[k=30,  ${\rm rel(k)} = 0.10405$, real part]
  {
      \includegraphics[width=0.45\linewidth]{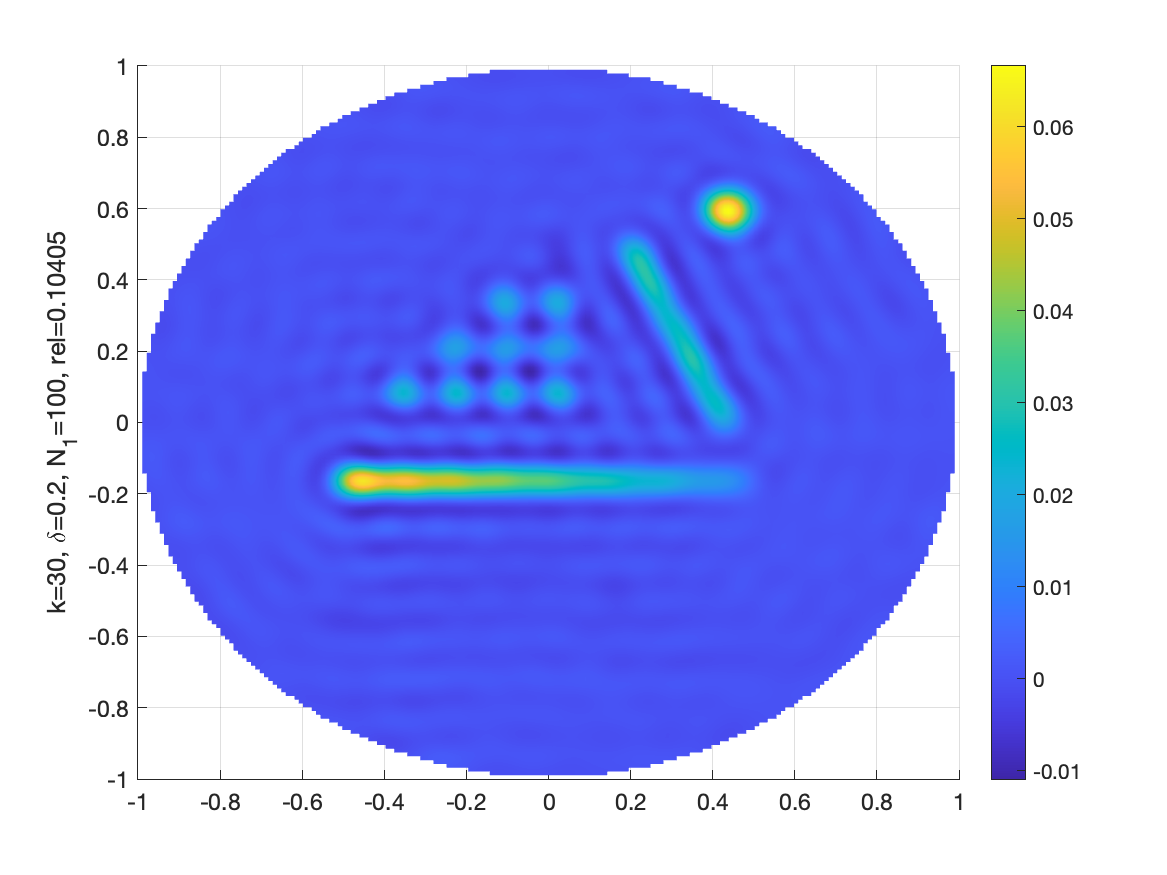}
  }
   \subfloat[k=30, imaginary part]
  {
      \includegraphics[width=0.45\linewidth]{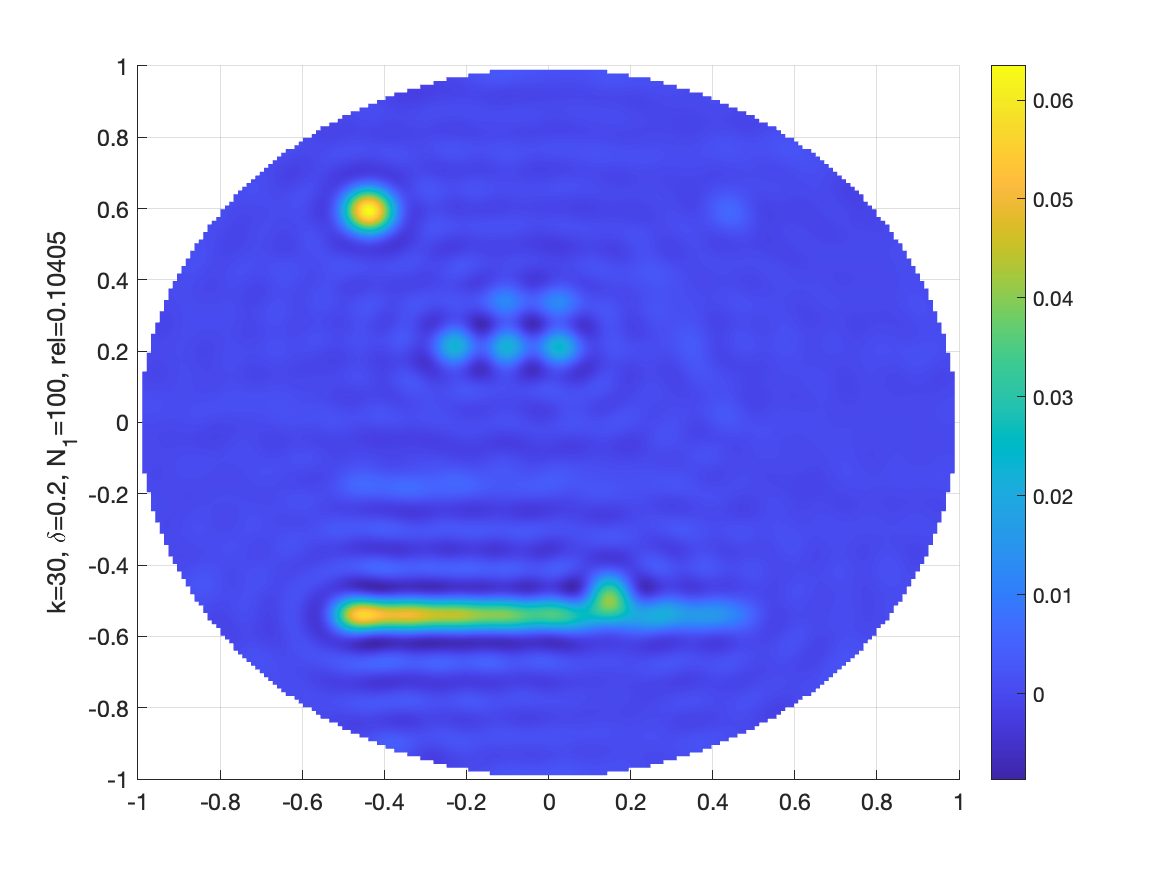}
  } \\
 \subfloat[k=45,   ${\rm rel(k)} = 0.17376$, real part]
  {
      \includegraphics[width=0.45\linewidth]{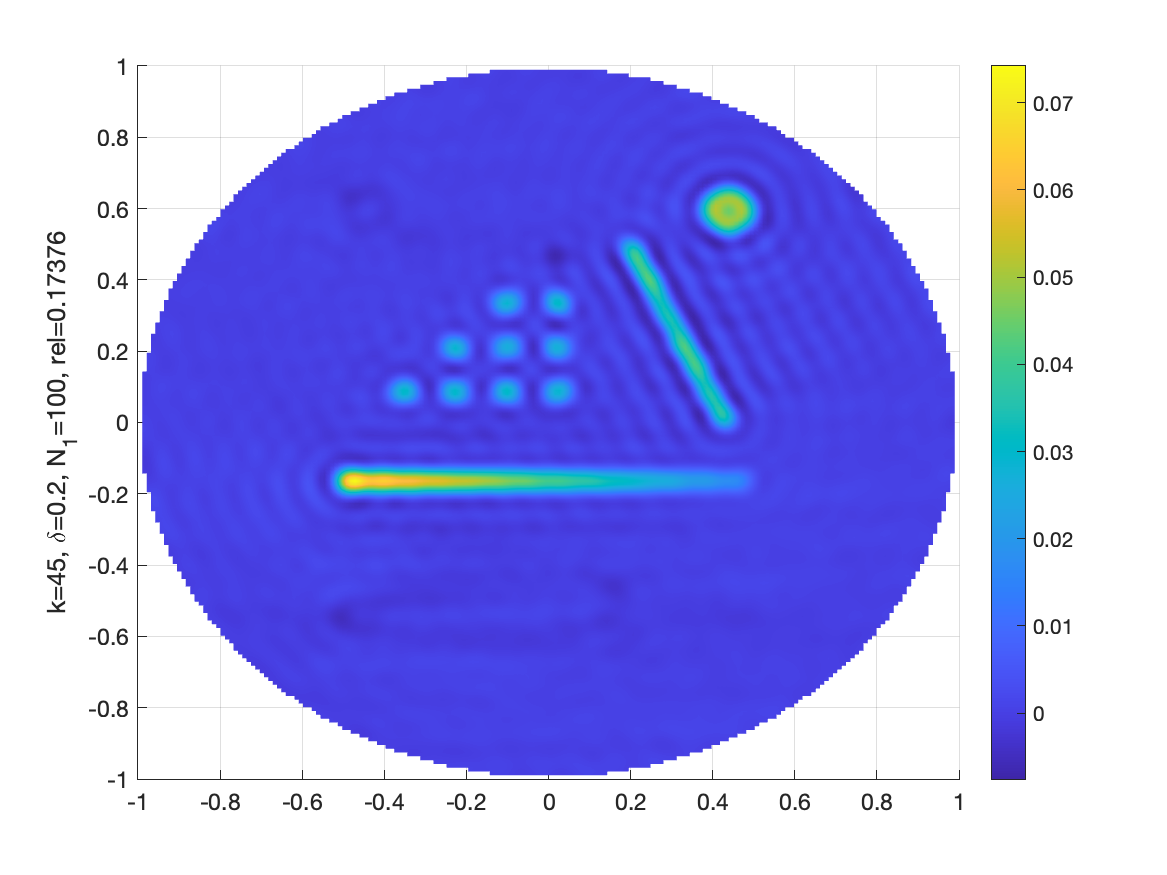}
  }
    \subfloat[k=45, imaginary part]
  {
      \includegraphics[width=0.45\linewidth]{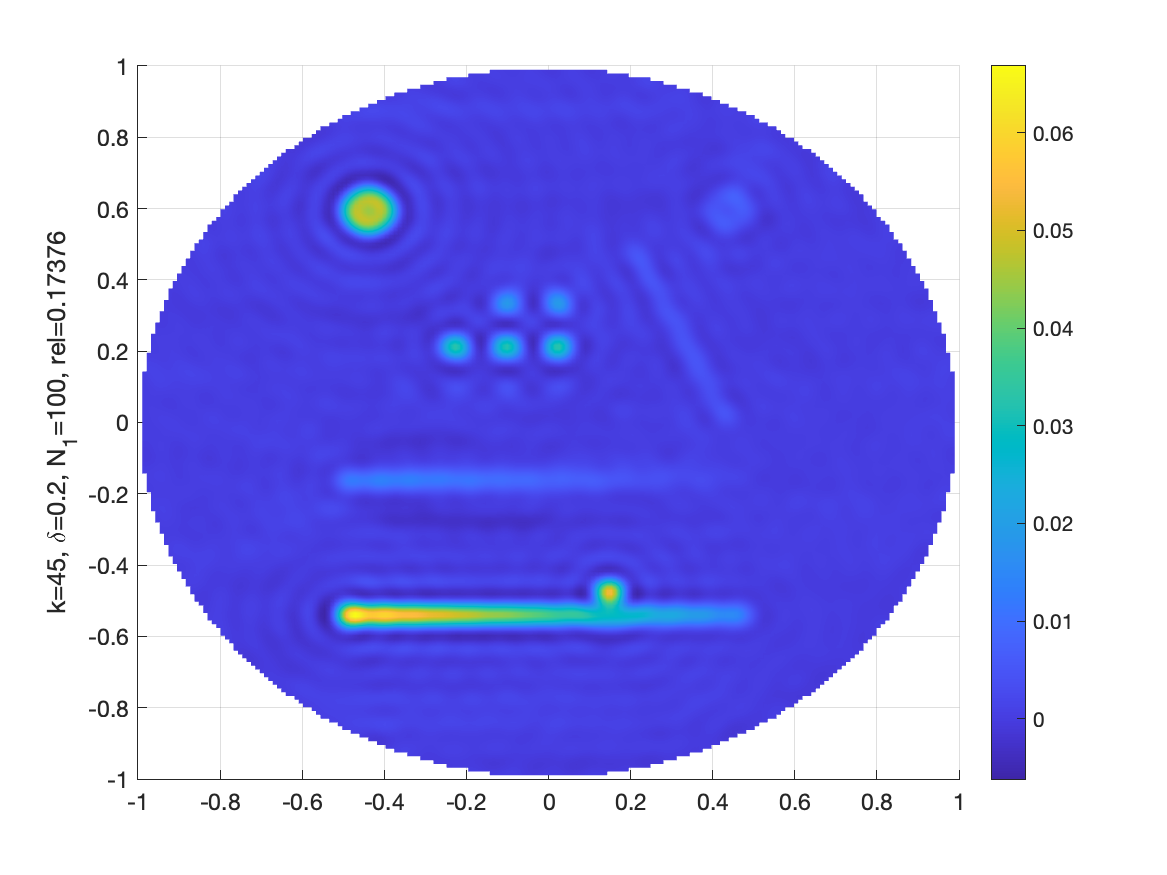}
  }
 \caption{Reconstruction of ``ship2D'' contrast using $20\%$ noisy full far-field data for $k=15,~30,~45$. Left column:  real part, right column: imaginary part.}
 \label{increasing stability2}
\end{figure}

We continue to test the algorithm when the relative modeling error is not small by choosing $k \in \{ 60, 75\}$. In this case, {  we chose to use  a larger data set $N_1\times N_2=200\times 200$ since the wave number is much larger.} %
 The reconstruction remains good when the relative modeling error is around $35\%$.
\begin{figure}[htp]
\centering
 \subfloat[k=60,~ ${\rm rel(k)}=0.25466$, real part]
  {
      \includegraphics[width=0.45\linewidth]{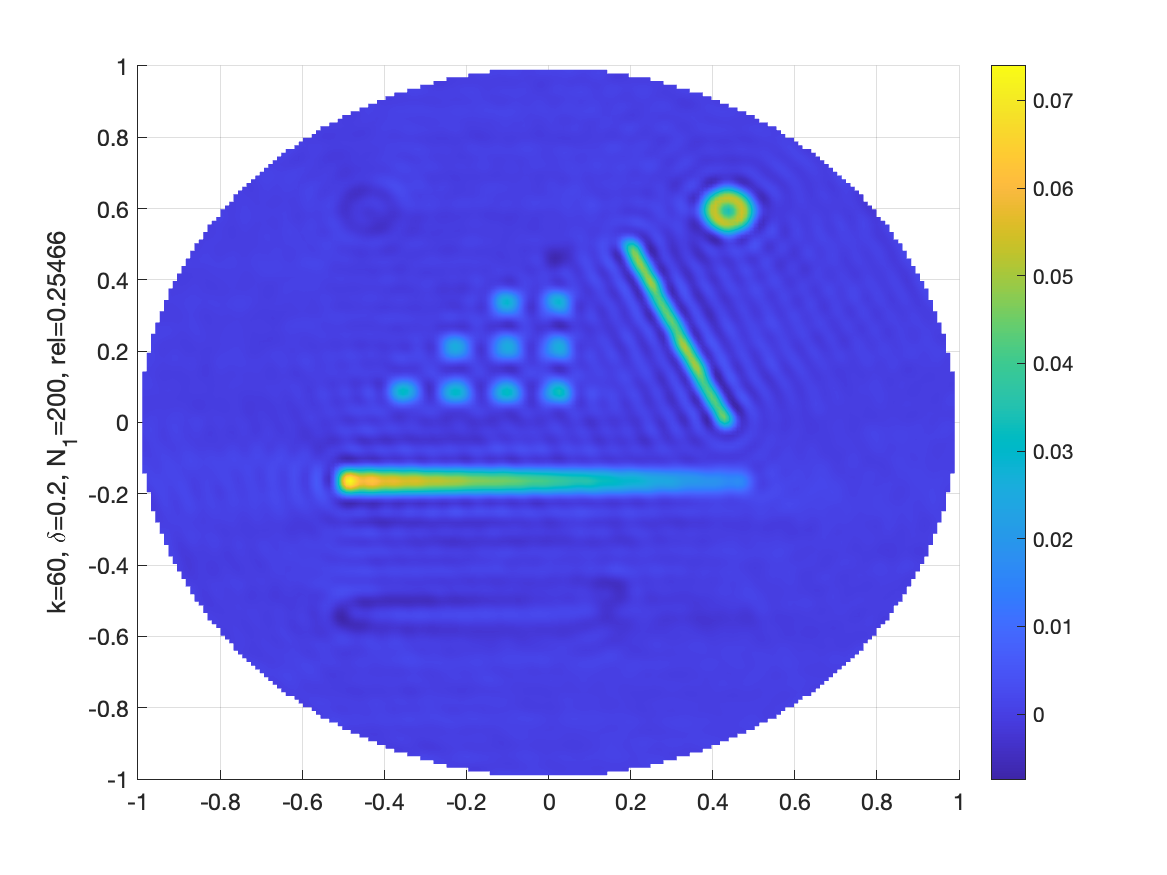}
  }
   \subfloat[k=60, imaginary part]
  {
      \includegraphics[width=0.45\linewidth]{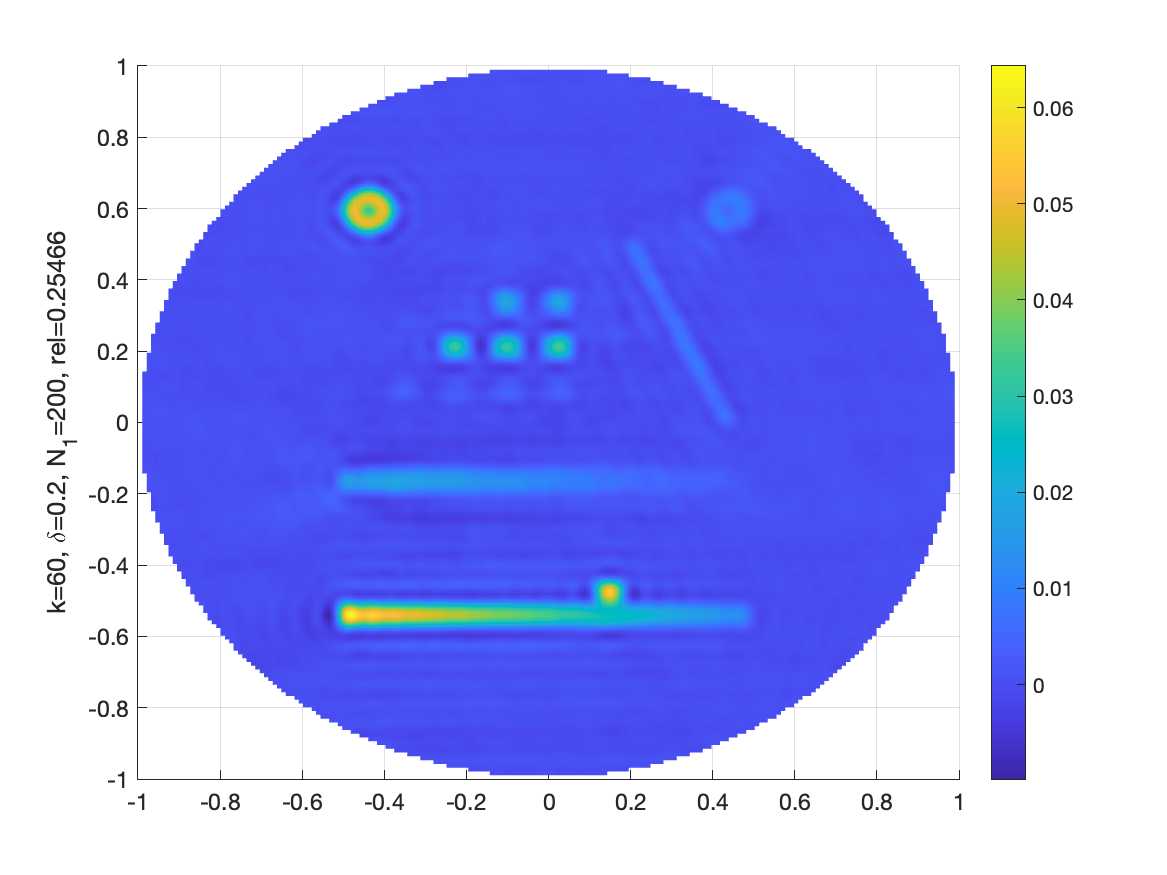}
  }\\
 \subfloat[k=75,  ~${\rm rel(k)}=0.34504$, real part]
  {
      \includegraphics[width=0.45\linewidth]{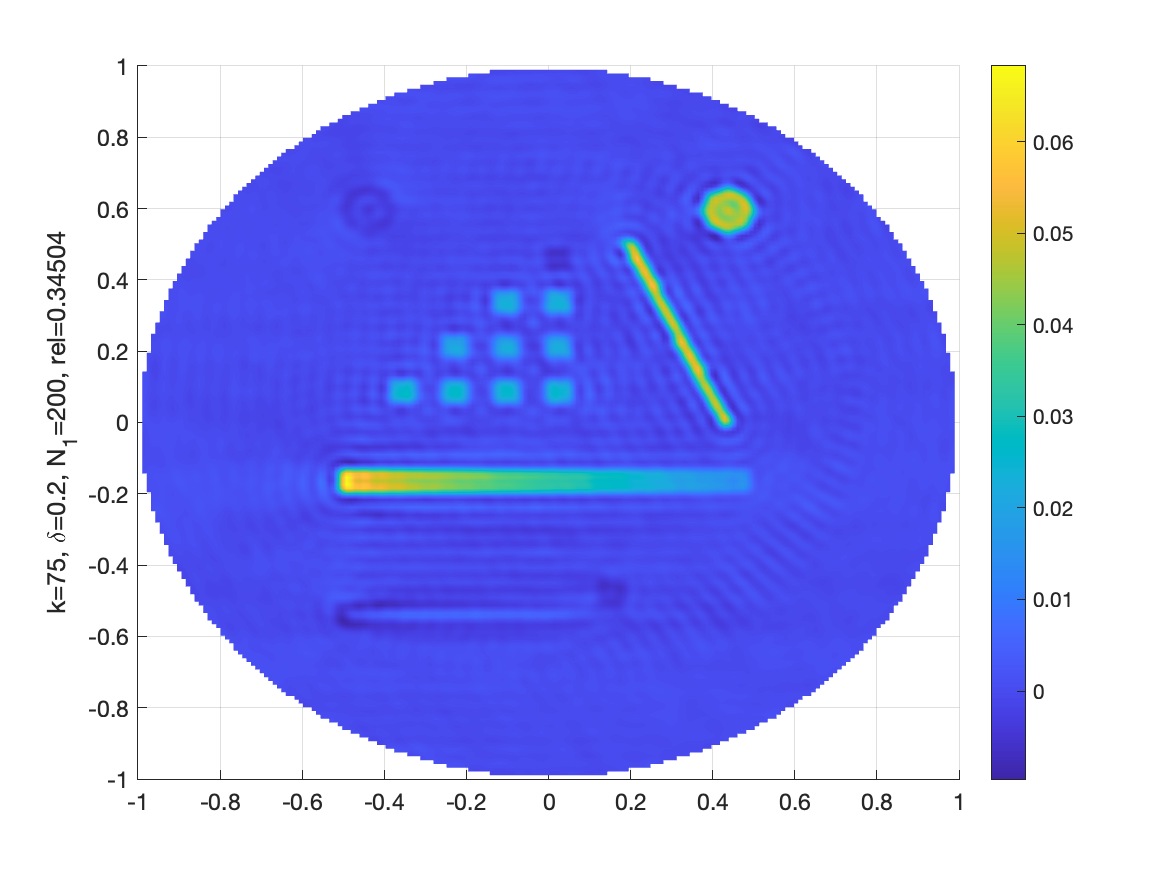}
  }
   \subfloat[k=75, imaginary part]
  {
      \includegraphics[width=0.45\linewidth]{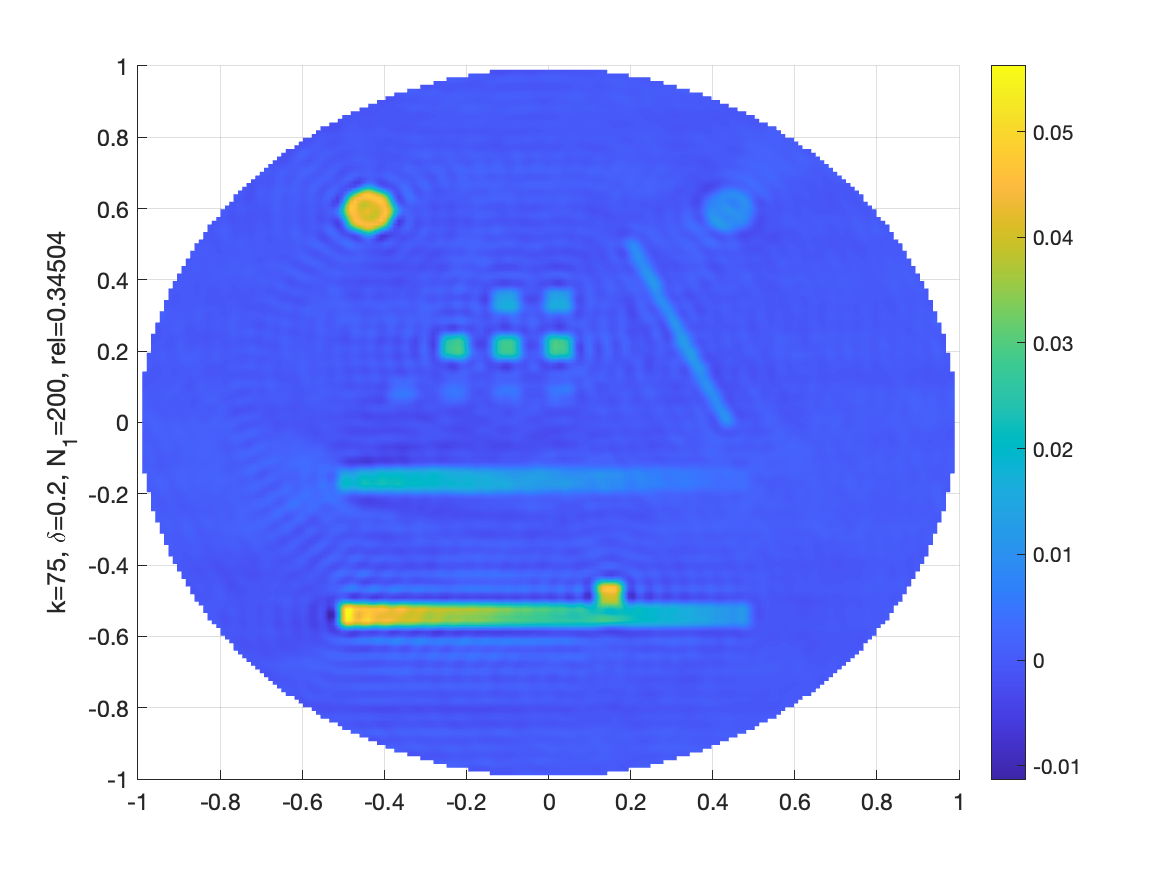}
  }
 \caption{Reconstruction using $20\%$ noisy full far-field data for $k \in \{ 60,   75\}$.}
\label{increasing stability}
\end{figure}

\subsubsection{Comparison with an iterative reconstruction method}
In this section we compare our algorithm with IPscatt's iterative method. Iterative methods usually depend on the choice of initial guess, are computationally expensive, and may suffer from local minima. In \Cref{IPscattComp}, the initial guess of IPscatt's iterative method is zero, and we run the iterative method with the number of iterations ${\rm pdaN}=100$. We also report that the result with number of iterations ${\rm pdaN}=50$ is similar. Our algorithm does not require any initial guess and the numerical experiments show that it gives very good reconstructions for relative modeling error less than $50\%$.
Our algorithm is a direct method which is more efficient than the iterative method.
However, we note that our algorithm does not work for very large modeling error, say larger than $100\%$. One future direction is to utilize our algorithm with other nonlinear techniques to deal with nonlinear model with very large modeling errors.

\begin{figure}[htp]
\centering
  \subfloat[Proposed Method, real part]
  {
      \includegraphics[width=0.45\linewidth]{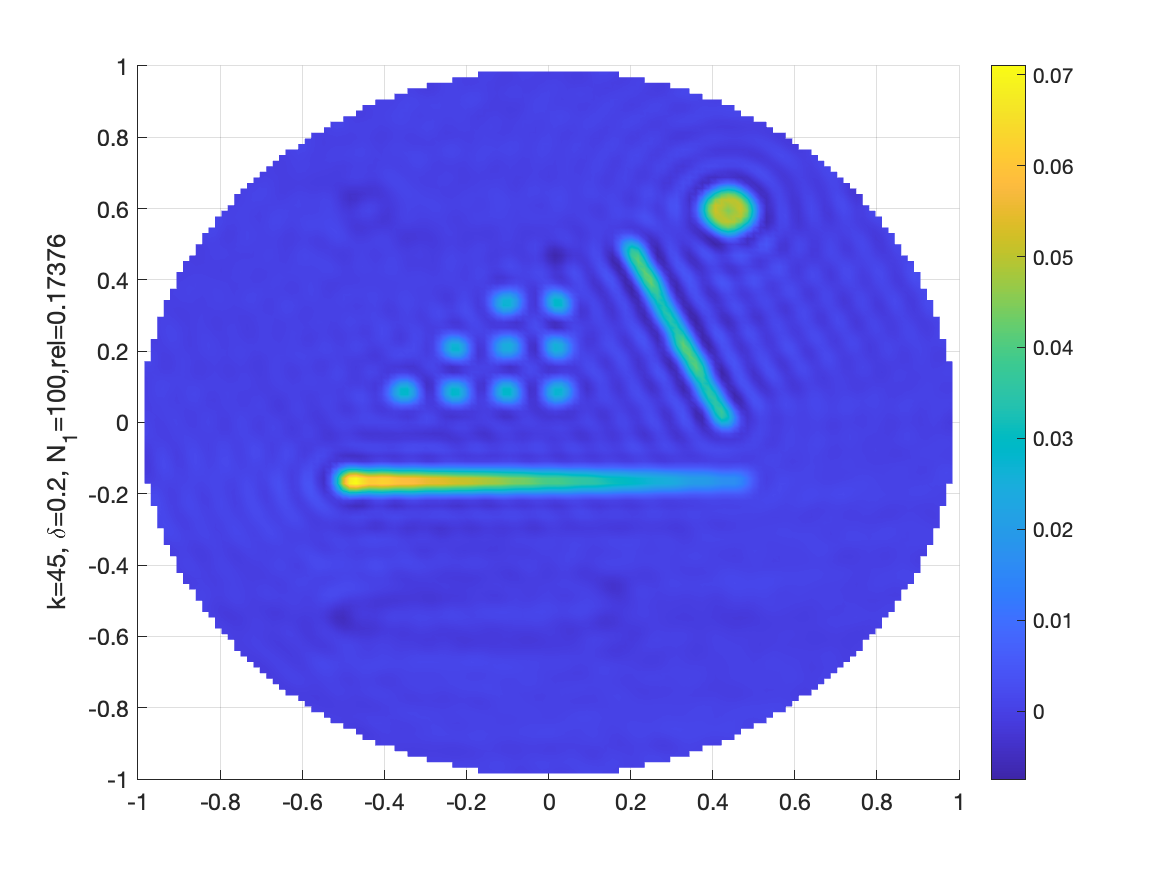}
  }
 \subfloat[IPscatt's Iterative Method, real part]
  {
      \includegraphics[width=0.45\linewidth]{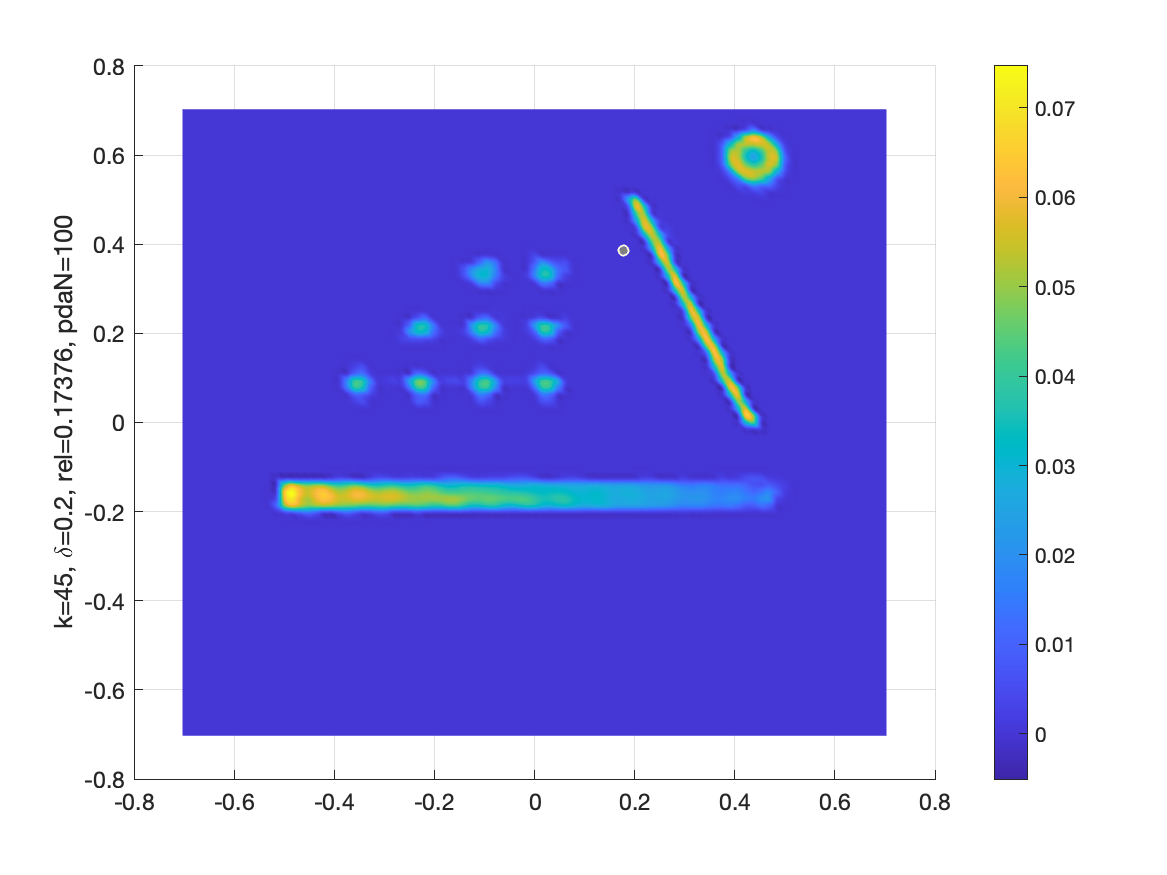}
  } \\
  \subfloat[Proposed Method, imaginary part]
  {
      \includegraphics[width=0.45\linewidth]{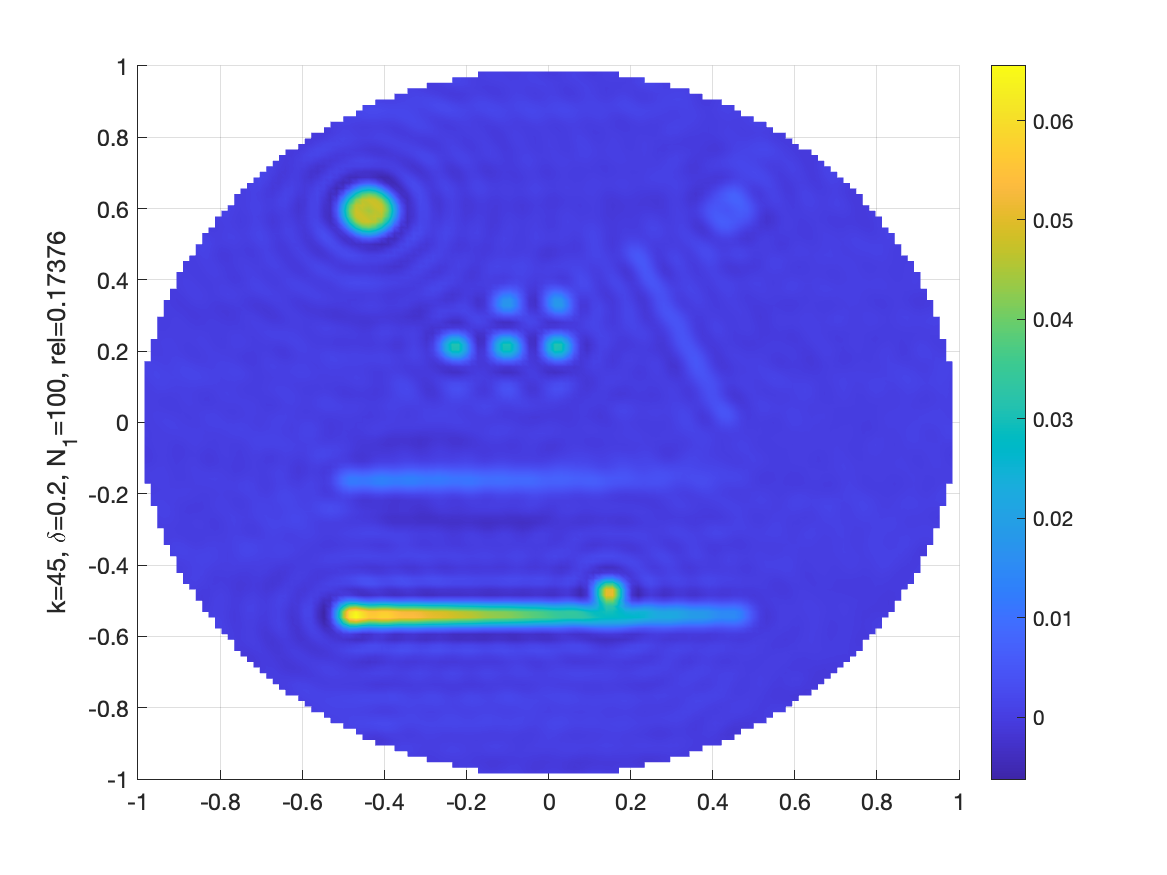}
  }
 \subfloat[IPscatt's Iterative Method, imaginary part]
  {
      \includegraphics[width=0.45\linewidth]{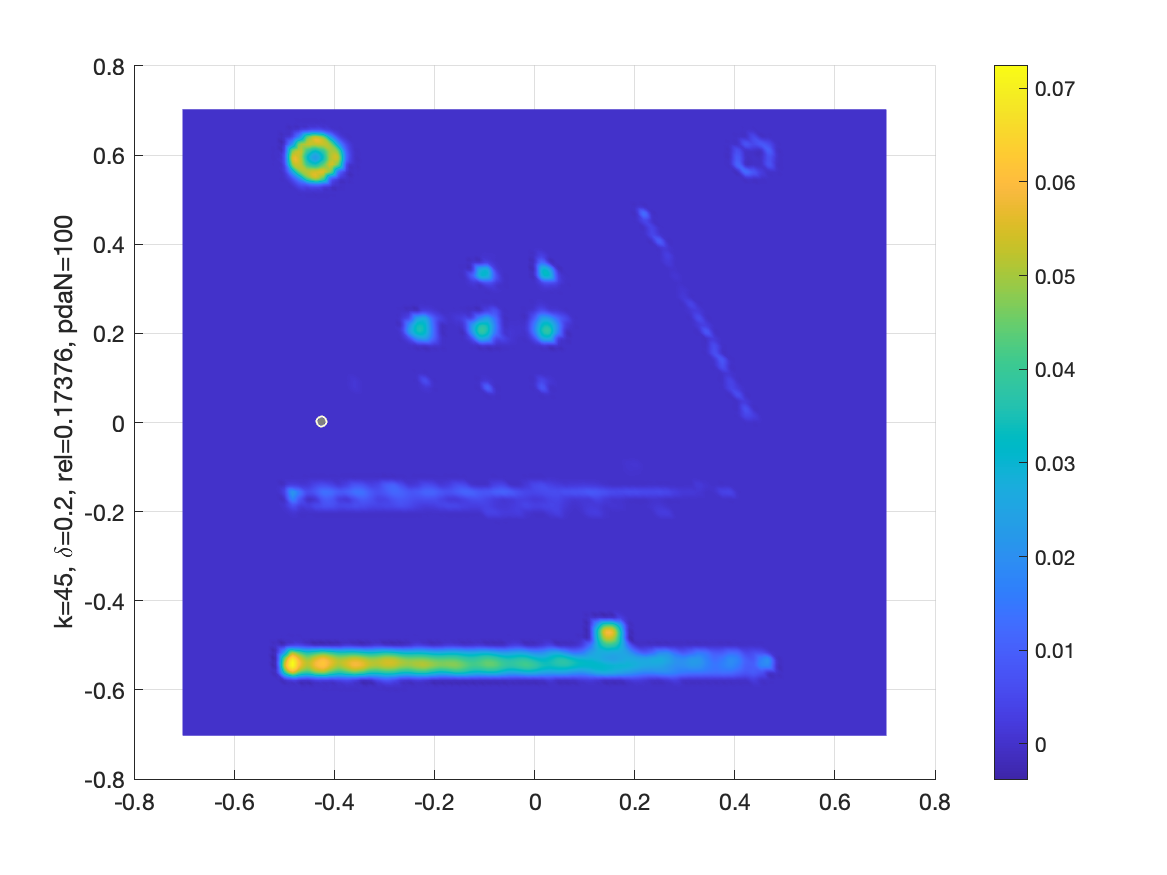}
  }

 \caption{Reconstruction of ``ship2D'' contrast using the our proposed algorithm and IPscatt's iterative method with $k=45$.}
\label{IPscattComp}
\end{figure}

{ 
\subsubsection{Far-field data extrapolation}
When the data set $\{u^\infty(\hat{x}_i;\hat{\theta}_j;k)\}_{i=1,j=1}^{N_1,N_2}$ is of small scale, one idea is to perform Fourier interpolation to have an extrapolated data set of dimension $\tilde{N}_1\times \tilde{N}_2$ such that $\tilde{N}_1, \tilde{N}_2$ are larger  and then apply our proposed algorithm. In Figure \ref{figure: Fourier interpolation}, we use fft interpolation to extrapolate the far-field data and then apply our proposed algorithm. The first row is to reconstruct contrast
$q(x)=0.02$ supported in a rectangle $\Omega=\{x\in \mathbb{R}^2:  |x_1|,|x_2| < 1/2\}$
with $N_2=N_1=16$ and $\tilde{N}_2=\tilde{N}_1=128$, and the second row is to reconstruct contrast ``ship2D'' with
$N_2=N_1=35$ and $\tilde{N}_2=\tilde{N}_1=140$. We observe that the unknowns can still be reasonably identified.

\begin{figure}[htp]
\centering
  \subfloat[]
  {
      \includegraphics[width=0.45\linewidth]{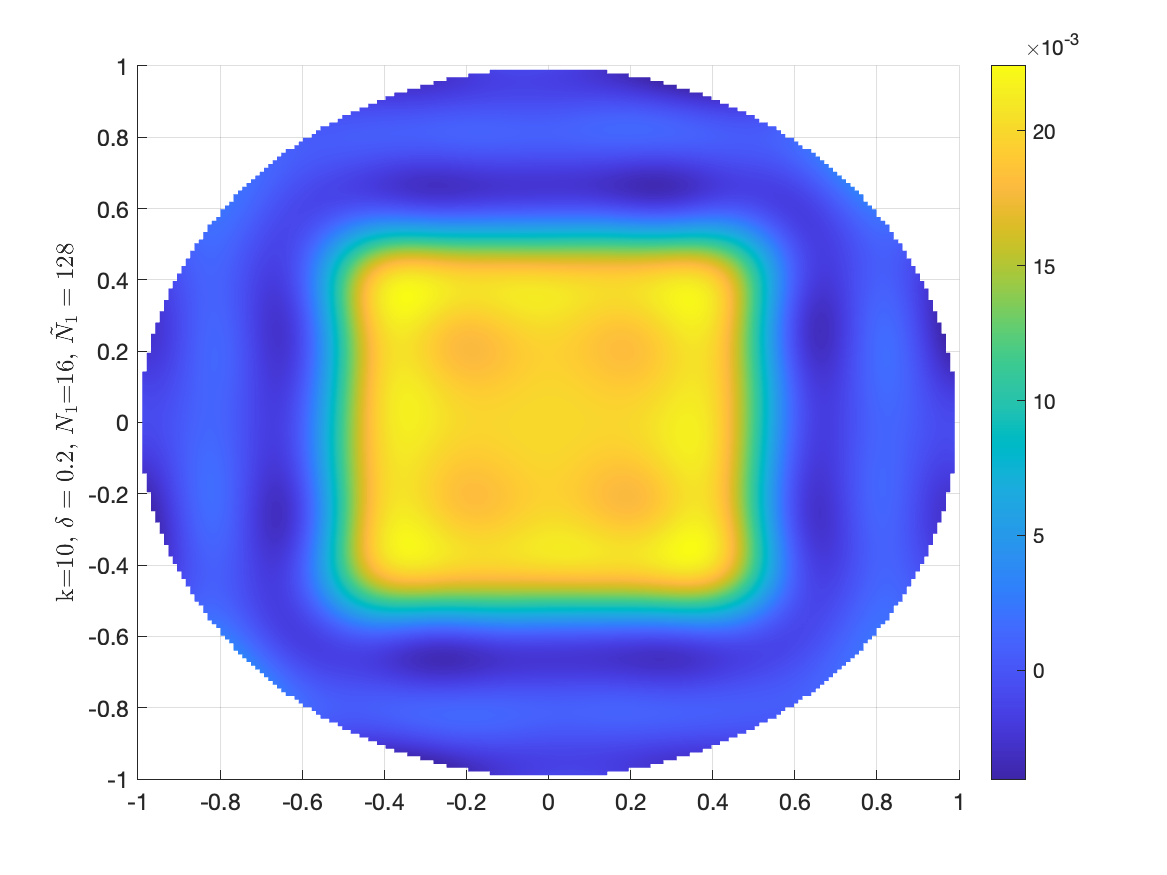}
  }
 \subfloat[]
  {
      \includegraphics[width=0.45\linewidth]{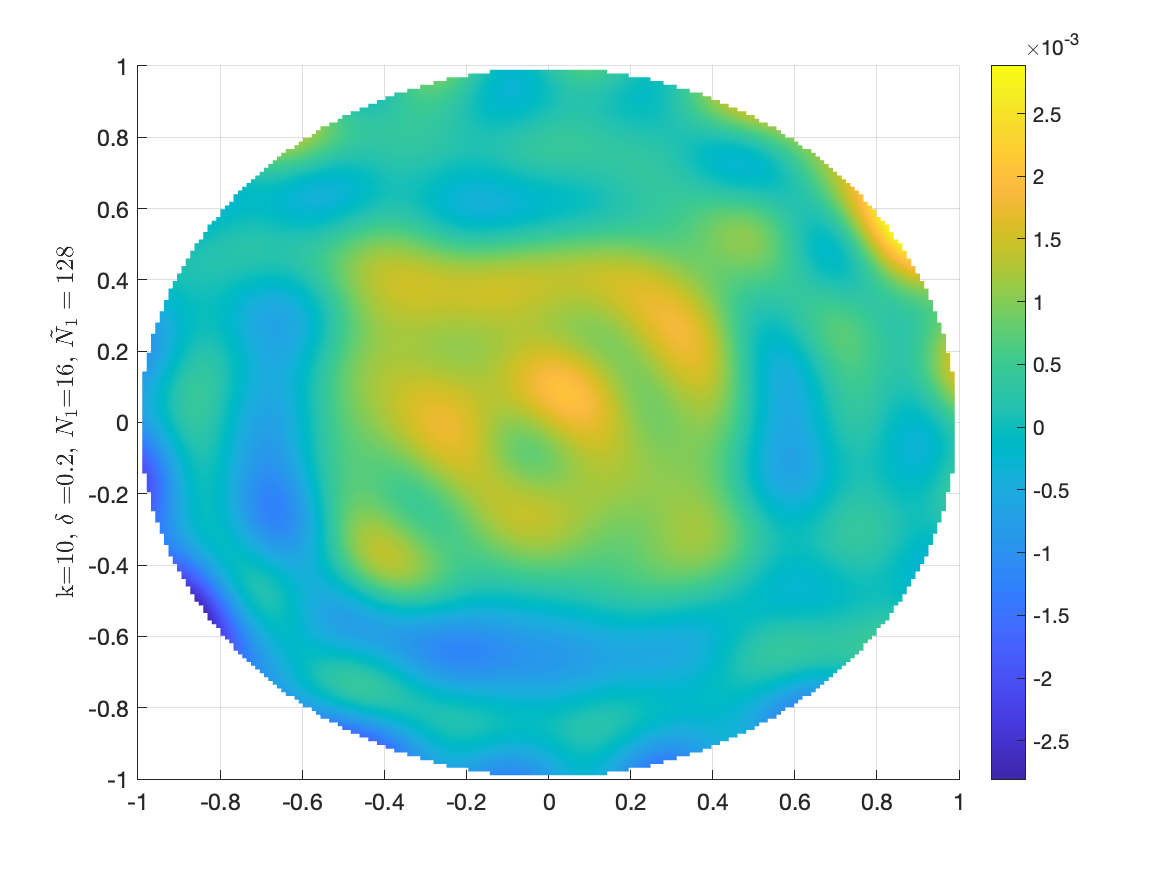}
  } \\
  \subfloat[]
  {
      \includegraphics[width=0.45\linewidth]{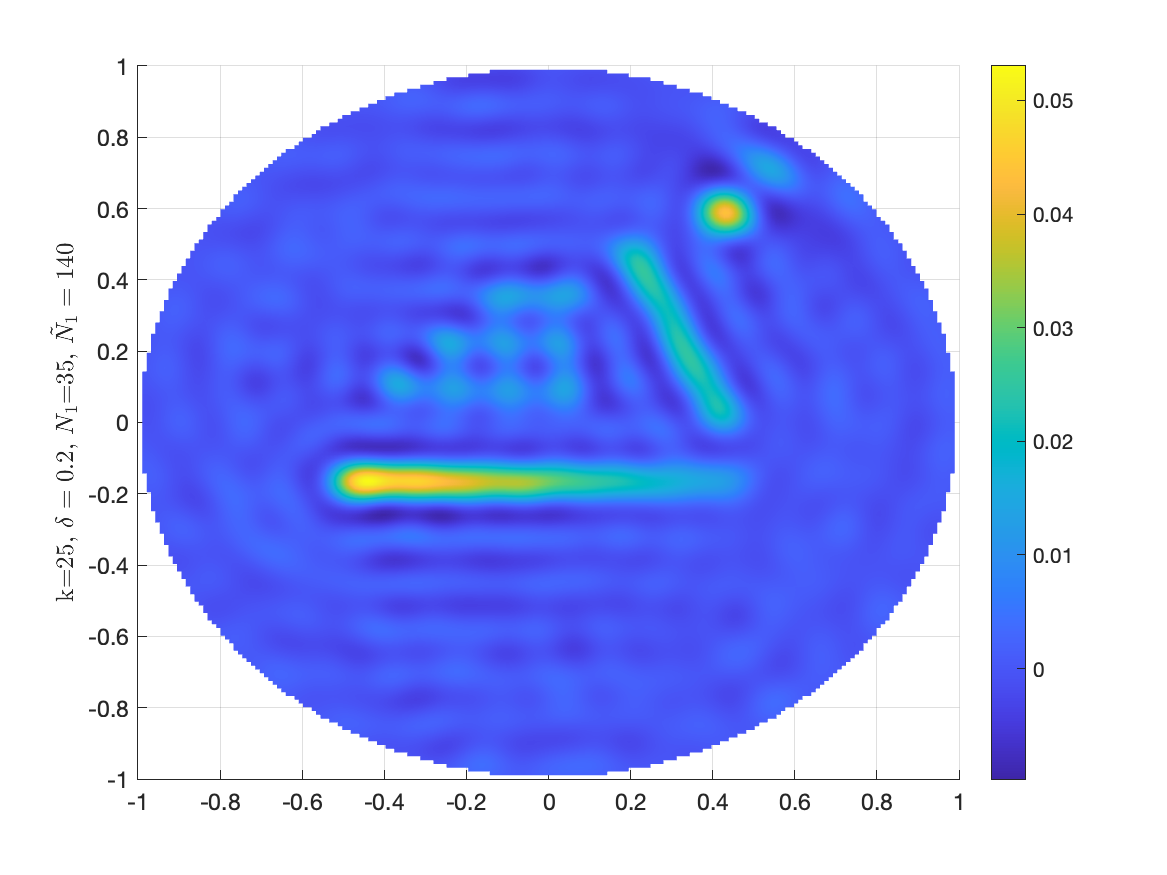}
  }
 \subfloat[]
  {
      \includegraphics[width=0.45\linewidth]{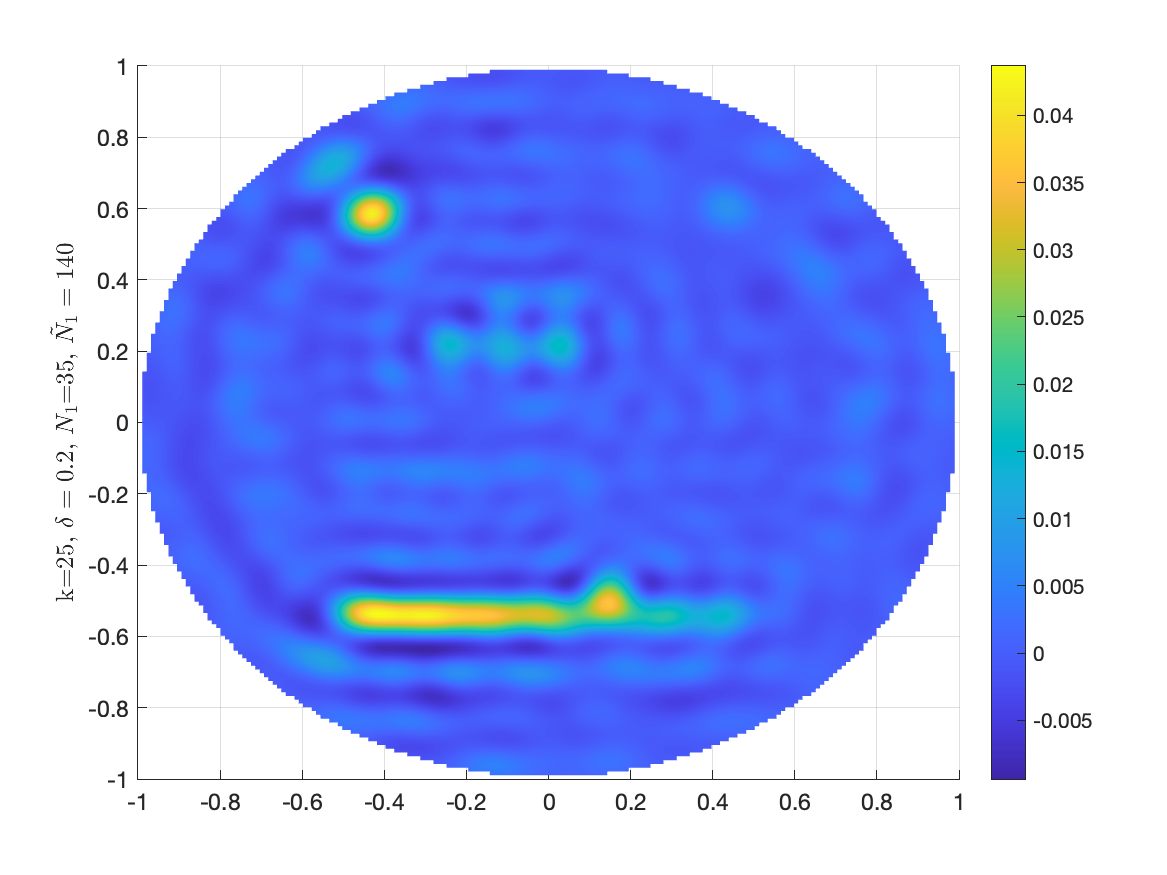}
  }

 \caption{Reconstruction using {  small data set of dimension $16\times 16$ (first row) and $35\times 35$ (second row) with far-field data extrapolation}. Left column: real part, right column: imaginary part.}
\label{figure: Fourier interpolation}
\end{figure}
}

\section*{Acknowledgments}

The work of Bo Zhang was supported by the NNSF of China (grant no. 12431016).

\bibliographystyle{siamplain}

\begin{thebibliography}{}

\bibitem{Abramowitz64}
{M~Abramowitz  and A~Irene, eds}.
\newblock {\it
Handbook of Mathematical Functions With Formulas, Graphs, and Mathematical Tables}, Washington: U.S. Govt. Print. Off., 1964.

\bibitem{audiberthaddar15}
{L~Audibert  and H~Haddar}.
A generalized formulation of the linear sampling method with exact characterization of targets in terms of far-field measurements,
{\it Inverse Problems \bf30}, 035011, 2015.

\bibitem{meng23parameter}
{L~Audibert and S~Meng}. {Shape and parameter identification by the
  linear sampling method for a restricted Fourier integral operator},  {\it Inverse Problems \bf 40} (9), 095007, 2024.

\bibitem{Bachmayr23}
M~Bachmayr.
Low-rank tensor methods for partial differential equations, {\it Acta Numerica}, pp~1--121, 2023.

  \bibitem{boyd05code}
J~ Boyd.
Algorithm 840: computation of grid points, quadrature weights and derivatives for spectral element methods using prolate spheroidal wave functions--prolate elements,
{\it ACM Trans. Math. Software \bf31}, no. 1, 149--165, 2005.

\bibitem{boydxu09}
J~Boyd and F~Xu.
Divergence (Runge Phenomenon) for least-squares polynomial approximation on an equispaced grid and Mock-Chebyshev subset interpolation,
{\it Applied Mathematics and Computation \bf210},pp.~158--168, 2009.

\bibitem{bukhgeim08}
A~Bukhgeim. Recovering a potential from Cauchy data in the two dimensiona case, {\it Jour. Inverse Ill-Posed Problems \bf 16}, pp.~19--33. 2008.

\bibitem{Burgel19}
F~ B\"{u}rgel, K~Kazimierski, and A~Lechleiter. Algorithm 1001: Ipscatt-a matlab toolbox
for the inverse medium problem in scattering, {\it ACM Trans. Math. Software \bf 45}(4),
2019.

\bibitem{cakoni2016qualitative}
F~Cakoni and D~Colton.
\newblock {\it Qualitative Approach to Inverse Scattering Theory}, Springer, 2016.


\bibitem{cakoni2016inverse}
F~Cakoni, D~Colton and H~Haddar.
\newblock {\it Inverse Scattering Theory and Transmission Eigenvalues},   CBMS-NSF, SIAM Publications {\bf98}, 2nd Edition, 2023.

\bibitem{ColtonKirsch96}
D~Colton and A~Kirsch. A simple method for solving inverse scattering problems in the resonance region,
{\it Inverse Problems \bf12}, pp. 383--393, 1996.

\bibitem{colton2012inverse}
D~Colton and R~Kress.
\newblock {\it Inverse Acoustic and Electromagnetic Scattering Theory}, Springer Nature, New York, 2019.

\bibitem{CV1995}
{C~Cortes and V~Vapnik}. {Support-vector networks},
{\it Machine learning \bf20}, pp.~273--297, 1995.


{ 
\bibitem{DesaiLahivaaraMonk}
A~Desai, T~Lahivaara and P~ Monk.
Neural-Enhanced {B}orn Approximation for Inverse Scattering,
{\it arXiv}, 2025.
}


\bibitem{gao2022artificial}
{Y~Gao, H~Liu, X~Wang, and K~Zhang}. {On an artificial neural
  network for inverse scattering problems},
  {\it Journal of Computational Physics \bf448},~110771, 2022.

  \bibitem{Goodfellow16}
I~Goodfellow, Y~Bengio, and A~Courville. {\em Deep learning}, MIT
  press, 2016.

  \bibitem{Gottlieb77book}
  D~Gottlieb and O~Steven.
  {\em Numerical analysis of spectral methods: theory and applications}, Society for Industrial and Applied Mathematics, 1977.

{ 
\bibitem{P Greengard2024arxiv}
 P~Greengard.
 {
Generalized prolate spheroidal functions: algorithms and analysis,{~\it Pure and Applied Analysis \bf6} (3) },  pp.~789--833, 2024.
}

\bibitem{GreengardSerkh18}
P~Greengard and K~ Serkh.
On generalized prolate spheroidal functions --Preliminary Report,
{\it Technical Report YALEU/DCS/TR--1542}, 2018.

 \bibitem{GriesmaierSchmiedecke-source}
{R~Griesmaier and C~Schmiedecke}. {A factorization method for
  multifrequency inverse source problems with sparse far-field measurements},
  {\it SIAM Journal on Imaging Sciences \bf10}, pp.~2119--2139, 2017.



\bibitem{Hackbusch19}
W~Hackbusch. Tensor Spaces and Numerical Tensor Calculus,  second edition, Springer, 2019.

\bibitem{HSS08}
{T~Hofmann, B~Sch\"{o}lkopf, and A~J. Smola}. {Kernel methods in
  machine learning},
  {\it The Annals of Statistics, \bf36}, pp.~1171--1220, 2008.

\bibitem{HrycakIsakov04}
 T~Hrycak and V~Isakov. Increased stability in the continuation of solutions to the helmholtz equation, 
 { \it Inverse Problems \bf20}(3):697--712, 2004.


\bibitem{novikov22}
M~Isaev and R~Novikov.
Reconstruction from the Fourier transform on the ball via prolate spheroidal wave functions, 
{\it Journal de Math\'{e}matiques Pures et Appliqu\'{e}es \bf 163}, pp.318--333, 2022.

\bibitem{novikov22a}
M~Isaev, R~Novikov, G~Sabinin.
Numerical reconstruction from the Fourier transform on the ball using prolate spheroidal wave functions,
{\it Inverse Problems \bf38}, 105002, 2022.

\bibitem{inverarity02}
 G~Inverarity.
 Fast computation of multidimensional fourier integrals, {\it SIAM Journal on Scientific
 Computing \bf 24}(2), 2002.




\bibitem{khooying19}
{Y~Khoo and L~Ying}.
SwitchNet: a neural network model for forward and inverse scattering problems,{\em ~SIAM J. Sci. Comput. \bf41} no.~5, A3182--A3201, 2019.


\bibitem{Kirsch98} A~Kirsch.
	Characterization of the shape of a scattering obstacle using the spectral data of the far-field operator,
	{\it Inverse Problems \bf14}, pp. 1489--1512, 1998.

 \bibitem{Kirsch17} A~Kirsch.
Remarks on the Born approximation and the Factorization Method,
{\it Applicable Analysis \bf 96}, no.~1, pp.~70--84, 2017.

%
{ 
\bibitem{2024Reconstruction}
K~Li, B~Zhang and H~Zhang. Reconstruction of inhomogeneous media by an iteration algorithm with a learned projector,
{\it Inverse Problems \bf40}, 075008, 2024.
}
%

\bibitem{LiChengLu20}
Y~Li, X~Cheng, and J~Lu. Butterfly-Net: optimal function representation based on convolutional neural networks,
{\it Commun. Comput. Phys. \bf 28} No.~5,~pp.1838--1885, 2020.

\bibitem{liu2022deterministic}
{Y~Liu, Z~Wu, J~Sun, and Z~Zhang}. {Deterministic-statistical
  approach for an inverse acoustic source problem using multiple frequency
  limited aperture data},
  {\it Inverse Problems and Imaging \bf 17}(6), 1329--1345, 2024.

\bibitem{meng23data}
S~Meng. Data-driven basis for reconstructing the contrast in  inverse  scattering: Picard criterion, regularity, 
regularization, and stability, {\it SIAM J. Appl. Math. \bf83} (5), 2003--2026, 2023.



\bibitem{mengzhang24}
S~Meng and B~Zhang.  A kernel machine learning for inverse source and scattering problems, {\it SIAM J. Numer. Anal. \bf 62}~(3), 1443--1464, 2024.

\bibitem{moskowSchotland08}
S~Moskow and J~Schotland. Convergence and stability of the inverse Born series for diffuse waves,
{\it Inverse Problems \bf24},  065004, 2008.



\bibitem{ning2023direct}
{J~Ning, F~Han, and J~Zou}. { A direct sampling-based deep learning
  approach for inverse medium scattering problems},
  {\it Inverse Problems \bf40}, ~015005, 2024.

%
%
%
%

\bibitem{osborn75}
 J~Osborn.  Spectral approximation for compact operators,
{\it Math. Comp. \bf29},~pp.712--725, 1975.

\bibitem{osipov2013prolate}
{A~Osipov, V~Rokhlin, and H~Xiao}.
{\em Prolate spheroidal wave
  functions of order zero}, Springer Ser. Appl. Math. Sci. 187, 2013.

\bibitem{quarteroni2000book}
A~Quarteroni, R~Sacco and F~Saleri.
\newblock{\it Numerical Mathematics}, Springer, 2000.

\bibitem{Shkolnisky07}
Y~Shkolnisky.
Prolate Spheroidal Wave Functions on a disc -- Integration and approximation of two-dimensional bandlimited functions,
{\it Appl. Comput. Harmon. Anal.~\bf22}:~235--256, 2007.

\bibitem{Slepian61} D~Slepian and H~ Pollak. Prolate Spheroidal Wave Functions, Fourier Analysis and Uncertainty -I, 
{\it Bell System Tech. J. \bf40}, pp. 43--64, 1961.

\bibitem{Slepian64} D~Slepian. Prolate Spheroidal Wave Functions, Fourier Analysis and Uncertainty -IV: Extensions to Many Dimensions; Generalized Prolate Spheroidal Functions, {\it Bell System Tech. J. \bf43}, pp. 3009--3057, 1964.

\bibitem{Slepian78}
D~Slepian.
Prolate spheroidal wave functions, Fourier analysis, and uncertainty V: The discrete case,
{\it Bell System Tech. J. \bf57}, pp.1371--1430, 1978.

\bibitem{SubbarayappaIsakov07}
{D~Subbarayappa and V~Isakov}.
\newblock{On increased stability in the continuation of the Helmholtz equation. {\em ~Inverse Problems \bf23}~no.4,~pp.~1689--1697, 2007.}

%
 %
 %
 %



\bibitem{ZLWZ20}
J~Zhang,  H~Li, L-L~Wang and Z~Zhang.
Ball prolate spheroidal wave functions in arbitrary dimensions,
{\it Appl. Comput. Harmon. Anal. \bf48}, no. 2, pp.~ 539--569, 2020.

{ 
\bibitem{Zhou2coef}
Z~Zhou.
On the recovery of two function-valued coefficients in the {H}elmholtz equation for inverse scattering problems via neural networks,
{\it Adv. Comput. Math. \bf51}, no.12, 2025.
}



\end{thebibliography}

\end{document}